\documentclass[10pt]{elsarticle}

\usepackage[utf8]{inputenc}

\usepackage{caption}
\usepackage{graphicx}
\usepackage{hyperref}
\usepackage{xcolor}
\usepackage{amsmath,amsthm} 
\usepackage{amssymb,mathrsfs} 
\usepackage{a4wide} 
\usepackage{physics}
\usepackage{color,subfigure} 
\usepackage{enumerate}
\usepackage[normalem]{ulem}
\usepackage{cancel}
\usepackage{bbm}
\usepackage{tikz}
\usepackage{nccmath}
\usepackage[french,english]{babel}
\usepackage{mathtools}
\usepackage{relsize}

\usetikzlibrary{shapes.misc}
\usetikzlibrary{calc}

\mathtoolsset{showonlyrefs=true}

\hypersetup{
  colorlinks   = true,
  linkcolor    = blue!70!black,
  urlcolor     = green!60!black,
  citecolor    = red!90!black,
  filecolor    = blue!40!black
}

\newcommand{\cL}{\mathcal{L}}
\renewcommand{\d}{\mathrm{d}}
\newcommand{\e}{\mathrm{e}}
\newcommand{\E}{\mathbb{E}}
\renewcommand{\P}{\mathbb{P}}

\newcommand{\R}{\mathbb{R}}
\newcommand{\n}{\mathrm{n}}

\newcommand{\Hess}{\mathrm{Hess}\,}
\newcommand{\N}{\mathbb N}
\newcommand{\1}{\mathbbm 1}
\newcommand{\supp}{\mathrm{supp}}
\newcommand{\largeRadius}{\delta}
\newcommand{\smallRadius}{\gamma}
\newcommand{\Ind}{\mathrm{Ind}}

\newcommand{\Aseti}{\mathbf A_\beta^{(i)}}
\newcommand{\Bseti}[1]{\mathbf B_\beta^{(i),#1}}
\newcommand{\Cseti}{\mathbf C_\beta^{(i)}}
\newcommand{\Dseti}[1]{\mathbf D_\beta^{(i),#1}}
\newcommand{\Eset}{\mathbf E_\beta}
\newcommand{\Fset}{\mathbf F_\beta}
\newcommand{\Gset}[1]{\mathbf G_\beta^{#1}}
\newcommand{\Vstar}{V^\star}
\newcommand{\basin}[1]{\mathcal A(#1)}

\newcommand{\testfuncs}{\mathcal C^\infty_{\mathrm c}}

\newcommand{\epsLimit}[1]{\alpha^{(#1)}} 
\newcommand{\localNeighborhood}[2][]{\mathcal{O}_{#2}^{#1}} 

\newcommand{\hessPassage}[2][]{U^{(#2)#1}} 
\newcommand{\hessEigvec}[2]{v^{(#1)}_{#2}} 
\newcommand{\hessEigval}[2]{\nu^{(#1)}_{#2}} 
\newcommand{\halfSpace}[1]{E^{(#1)}}

\newcommand{\localCutoff}[1]{\chi_\beta^{(#1)}} 

\newcommand{\deltaScalingExp}{s}
\newcommand{\gammaScalingExp}{t}
\newcommand{\Id}{\mathrm{Id}}
\newcommand{\shift}{h} 

\newcommand{\Ki}[1]{K^{(i)}_{#1}}
\newcommand{\psii}[2]{\psi^{(i)}_{#1,#2}}
\newcommand{\wi}[2]{w^{(i)}_{#1,#2}}
\newcommand{\lambdai}[2]{\lambda^{(i)}_{#1,#2}}
\newcommand{\omegai}[2]{\omega^{(i)}_{#1,#2}}

\newcommand{\widesim}[2][1.5]{
  \mathrel{\overset{#2}{\scalebox{#1}[1]{$\sim$}}}
}

\newcommand{\gaussianCutoffConst}{C_\xi}
\newcommand{\energyCutoffConst}{C_\eta}
\newcommand{\Lmu}{L_\beta^2}

\renewcommand{\O}{\mathcal{O}}
\newcommand{\smallo}{\scalebox{0.7}{$\mathcal O$}}

\newcounter{thmlike}
\newtheorem{lemma}[thmlike]{Lemma}

\newtheorem{theorem}[thmlike]{Theorem}
\newtheorem{corollary}[thmlike]{Corollary}
\newtheorem{proposition}[thmlike]{Proposition}

\newtheorem{hypothesis}{Assumption}
\newtheorem{remark}{Remark}

\tikzset{
    cross/.pic = {
    \draw[rotate = 45] (-#1,0) -- (#1,0);
    \draw[rotate = 45] (0,-#1) -- (0, #1);
    }
}

\tikzset{cross/.style={cross out, draw=black, minimum size=2*(#1-\pgflinewidth), inner sep=0pt, outer sep=0pt},
cross/.default={1pt}}

\begin{document}
\begin{frontmatter}

\title{Quantitative low-temperature spectral asymptotics for reversible diffusions in temperature-dependent domains}
 \author[a,b]{No\'e Blassel\corref{cor}}
  \ead{noe.blassel@epfl.ch}
  \author[a,b]{Tony Leli\`evre}
  \author[a,b]{Gabriel Stoltz}

  \cortext[cor]{Corresponding author}

  \affiliation[a]{organization={CERMICS, \'Ecole nationale des ponts et chauss\'ees, Institut Polytechnique de Paris},
                  city={Marne-la-Vall\'ee},
                  country={France}}
  \affiliation[b]{organization={MATHERIALS project-team, Inria Paris},
                  city={Paris},
                  country={France}}

\begin{abstract}
We derive novel low-temperature asymptotics for the spectrum of the infinitesimal generator of the overdamped Langevin dynamics. The novelty is that this operator is endowed with homogeneous Dirichlet conditions at the boundary of a domain which depends on the temperature. From the point of view of stochastic processes, this gives information on the long-time behavior of the diffusion conditioned on non-absorption at the boundary, in the so-called quasistationary regime. Our results provide precise estimates of the spectral gap and principal eigenvalue, extending the Eyring--Kramers formula. The phenomenology is richer than in the case of a fixed boundary and gives new insight into the sensitivity of the spectrum with respect to the shape of the domain near critical points of the energy function. Our work is motivated by the optimization of accelerated molecular dynamics algorithms (such as Parallel Replica dynamics), where optimal state definitions naturally vary with the temperature.

 \medskip
  \noindent\textbf{R\'esum\'e}\par
  \begin{otherlanguage}{french}
\noindent
  Nous obtenons de nouvelles asymptotiques spectrales pour le générateur infinitésimal de la dynamique de Langevin suramortie dans le régime de basse température. La nouveauté est que cet opérateur est muni de conditions de Dirichlet homogènes au bord d'un domaine dépendant de la température.
  Du point de vue des processus stochastiques, ceci décrit le comportement en temps long de la dynamique conditionnée à ne pas être absorbée au bord, dans le régime dit quasi-stationnaire. Nos résultats fournissent des estimées précises du trou spectral et de la valeur propre principale, qui généralisent la formule d'Eyring--Kramers.
  La phénoménologie étudiée est plus riche que dans le cas d'un bord fixe, et apporte un éclairage sur la sensibilité du spectre à la forme du domaine au voisinage des points critiques du potentiel. Ce travail est motivé par l'optimisation d'algorithmes de dynamique moléculaire accélérée (tels que la méthode Parallel Replica), pour lesquels les définitions optimales des états métastables dépendent naturellement de la température.
  \end{otherlanguage}
\end{abstract}
\end{frontmatter}


    \section{Introduction}
    \label{sec:intro}
    We study characteristic timescales of the diffusion process defined as the strong solution~$(X^{\beta}_t)_{t\geq 0}$ to the stochastic differential equation
    \begin{equation}
        \label{eq:overdamped_langevin}
        \d X^{\beta}_{t} = -\nabla V(X^{\beta}_{t})\,\d t + \sqrt{\frac2\beta}\, \d W_t,
    \end{equation}
    where~$(W_t)_{t\geq 0}$ is a standard Brownian motion on~$\R^d$,~$V:\R^d \to \R$ is a smooth function, and~$\beta>0$ is a parameter modulating the magnitude of the noise.
    In the context of atomistic simulation, the process~\eqref{eq:overdamped_langevin} is called the overdamped Langevin dynamics, and is commonly used to model the motion of particles subject to an interaction potential~$V$ at thermal equilibrium with inverse temperature~$\beta=(kT)^{-1}$.
    
    More generally, the function~$V$ is, up to an additive constant, the log-likelihood of the Gibbs measure
    \begin{equation}
        \label{eq:gibbs_measure}\mu(\d x) = \mathcal Z_\beta^{-1}\e^{-\beta V(x)}\,\d x,
    \end{equation}
    which is a probability measure with respect to which the dynamics~\eqref{eq:overdamped_langevin} is known to be reversible, and, under a general set of assumptions, ergodic.
    We refer the reader to~\cite{RB06} for sufficient conditions ensuring the well-posedness and ergodicity of~\eqref{eq:overdamped_langevin} with respect to~$\mu$.

    As it allows sampling from probability measures whose densities are explicit up to the normalization constant~$\mathcal Z_\beta$, the dynamics~\eqref{eq:overdamped_langevin} is also used in Bayesian statistics to sample from the posterior distribution. In theoretical machine learning, the process~\eqref{eq:overdamped_langevin} can also be seen as an idealized model for the stochastic gradient descent algorithm, after an appropriate normalization of the data, in which case $V$ plays the role of the loss function.

    \paragraph{The local approach to metastability.}
    In many cases, the trajectories of~\eqref{eq:overdamped_langevin} are subject to the phenomenon of metastability, which is indicated by the presence of a wide range of well-separated timescales, often exponentially wide in the inverse temperature. This corresponds to the regime in which the Arrhenius law (see~\cite{A67}) applies.
    Longer timescales correspond to rare transitions between attractive regions of the configuration space $\R^d$, which trap the dynamics into long-lived local ensembles of configurations, which we refer to as metastable states. The shorter timescales correspond to thermal fluctuations inside these states.
    
    The nature of the trapping mechanism itself may vary. It may be that energetic barriers tend to confine the dynamics inside a potential well for long times, which is the case on which we focus in this work. It may also happen that different subdomains are joined by low-energy paths, but in narrow configurational corridors, which require well-coordinated collective motion of the system's degrees of freedom to successfully navigate.
    In this case, the obstacle is of an entropic nature, and the dynamics is in fact confined in a free-energetic rather than purely energetic well. 
    
    Moreover, in molecular dynamics (MD) simulations, monitoring the long-time behavior of the dynamics~\eqref{eq:overdamped_langevin} is of crucial importance to reliably estimate macroscopic dynamical properties of materials and biomolecules, as well as parametrizing models on larger scales, such as discrete Markov models or PDEs.
    
    Drawing meaningful long trajectories from metastable dynamics is however challenging with naive techniques. To alleviate this, so-called accelerated dynamics methods~\cite{V97,V98,SV00} have been proposed by Arthur Voter in the late 1990s, all of which rely on a local approach to metastability.

    In this local approach, the notion of metastable state can be formalized using the quasi-stationary distribution (QSD), which, given an arbitrary bounded subset of configuration space~$\Omega\subset\R^d$, can be loosely understood as the long-time limiting distribution of the process conditioned on staying trapped inside~$\Omega$.
    Defining the so-called Yaglom limit:
    $$ \nu = \underset{t\to\infty}{\lim}\,\mu_t,\qquad \mu_t := \mathrm{Law}\left(X_t^\beta \middle|\forall\, 0\leq s\leq t,\,X_s^\beta \in \Omega\right),$$
    it can be shown under mild assumptions~(see~\cite{MV12}) that the limit $\nu$ is well-defined, and that it is the unique QSD for the process~\eqref{eq:overdamped_langevin} in~$\Omega$, with moreover~$\mu_t$  converging exponentially fast to~$\nu$ in total variation norm. Local metastability inside~$\Omega$ can then be understood as a large separation between two natural timescales related to the QSD. The first timescale is the average exit time from~$\Omega$ for the dynamics~$X_t^\beta$ with initial distribution~$\nu$:
    $$ \tau_1(\Omega) = \E^\nu\left[\inf\{t>0:\,X_t^\beta \not\in \Omega\}\right].$$
    The second timescale~$\tau_2(\Omega)$ is that at which~$\mu_t$ reaches the QSD and the process~$X_t^\beta$ thus forgets its initial distribution conditioned on not exiting~$\Omega$.
    If~$\tau_1(\Omega) \gg \tau_2(\Omega)$, then the domain $\Omega$ acts as a metastable trap for the dynamics~$X_t^\beta$. Conditioned on~$\{X_s^\beta\in\Omega,\,\forall\,0\leq s\leq t\}$, for~$\tau_2(\Omega)\ll t\ll \tau_1(\Omega)$, the state of the dynamics at time~$t$ is approximately distributed according to the local equilibrium~$\nu$, and remains so until the dynamics exits once again.

    It is possible to show (see~\cite{LBLLP12} and Propositions~\ref{prop:qsd_spectral} and~\ref{prop:decorr} below) that the two local timescales associated with the metastable behavior of the process~$X^\beta_t$ inside~$\Omega$ can be related to the spectrum of the infinitesimal generator~$\cL_\beta$ of the dynamics~\eqref{eq:overdamped_langevin}, supplemented with Dirichlet boundary conditions on~$\partial \Omega$.
    Namely, writing~$\lambda_{k,\beta}(\Omega)$ for the~$k$-th smallest Dirichlet eigenvalue of~$-\cL_\beta$, the following holds:
    \begin{itemize}
        \item{The metastable exit rate is given by~$\lambda_{1,\beta}(\Omega)$.}
        \item{The asymptotic convergence rate to the QSD is given by~$\lambda_{2,\beta}(\Omega)-\lambda_{1,\beta}(\Omega)$.}
    \end{itemize}
    Strictly speaking, Proposition~\ref{prop:decorr} only provides an upper bound on the convergence rate, which we expect in practice to depend on the initial condition. Nevertheless, these spectral characterizations provide us with a natural and tractable measure of the local metastability associated with~$\Omega$, namely the separation of timescales
    \begin{equation}
        \label{eq:separation}
        J(\Omega) := \frac{\lambda_{2,\beta}(\Omega)-\lambda_{1,\beta}(\Omega)}{\lambda_{1,\beta}(\Omega)}.
    \end{equation}
    If~$J(\Omega)\gg 1$, i.e. $\lambda_{2,\beta}(\Omega)-\lambda_{1,\beta}(\Omega)\gg \lambda_{1,\beta}(\Omega)$, then~$\Omega$ acts as a highly locally metastable trap for the dynamics~\eqref{eq:overdamped_langevin}.
    The link between characteristic timescales for conformational dynamics and the spectrum of various operators has been widely discussed in the literature, see~\cite{HS06,NN13,HKN04} for an overview of various approaches.
    
    \paragraph{Shape optimization for the timescale separation.}
    A natural question arises: how to choose the shape of the domain~$\Omega$ in order to maximize~$J(\Omega)$? In other words: how to choose~$\Omega$ in order for the process~$(X_t^\beta)_{t\geq 0}$ to be as locally metastable as possible?
    This line of investigation is natural from the point of view of spectral geometry, as it generalizes the problem of finding extremal shapes for ratios of eigenvalues of the Dirichlet Laplacian (which corresponds to setting~$V=0$ in our context), see for example the Payne--P\'olya--Weinberger conjecture~\cite{PPW56} and its solution by Ashbaugh and Benguria in~\cite{AB92}.
    
    Moreover, the answer to this question has concrete implications for accelerated molecular dynamics simulations. The timescale separation has been directly related to the efficiency of the parallel replica algorithm (ParRep,~\cite{V98}) (see~\cite{LBLLP12,SL13,PUV15,BLS25b}).
    More generally, many algorithms from molecular dynamics explicitly rely on the definition of metastable states and a separation of timescales assumption. This is the case for the other two methods of Arthur Voter, namely Hyperdynamics~\cite{V97} and Temperature Accelerated dynamics~\cite{SV00}, and other algorithms approximating the dynamics~\eqref{eq:overdamped_langevin} by Markov jump processes, such as in kinetic Monte Carlo (kMC)~\cite{V07} or Markov state models~\cite{HP18}, which also rely on a local separation of timescales assumption to justify the Markovianity of the discrete-space dynamics.
    More recent works~\cite{AGHVP20,AJP23} approximating the dynamics by Markov renewal processes also rely on the definition of metastable states, and identify the timescale ratio~\eqref{eq:separation} as the key parameter governing the quality of this approximation.

    It is therefore natural to seek definitions of metastable states which maximize this timescale separation. In low-dimensional settings, or when low-dimensional representations of the dynamics are available, it is possible to locally optimize the shape of an initial domain~$\Omega$, using shape perturbation results for eigenvalues. This approach is developed and demonstrated in~\cite{BLS25b}.
    However, the explicit computation of these shape perturbations is generally intractable for systems of practical interest, as it requires the solution to high-dimensional boundary value problems, besides the intrinsic difficulty of parametrizing high-dimensional shapes. To circumvent these difficulties, an alternative approach relies on choosing a limiting regime, and finding asymptotically optimal shapes within a low-dimensional parametric family of shapes.
    In this work, we perform the mathematical analysis necessary to realize this strategy in the low-temperature regime~$\beta\to\infty$, for a class of parametrizations which allow for explicit computations.

    The intuitive fact that optimal domain definitions depend on the temperature is known to practitioners of large-scale accelerated molecular dynamics simulations, see for instance~\cite[Figure 6]{PUV15}.
    This observation motivates the study of eigenvalue asymptotics for temperature-dependent domains, in the tractable low-temperature regime.
    This is the framework in which we work, and which we make precise in Section~\ref{subsec:geometric_assumptions}.

    Numerical experiments (see for instance~\cite[Section 5.2, in particular Figure 9]{BLS25b}) suggest that, in the limit~$\beta\to\infty$, the low-lying eigenvalues are rather insensitive to the geometry of the domain, except when the boundary of the domain nears specific critical points of $V$, where the spectrum experiences sharp transitions.

    \paragraph{Low-temperature spectral asymptotics for metastable stochastic processes.}
    In the low-temperature limit, methods from semi-classical analysis, potential theory and large deviations have been successfully leveraged in previous works to address the problem of finding quantitative spectral estimates for the dynamics~\eqref{eq:overdamped_langevin} as well as the associated Dirichlet eigenvalue problem on a fixed domain~$\Omega$, see for instance~\cite{HKN04,BGK05,HN06,LN15,LLPN19,DGLLPN19,LPN21,LLPN22}.
    In particular, many efforts have been dedicated to rigorously derive precise asymptotics for the principal eigenvalue:
    \begin{equation}
        \label{eq:eyring_kramers_classic}
        \lambda_{1,\beta}(\Omega) \widesim{\beta\to\infty} C(\Omega,\beta)\,\e^{-\beta E(\Omega)},
    \end{equation}
    where~$E(\Omega)$ is analogous to the activation energy in the Arrhenius law, and~$C(\Omega,\beta)$ is a pre-exponential factor whose expression generally depends on both the temperature and the domain, but which can, under various sets of assumptions, be computed, at least to first order in~$\beta$.
    Such results are known as Eyring--Kramers formul\ae, following first proposals concerning the behavior of reaction rates~\cite{E35} guided by chemical intuition, and early computations~\cite{K40} supported by physical modeling. Eyring--Kramers type results have, since the early 2000s, been rigorously proven and generalized in the mathematical community, using tools from dynamical systems, quantum theory, and stochastic processes. We refer to~\cite{B13} for an overview.

    The link with the classical Eyring--Kramers formula, which concerns the closely related problem of computing asymptotics for the exit rate~$\E_x[\tau_\Omega]^{-1}$ for some deterministic initial condition~$x\in\Omega$, requires some discussion. It is justified by the interpretation of~$1/\lambda_{1,\beta}(\Omega)$ as a metastable exit time, i.e.~$1/\lambda_{1,\beta}(\Omega) = \int_{\Omega} \E_x[\tau_\Omega]\nu(\d x)$ (see Proposition~\ref{prop:qsd_spectral} below). The connection can be made fully rigorous by exploiting exponential leveling results, see for instance~\cite{N20}.
    In the context of deterministic initial conditions, let us mention that tools from the theory of Freidlin and Wentzell (see~\cite{FW12}) relying on large deviations are capable of identifying the activation energy~$E(\Omega)$ under very general conditions, including non-reversible diffusions, and these techniques have also been recently extended to a class of non-Markovian processes (so-called self-interacting diffusions, see~\cite{ADMKT24}).

    Let us stress that, while our results are limited to the setting of the overdamped Langevin dynamics~\eqref{eq:overdamped_langevin}, various extensions of our results to other dynamics would be of practical interest and are left to future work. For instance, the underdamped Langevin dynamics (see~\cite{HHS08a,HHS08b,HHS11,BLPM24} for the case without boundary), or non-reversible diffusions with non-degenerate diffusion coefficients, as in~\cite{BR16,LMS19,LPM20,LS22,LPMN25}, are both interesting avenues for future research.

    Previous results from the semi-classical literature motivate the setting of temperature-dependent domains from a theoretical perspective. In the case of fixed domains, these results show that different asymptotic regimes hold depending on whether certain critical points of $V$ (typically, specific saddle points of the potential) lie within the domain, on the boundary of the domain, or outside the domain. See for example~\cite{LPN21} where these three regimes are treated. The leading-order asymptotic behavior of~$\lambda_{1,\beta}(\Omega)$, in particular, is found to be a discontinuous shape-functional of~$\Omega$, with discontinuities occurring when~$\partial\Omega$ crosses these critical points.
    On the other hand, shape-perturbation theory (see for instance~\cite[Theorem 1]{BLS25b}) implies that~$\lambda_{1,\beta}(\Omega)$ itself is a continuous shape-functional of~$\Omega$ at any fixed value of~$\beta>0$.
    These two facts imply that, when~$\beta\gg 1$, $\lambda_{1,\beta}(\Omega)$ is extremely sensitive to~$\Omega$, and in particular to the position of~$\partial\Omega$ relative to specific critical points of~$V$ on some microscopic scale. As it turns out, a similar sensitivity affects every eigenvalue~$\lambda_{k,\beta}(\Omega)$, and the relevant microscopic scale is $1/\sqrt\beta$.
    The framework of temperature-dependent domains is a natural way to interpolate between the three asymptotic regimes, and to study more precisely these steep spectral transitions.

    Finally, let us notice that, because of the fact that energetic barriers contribute dominantly to the slowest timescales in the small-noise limit~$\beta\to\infty$, the literature has focused overwhelmingly on the case where metastability occurs in the presence of energetic barriers, as does the present work.
    However, entropic barriers are often relevant to applications, most notably in the context of cellular biology, and recent works have started to tackle rigorously spectral asymptotics for other metastable systems, in particular the so-called narrow escape problem, where the asymptotic parameter is not the temperature but instead the size of a vanishing exit region for a Brownian motion in a confining spatial domain (which corresponds to a purely entropic trapping mechanism).
    For rigorous results in this direction, we cite~\cite{FNO21} where coarse asymptotics of the spectrum of the Laplacian with mixed Dirichlet--Neumann boundary conditions are given in the limit of a vanishing absorbing region, as well as~\cite{AKKL12,LRS24} where finer asymptotics of low-lying eigenvalues and normal derivatives of the associated eigenmodes are derived, in a model two-dimensional situation.

    We finally mention that the previous works~\cite{LN15,DGLLPN19} have applied semiclassical techniques to the Witten Laplacian, motivated by the numerical analysis of Arthur Voter's accelerated dynamics methods, respectively Hyperdynamics~\cite{V97} and Temperature Accelerated Dynamics~\cite{SV00}. This work in some sense continues in this vein, deriving rigorous results to analyze the third and last of the accelerated dynamics methods of Voter, namely the Parallel Replica method~\cite{V98}.

    \paragraph{Contributions and outline.}
    The purpose of this work is to extend previous results on low-temperature spectral asymptotics for the dynamics~\eqref{eq:overdamped_langevin} in the framework of temperature-dependent Dirichlet boundary conditions.
    We develop a set of geometric assumptions on these temperature-dependent boundaries which ensure that these asymptotics can be easily computed.
    We then identify the first-order behavior of the spectrum (Theorem~\ref{thm:harm_approx}), extending to~$\beta$-dependent Dirichlet boundary conditions the so-called harmonic limit from the semi-classical analysis of the Witten Laplacian.
    We generalize the Eyring--Kramers formula (Theorem~\ref{thm:eyring_kramers}) to the case of temperature-dependent boundaries and a single-well domain, and obtain an explicit expression for the prefactor.
    We give explicit asymptotic equivalents of both the metastable exit rate and the asymptotic convergence rate to the QSDs associated with these domains (Corollary~\ref{corr:practical_implication}).
    We discuss the practical implications of these results, and how to optimize the asymptotic separation of timescales with respect to the shape of the boundary, within the class of temperature-dependent domains satisfying our geometric assumptions.

    This work is organized as follows. In Section~\ref{sec:setting_and_notations}, we introduce the necessary notation and present our geometric framework. In Section~\ref{sec:main_results}, we state and discuss our main results. In Section~\ref{sec:proof_harm_approx}, we construct the harmonic approximation and prove the first spectral asymptotics. We finally prove in Section~\ref{sec:proof_ek} the modified Eyring--Kramers formula.
    \section{Setting and notation}
    In this section, we introduce various notation (see Section~\ref{subsec:notation}), define the QSD in Section~\ref{subsec:qsd}, and present in Section~\ref{subsec:geometric_assumptions} the geometric framework which will be used throughout this work.
    In Section~\ref{subsec:genericity}, we motivate these assumptions, discuss their genericity and compare them with previous related works. 
    \label{sec:setting_and_notations}
    \subsection{Notation}
    \label{subsec:notation}
    We introduce various notation and classical properties which will be used throughout this work.
    \paragraph{Temperature-dependent domains.}
    The basic setting we work in, and the main novelty compared to previous related works, is the presence of a temperature-dependent boundary. To this end, we introduce a family of domains~$\left(\Omega_\beta\right)_{\beta>0}$ which will be considered throughout this work, and which we will always assume to be smooth, open and bounded.
    In fact, we make the following stronger assumption.
    \begin{hypothesis}
        For all~$\beta>0$, the domain~$\Omega_\beta$ is a smooth open set, and there exists a compact set~$\mathcal K \subset \R^d$ such that
        \begin{equation}
            \label{hyp:uniformly_bounded}
            \tag{\bf H0}
            \forall\,\beta>0,\qquad\Omega_\beta \subset \mathcal K.
        \end{equation}
    \end{hypothesis}
    We denote by~$\sigma_{\Omega_\beta}$ the signed Euclidean distance to the boundary of~$\Omega_\beta$:
    \begin{equation}
        \label{eq:sdf}
        \sigma_{\Omega_\beta}(x) = \left\{\begin{aligned}
            d(x,\partial\Omega_\beta)&&\quad\textrm{for }x\in\Omega_\beta,\\
            -d(x,\partial\Omega_\beta)&&\quad\textrm{for }x\not\in\Omega_\beta.
        \end{aligned}\right.
    \end{equation}
    While the opposite sign is sometimes preferred in the definition of the signed distance, our choice of convention is motivated by the identities~$A\cap B = (\sigma_A \land \sigma_B)^{-1}(\R_+^*)$,~$A\cup B = (\sigma_A \lor \sigma_B)^{-1}(\R_+^*)$ for two open sets~$A,B\subset \R^d$.
    We denote the unit outward normal at a point~$x\in\partial \Omega_\beta$ by~$\n_{\Omega_\beta}(x) = -\nabla \sigma_{\Omega_\beta}(x)$.

    Various other assumptions which will enter in the statement of our results are given in Section~\ref{subsec:geometric_assumptions}.
    \paragraph{Critical points of the potential.}
    We assume throughout this work that the potential~$V$ is a smooth Morse function over~$\R^d$, meaning that at each point~$z$ such that~$\nabla V(z)=0$, the Hessian~$\nabla^2 V(z)$ is non-degenerate.
    Relaxing this (generic) assumption leads to a range of possible behaviors for the dynamics~\eqref{eq:overdamped_langevin}, see~\cite{BG10,LN15,LPNV24,D24,ABM25} for related quantitative studies of metastability in such settings. These effects are however not the primary concern of this work.

    This condition implies in particular that critical points of~$V$ are isolated, and are therefore finitely many inside~$\mathcal K$.
    We also assume that~$V$ is globally bounded from below.
    We recall that the index of a critical point~$z\in\R^d$ for~$V$ is the number of negative eigenvalues of~$\nabla^2 V(z)$:
    \[\Ind(z) = \#\, \left[\mathrm{Spec}(\nabla^2 V(z)) \cap (-\infty,0)\right].\]
    For~$0\leq k \leq d$, denote by~$N_k$ the number of critical points of index-$k$ located in~$\mathcal K$:
    \begin{equation}
        \label{eq:morse_indices}
        N_k = \#\,\left\{z \in \mathcal K\,\middle|\,\nabla V(z)=0,\,\Ind(z)=k\right\},\qquad N:=\sum_{k=0}^d N_k.
    \end{equation}
    Thus,~$V$ has~$N$ critical points in~$\mathcal K$, and, for instance,~$N_0$ local minima.
    Finally, we fix an enumeration~$(z_i)_{0\leq i< N}$
    of the critical points of~$V$ in~$\mathcal K$, chosen so that (with the convention~$N_0+\dots+N_{k-1}=0$ for~$k=0$):
    \[\left\{z_i\,\middle|\,N_0+\dots+N_{k-1}\leq i <N_0+\dots+N_k\right\} = \left\{z \in \mathcal K\,\middle|\,\nabla V(z)=0,\,\Ind(z)=k\right\}.\]
    Such a finite enumeration is guaranteed to exist, owing to the Morse property of~$V$ and the compactness of~$\mathcal K$.
    In particular, the local minima of~$V$ in~$\mathcal K$ are enumerated as~$z_0,\dots,z_{N_0-1}$.

    We also fix an eigendecomposition of the Hessian~$\nabla^2 V$ at each critical point~$z_i$, letting
    \begin{equation}
        \label{eq:eigvals_hessian}
        \mathrm{Spec}(\nabla^2 V(z_i)) = \left\{\hessEigval{i}{1}, \hessEigval{i}{2}, \dotsm, \hessEigval{i}{d}\right\}
    \end{equation}
    denote the spectrum of~$\nabla^2 V$ at~$z_i$, with an associated orthonormal eigenbasis:
    \begin{equation}
        \label{eq:eigvecs_hessian}
        \hessPassage{i} =\begin{pmatrix}\hessEigvec{i}{1}&\dotsm&\hessEigvec{i}{d}\end{pmatrix}\in\R^{d\times d},\quad \mathrm{diag}(\hessEigval{i}{1},\dotsm,\hessEigval{i}{d}) = \hessPassage[\intercal]{i} \nabla^2 V(z_i) \hessPassage{i}.
    \end{equation}
    Since~$V$ has the Morse property,~$\nu_j^{(i)} \neq 0$ for all~${1\leq j\leq d}$ and~${0\leq i < N}$.
    Let us also assume that, in the case where~$\Ind(z_i)=1$,~$\hessEigval{i}{1}<0$ is the unique negative eigenvalue of~$\nabla^2 V(z_i)$. The orientation convention for the first eigenvector~$\hessEigvec{i}{1}$ is either fixed by the geometry of the domains~$\Omega_\beta$ (as made precise in Assumption~\eqref{hyp:locally_flat} below), or else plays no role in the analysis.
    
    The eigenrotation induces a unitary transformation in $L^2$, via:
   \begin{equation}
    \label{eq:eigvecs_unitary}
     \left(\mathcal U^{(i)} f\right)(x) = f\left(U^{(i)\intercal}x\right),\quad  \left(\mathcal U^{(i)*} f\right)(x) = f\left( U^{(i)}x\right).
   \end{equation}

    We will make repeated use of the following half-spaces associated with each critical point, defined for $\theta\in(-\infty,+\infty]$ by:
    \begin{equation}
        \label{eq:half_space}
        \halfSpace{i}(\theta) = \hessPassage{i}\left[(-\infty,\theta)\times\R^{d-1}\right].
    \end{equation}
    Of course, one simply has~$\halfSpace{i}(+\infty) = \R^d$.

    For notational convenience, we also introduce the following local coordinates, adapted to the quadratic behavior of~$V$ near $z_i$:
    \begin{equation}
        \label{eq:local_coordinates}
        y^{(i)}(x) = \hessPassage[\intercal]{i}(x-z_i) = \left(y_1^{(i)}(x),y^{(i)\prime}(x)\right) \in \R \times \R^{d-1}.
    \end{equation}
    Note that~$y^{(i)}$ is an affine isometry, with~$\nabla y_j^{(i)} = \hessEigvec{i}{j}$ and~$\nabla^2 y^{(i)}=0$.

    \paragraph{Gradient flow and invariant manifolds.}
    We denote by~$(\phi_t)_{t\in \R}$ the flow associated with the dynamics~$X'(t) = -\nabla V(X(t))$, which can be seen as the formal limit of the process~\eqref{eq:overdamped_langevin} as~$\beta\to+\infty$. For~$0\leq i< N$, we introduce the following sets:
    \begin{equation}
        \label{eq:stable_manifold}
        \mathcal W^{\pm}(z_i) = \left\{x\in\R^d\middle|\underset{t\to\pm\infty}{\lim}\,\phi_t(x)=z_i\right\}.
    \end{equation}
    The stable manifold theorem (see~\cite[Section 9.2]{T24}) asserts that these sets are smooth submanifolds with boundary. The manifold~$\mathcal W^+(z_i)$ is called the stable manifold and~$\mathcal W^-(z_i)$ is called the unstable manifold. Moreover their tangent spaces at~$z_i$ are given by
    \[T_{z_i}(\mathcal W^{\pm}) = \mathrm{Vect}\,\left\{\hessEigvec{i}{j},\,1\leq j\leq d:\,\pm\hessEigval{i}{j} > 0 \right\},\]
    so that~$\mathcal W^+(z_i)$ and~$\mathcal W^-(z_i)$ are of dimensions~$d-\mathrm{Ind}(z_i)$ and~$\mathrm{Ind}(z_i)$, respectively.

    For~$0\leq i < N_0$, we alternatively denote
    \begin{equation}
        \label{eq:basin}
        \basin{z_i} := \mathcal W^+(z_i),
    \end{equation}
    which we refer to as the basin of attraction for the critical point~$z_i$.
    \paragraph{Weighted Sobolev spaces.}
    Let us introduce the following Hilbert spaces, defined for an open set~$\Omega_\beta\subset \R^d$ by
    \begin{equation}
        \label{eq:sobolev_spaces}
        \begin{aligned}
        &\Lmu(\Omega_\beta) = \left\{u\text{ measurable}\,\middle|\,\|u\|^2_{\Lmu(\Omega_\beta)} :=\int_{\Omega_\beta} u^2 \,\e^{-\beta V} < +\infty\right\},\\
         &H^{k}_\beta(\Omega_\beta) = \left\{u\in \Lmu(\Omega_\beta)\middle|\,\partial^{\alpha}u\in \Lmu(\Omega_\beta),\,\forall\, |\alpha|\leq k\right\},
        \end{aligned}
    \end{equation}
    where~$\partial^\alpha = \partial_{x_1}^{\alpha_1}\dots\partial_{x_d}^{\alpha_d}$ denotes the weak differentiation operator associated with a multi-index~$\alpha = (\alpha_1,\dots,\alpha_d)\in \N^d$. For the flat case (i.e. when~$V\equiv 0$), we simply write~$L^2(\Omega_\beta)$ and~$H^k(\Omega_\beta)$. As in the flat case, we let~$H_{0,\beta}^k(\Omega_\beta)$ denote the~$H_\beta^k(\Omega_\beta)$-norm closure of~$\testfuncs(\Omega_\beta)$. 

    \paragraph{Dirichlet realizations of the generator.}
    The infinitesimal generator~$\cL_\beta$ of the evolution semigroup associated with the dynamics~\eqref{eq:overdamped_langevin} is formally defined by:
    \begin{equation}
        \label{eq:generator}
        \forall u \in \testfuncs(\R^d),\qquad-\cL_\beta u = - \frac1\beta \Delta u+\nabla V \cdot \nabla u .
    \end{equation}
    A direct computation shows that~$-\cL_\beta$ is a symmetric positive operator on~$\Lmu(\R^d)$, with associated quadratic form:
    \begin{equation}
        \label{eq:quadratic_form_generator}
        \left\langle -\cL_\beta u, u \right\rangle_{\Lmu(\R^d)} = \frac1\beta \int_{\R^d} |\nabla u|^2\,\e^{-\beta V}.
    \end{equation}
    This identity is the analytic formulation of the reversibility of the dynamics~\eqref{eq:overdamped_langevin} with respect to~$\mu$.

    For bounded open domains~$\Omega_\beta\subset\R^d$ with Lipschitz boundary, we will still denote by~$\cL_\beta$ the Dirichlet realization of the generator, defined as the Friedrichs extension (see~\cite{RS75}) of the quadratic form~\eqref{eq:quadratic_form_generator} on~$\testfuncs(\Omega_\beta)$.
    This is a self-adjoint operator, with domain~$$\mathcal D(\cL_\beta)={H_{0,\beta}^1(\Omega_\beta)\cap H^2_{\beta}(\Omega_\beta) \subset \Lmu(\Omega_\beta)}.$$
    We will also consider the Dirichlet realization of the differential operator~$\cL_\beta$ on other domains than~$\Omega_\beta$; this will be made precise when needed.
    
    The compact embedding~$H_{0,\beta}^1(\Omega_\beta)\to \Lmu(\Omega_\beta)$ (which follows immediately from the Rellich--Kondrachov theorem since~$V$ is smooth and~$\Omega_\beta$ is smooth and bounded) implies that~$-\cL_\beta$ seen as an operator on~$L^2_\beta(\Omega_\beta)$ has a compact resolvent, so that its spectrum consists of non-negative, isolated eigenvalues of finite multiplicity tending to~$+\infty$:
   ~$$0\leq\lambda_{1,\beta} \leq \lambda_{2,\beta}\leq \dotsm \leq\lambda_{k,\beta} \xrightarrow{k\to\infty}+\infty,$$
   where for~$k\geq 1$,~$\lambda_{k,\beta}$ denotes the $k$-th smallest eigenvalue counted with multiplicity.
    Furthermore, standard arguments~(see~\cite{LBLLP12}) show that the first eigenvalue is strictly positive and non-degenerate, so that~$0<\lambda_{1,\beta} < \lambda_{2,\beta}$, with an associated eigenmode~$u_{1,\beta} > 0$ on~$\Omega_\beta$.
    
    We also consider eigenmodes~$u_{k,\beta}$ associated with~$\lambda_{k,\beta}$ with, for any integers~$i,j\geq 1$, the normalization convention:
    \begin{equation}
        \label{eq:eigfunc_normalization}
        \int_{\Omega_\beta}u_{i,\beta}u_{j,\beta} \e^{-\beta V} = \delta_{ij}.
    \end{equation}

    When considering the Dirichlet realization of~$\cL_\beta$ on a domain~$A\subset\R^d$ other than~$\Omega_\beta$, we will explicitly denote the dependence of the eigenelements on~$A$ as follows: for all~$k,\ell\geq 1$,
    \begin{equation}
        \label{eq:eigenpairs_alt_domain}
        -\cL_\beta u_{k,\beta}(A) = \lambda_{k,\beta}(A)u_{k,\beta}(A),\qquad \int_A u_{k,\beta}(A)u_{\ell,\beta}(A)\,\e^{-\beta V}=\delta_{k\ell}.
    \end{equation}
    \paragraph{Witten Laplacian.}
    It is sometimes convenient to change representations to a flat~$L^2$ setting, using the unitary map $u\mapsto \e^{-\frac{\beta V}2}u$ from~$\Lmu(\Omega_\beta)$ to~$L^2(\Omega_\beta)$. The conjugated operator associated with~$-\cL_\beta$ is proportional to the celebrated Witten Laplacian~\cite{W82} acting on 0-forms (or functions):
    \begin{equation}
        \label{eq:witten_laplacian}
        H_{\beta} = -\beta \e^{-\frac{\beta V}2} \cL_\beta \e^{\frac{\beta V}2}\cdot = - \Delta + U_\beta ,\qquad U_\beta = \frac{\beta^2}4|\nabla V|^2 - \frac{\beta}2 \Delta V,
    \end{equation}
    with domain~$\mathcal D(H_\beta)=H_0^1(\Omega_\beta)\cap H^2(\Omega_\beta) \subset L^2(\Omega_\beta)$. The operator~$\beta^{-2}H_\beta$ is a Schr\"odinger Hamiltonian operator (with the semiclassical parameter~$h$ proportional to~$1/\beta$) whose potential is perturbed by an order-$h$ term.
    For convenience, we give the correspondence with the notation commonly used in the semiclassical literature for the Witten Laplacian, following~\cite{HKN04}:
    \begin{equation}
        \beta^{-2}H_\beta = \Delta_{f,h}^{(0)} := - h^2 \Delta+|\nabla f|^2 -h\Delta f  ,\quad\text{with } h := 1/\beta\text{ and } f := V/2.
    \end{equation}
    We finally denote by
    \begin{equation}
        \label{eq:witten_quad_form}
        Q_\beta(u) = \left\langle H_\beta u , u\right\rangle_{L^2(\Omega_\beta)} = \|\nabla u\|_{L^2(\Omega_\beta)}^2+\left\langle U_\beta u ,u \right\rangle_{L^2(\Omega_\beta)},
    \end{equation}
    the quadratic form associated with~$H_\beta$, with form domain~$\mathcal D(Q_\beta) = H_0^1(\Omega_\beta)$.
    \subsection{Quasistationary distributions and the Dirichlet spectrum}
    \label{subsec:qsd}
    Given a subdomain~$\Omega_\beta\subset \R^d$, define the first exit time out of~$\Omega_\beta$ by
    \[\tau_{\Omega_\beta} = \inf\left\{t\in \R_+\,\middle| X^\beta_t\not\in \Omega_\beta\right\}.\]
    A QSD in~$\Omega_\beta$ is a probability measure~$\nu$ satisfying:
    \begin{equation}
        \label{eq:qsd}
        \forall\,t>0,\,\forall\,A \in \mathcal B\left(\Omega_\beta\right),\qquad \P^\nu\left(X^\beta_t\in A\,\middle|\,\tau_{\Omega_\beta}>t\right) = \nu(A).
    \end{equation}
    In other words, the QSD is stationary for the process~$X^\beta_t$ conditioned on not being absorbed at the boundary of~$\Omega_\beta$ up to the time $t$.
    Under generic assumptions,~$\nu$ may also be defined by the Yaglom limit:
    \begin{equation}
        \label{eq:yaglom}
        \forall x_0\in\Omega_\beta,\qquad\nu(A) = \underset{t\to\infty}{\lim}\,\P^{x_0}\left(X^\beta_t\in A\,\middle|\,\tau_{\Omega_\beta}>t\right).
    \end{equation}
    Quasi-stationary distributions were initially introduced to study long-lived phases of populations in the pre-extinction time regime.
    We refer to~\cite{MV12} for a review and examples on this topic, as well as~\cite{CMSM13} for general background material on QSDs.

    It happens that the QSD inside a smooth bounded domain~$\Omega_\beta$ is directly related to the spectrum of the Dirichlet realization of the infinitesimal generator~\eqref{eq:generator}.
    More precisely, the following result, adapted from~\cite{LBLLP12}, holds.
    \begin{proposition}[\cite{LBLLP12}]
        \label{prop:qsd_spectral}
        Let~$(\lambda_{1,\beta},u_{1,\beta})$ be the principal Dirichlet eigenpair of~$-\cL_\beta$ in~$\Omega_\beta$ (which is unique up to the sign of~$u_{1,\beta}$), i.e.
        \begin{equation}
            \lambda_{1,\beta} = \underset{u\in H_{0,\beta}^1(\Omega_\beta)\setminus\{0\}}{\inf}\,\frac{\left\langle-\cL_\beta u,u\right\rangle_{\Lmu(\Omega_\beta)}}{\left\|u\right\|^2_{\Lmu(\Omega_\beta)}} = \frac1\beta\int_{\Omega_\beta}|\nabla u_{1,\beta}|^2\,\e^{-\beta V}.
        \end{equation}
        Then the probability measure
        \begin{equation}
            \label{eq:qsd_spectral}
            A\mapsto\nu(A) = {\frac{\displaystyle\int_A u_{1,\beta}\,\e^{-\beta V}}{\displaystyle\int_{\Omega_\beta}u_{1,\beta}\,\e^{-\beta V}}}
        \end{equation}
        is the unique QSD for the process~\eqref{eq:overdamped_langevin} on~$\Omega_\beta$.
        Moreover, the exit time~$\tau_{\Omega_\beta}$ is exponentially distributed with rate~$\lambda_{1,\beta}$ when the initial conditions follow~$\nu$, and is independent of the exit point:
            \begin{equation}
            \forall\,\varphi\in L^\infty(\partial\Omega_\beta),\qquad\E^\nu\left[\varphi(X^\beta_{\tau_{\Omega_\beta}})\1_{\tau_{\Omega_\beta}>t}\right] = \e^{-\lambda_{1,\beta}t}\E^\nu\left[\varphi(X^\beta_{\tau_{\Omega_\beta}})\right].
        \end{equation}
        \end{proposition}
    The Dirichlet spectrum also provides an upper bound on the relaxation timescale to the QSD.
    Let us define the total variation distance between signed measures~$\nu,\nu'\in\mathcal M(\Omega_\beta)$ by
    \[\|\nu-\nu'\|_{\mathrm{TV}} = \underset{A\in\mathcal B(\Omega_\beta)}{\sup} |\nu(A)-\nu'(A)|.\]
    We have the following result, adapted again from~\cite{LBLLP12}.
    \begin{proposition}[\cite{LBLLP12}]
        \label{prop:decorr}
        Assume that the initial law~$\mu_0 := \mathrm{Law}(X^\beta_0)$ has an~$\Lmu(\Omega_\beta)$ Radon--Nikodym derivative with respect to~$\e^{-\beta V(x)}\,\d x$. Then there exist~$C_1,C_2>0$ such that, for all~$t>0$,
        \begin{equation}
            \label{eq:decorr}
            \left\|\mathrm{Law}^{\mu_0}\left[X_t^\beta\,\middle|\,\tau_{\Omega_\beta}>t\right] -\nu\right\|_{\mathrm{TV}} \leq C_1 \e^{-(\lambda_{2,\beta}-\lambda_{1,\beta})t},
        \end{equation}
        \begin{equation}
            \label{eq:decorr_event}
            \left\|\mathrm{Law}^{\mu_0}\left[\left(X^\beta_{\tau_{\Omega_\beta}},\tau_{\Omega_{\beta}}-t\right)\,\middle|\,\tau_{\Omega_\beta}>t\right] -\mathrm{Law}^\nu\left[\left(X^\beta_{\tau_{\Omega_\beta}},\tau_{\Omega_\beta}\right)\right]\right\|_{\mathrm{TV}} \leq C_2 \e^{-(\lambda_{2,\beta}-\lambda_{1,\beta})t}.
        \end{equation}
    \end{proposition}
    Equation~\eqref{eq:decorr} states that the time-marginal of the process~$X_t^\beta$, conditioned on remaining in the domain during a positive time~$t$, converges to the QSD exponentially fast, and Equation~\eqref{eq:decorr_event} states an analogous result for the law of the exit event.
    The requirement on the initial condition can be considerably weakened by using estimates on the heat semigroup for the diffusion process~\eqref{eq:overdamped_langevin}, see for example~\cite{SL13}.    
    \subsection{Geometric assumptions}
    \label{subsec:geometric_assumptions}
    We now discuss the geometric setting and basic assumptions which will be used in the remainder of this work.
    When considering a Gaussian approximation of the Gibbs measure~\eqref{eq:gibbs_measure} around a minimum of~$V$, one finds that the covariance matrix scales as~$\beta^{-1}$ when~$\beta\to+\infty$. Thus, it appears that the relevant scale to analyze the localization of low-temperature distributions is~$\beta^{-\frac 12}$.
    This heuristic is justified by analysis, and motivates the scaling with respect to~$\beta$ in the following geometric hypothesis.

    \begin{hypothesis}
    The following limit is well-defined in~$\R\cup\{+\infty\}$ for each~$0\leq i < N$:
    \begin{equation}
        \label{hyp:alpha_exists}
        \tag{\bf H1}
        \epsLimit{i} =\underset{\beta\to\infty}{\lim}\, \sqrt\beta\sigma_{\Omega_\beta}(z_i) \in (-\infty,+\infty].
    \end{equation}
    \end{hypothesis}

    Let~$0\leq i < N$. We distinguish two regimes depending on the nature of~$\epsLimit{i}$:
    \begin{itemize}
        \item If~$\epsLimit{i}= +\infty$, we say that~$z_i$ is {\bf far} from the boundary.
        \item If~$\epsLimit{i}$ is finite, we say that~$z_i$ is {\bf close} to the boundary.
    \end{itemize}
    
    When~$\epsLimit{i}<+\infty$, $z_i$ is at a distance of order~$1/\sqrt\beta$ from~$\partial\Omega_\beta$.  Since~$V$ is a Morse function, this implies that any path connecting~$z_i$ to the boundary undergoes energy variations of order~$\beta^{-1}$. This is the characteristic energy scale in the Eyring--Kramers formula, which explains why the scaling in Assumption~\eqref{hyp:alpha_exists} is natural.

    \begin{remark}
        Note that~\eqref{hyp:alpha_exists} excludes the case~$\epsLimit{i}=-\infty$, so that critical points which are far from the boundary but outside the domain do not appear.
        This is merely for convenience, since the forthcoming analysis would be unaffected by the presence of such points.
        Notice in particular that if~$z$ is a critical point of $V$ such that~$\d(z,\Omega_\beta)>c$ for some~$c>0$ and all~$\beta$ sufficiently large, one may ignore it by considering~$\mathcal K\setminus B(z,c)$ instead of~$\mathcal K$.
    \end{remark}
    If a critical point~$z_i \in \Omega_\beta$ is far from the boundary,~$\Omega_\beta$ contains a ball centered around~$z_i$ of radius much larger than~$\beta^{-\frac12}$, namely of radius~$\frac{\sigma_{\Omega_\beta}(z_i)}2$.
    When~$z_i$ is close to the boundary, this is not the case. In order to make quantitative statements on the spectrum, we need to constrain the local geometry of~$\partial\Omega_\beta$ around~$z_i$.
    \begin{hypothesis}     
        There exist functions~$\largeRadius,\smallRadius:\R_+^*\to\R_+$ such that the following holds for~$\beta$ large enough and each~$0\leq i < N$:
        \begin{equation}
            \label{hyp:locally_flat}
            \tag{\bf H2}
            \left\{\begin{aligned}
                &\sqrt\beta\largeRadius(\beta)\xrightarrow{\beta\to\infty}\,+\infty,\\
                &\sqrt\beta\smallRadius(\beta)\xrightarrow{\beta\to\infty}\,0,\\
                &\localNeighborhood[-]{i}(\beta) \subseteq B(z_i,\largeRadius(\beta))\cap\Omega_\beta \subseteq \localNeighborhood[+]{i}(\beta),
            \end{aligned}\right.
        \end{equation}
        where we denote
        \begin{equation}
            \label{eq:capped_balls}
            \localNeighborhood[\pm]{i}(\beta) = z_i + \left[B(0,\largeRadius(\beta)) \cap \halfSpace{i}\left(\frac{\epsLimit{i}}{\sqrt\beta}\pm\smallRadius(\beta)\right)\right],
        \end{equation}
        recalling the definition~\eqref{eq:half_space} of the half-space.
    \end{hypothesis}
    In the case~$\epsLimit{i} = +\infty$, one has~$\localNeighborhood[\pm]{i}(\beta)=B(z_i,\largeRadius(\beta))$, and so Assumption~\eqref{hyp:locally_flat} only imposes~$B(z_i,\largeRadius(\beta))\subset \Omega_\beta$.
    We schematically represent in Figure~\ref{fig:local_neighborhood} the local geometry of~$\partial\Omega_\beta$ under Assumption~\eqref{hyp:locally_flat}, in a case where~$z_i$ is close to the boundary and inside the domain.
    
    Geometrically, the sets~\eqref{eq:capped_balls} correspond to hyperspherical caps centered around~$z_i$ and cut in the eigendirection~$\hessEigvec{i}{1}$ of~$\nabla^2 V(z_i)$. Thus, the condition~\eqref{hyp:locally_flat} fixes the orientation convention for~$\hessEigvec{i}{1}$ in the case where~$z_i$ is close to the boundary, namely that~$\hessEigvec{i}{1}$ always points outwards from~$\partial\Omega_\beta$ (including when~$z_i$ is outside the domain).
    The content of Assumption~\eqref{hyp:locally_flat} is that the boundary can be approximated by a hyperplane normal to~$\hessEigvec{i}{1}$, up to negligible perturbations relative to~$\beta^{-\frac12}$, in a local neighborhood of size~$\largeRadius(\beta)$ around~$z_i$.
    Note that in the case~$d=1$, Assumption~\eqref{hyp:locally_flat} is a direct consequence of Assumption~\eqref{hyp:alpha_exists}. Moreover, from the boundedness~\eqref{hyp:uniformly_bounded},~$\localNeighborhood[-]{i}(\beta)\subset\mathcal K$ and~$\largeRadius(\beta)$ is necessarily uniformly bounded with respect to~$\beta$ under Assumption~\eqref{hyp:locally_flat}.
    
    \begin{remark}
        We will sometimes need~$\delta(\beta)$ to be sufficiently small so that various local approximations are sufficiently precise.
    This desideratum is not restrictive, as Assumption~\eqref{hyp:locally_flat} remains valid upon reducing~$\delta$ or increasing~$\gamma$ by a constant multiplicative factor.
    Consequently, we will often assume at no cost to generality that~$\largeRadius$ is sufficiently small or that~$\smallRadius$ is sufficiently large for the purposes of the analysis.

    Let us also note that since $\mathcal K$ contains only finitely many critical points,~\eqref{hyp:locally_flat} need only be verified locally around each critical point in order to hold globally.
    \end{remark}
    
    The last geometric assumption we make is more technical in nature and relates to the scaling of the quantity~$\delta$ constraining the local geometry of the domains in~\eqref{eq:capped_balls}.
    More precisely, we will sometimes need~$\sqrt\beta\delta(\beta)$ to grow sufficiently fast for the neighborhoods~\eqref{eq:capped_balls} to contain the bulk of the support for various quasimodes.

    The rather mild growth condition on~$\sqrt\beta\largeRadius(\beta)$ we will require is the following.
    \begin{hypothesis}
        The following limit holds for any~$0\leq i<N$:
        \begin{equation}
            \tag{\bf H3}
            \label{hyp:scaling_deltai}
            \underset{\beta\to\infty}{\lim}\, \largeRadius(\beta)\sqrt{\frac{\beta}{\log\beta}}= +\infty.
        \end{equation}
    \end{hypothesis}

    We stress that~\eqref{hyp:scaling_deltai} forces in particular critical points which are far from~$\partial \Omega_\beta$ to be sufficiently far relative to~$1/\sqrt\beta$, namely farther than~$\sqrt{\log \beta}/\sqrt{\beta}$.
    The main use of Assumption~\eqref{hyp:scaling_deltai} is to ensure that for any univariate polynomial $P$, one has~$|P(\beta)|\e^{-\beta \delta(\beta)^2}\xrightarrow{\beta\to+\infty} 0$ (that is,~$\e^{-\beta \delta(\beta)^2}$ decays superpolynomially), a fact which will be used throughout this work to absorb various error terms.

    We summarize relevant length scales of Assumptions~\eqref{hyp:locally_flat},~\eqref{hyp:scaling_deltai} by the following chain of scaling inequalities:
    \begin{equation}
        \label{eq:scaling_inequalities}
        \smallRadius(\beta)\ll \beta^{-\frac12} \ll \sqrt{\frac{\log \beta}{\beta}} \ll \largeRadius(\beta) .
    \end{equation}
    \begin{figure}
    \centering        
    \begin{tikzpicture}[rotate = 20,scale=1.5, transform shape=false]

        \draw[black, thick] plot [smooth] coordinates {(0,-0.7) (1,-1.85) (2,-1.65) (3,-1.8) (4,-1.7) (5,-1.6) (6,-3)}; 
        \draw[dashed] (3,-2.5) circle (0.8  ); 
        \draw (3,-2.5) ++(35:1.6) arc (35:-215:1.6); 
        \draw (3,-2.5) +(25:1.6) -- +(-205:1.6); 
        \draw[dashed] (3,-2.5) +(30:1.6) -- +(-210:1.6); 
        \draw (3,-2.5) +(35:1.6) -- +(-215:1.6); 
        
        \draw[<->,dotted] (3,-2.5) +(0:0) -- +(-110:1.6);
        \draw[->,thick,blue ] (3,-2.5) +(0:0) -- +(90:0.8);
        \draw[<->,dotted] (3,-2.5) +(1.6,0.66) -- +(1.6,0.9);
        
        \draw (0,-0.7) node[above left,scale=0.7,fill=white] {$\partial \Omega_\beta$};
        \draw (3,-2.5) node[left,scale=0.7] {$z_i$};
        \draw (3,-2.5) +(1.7,0.66) node[above right,scale=0.7,fill=white] {$2\smallRadius(\beta)$};
        \draw (3.1,-2.5) +(90:0.4) node[right,scale=0.69,fill=white] {$\beta^{-\frac12}\epsLimit{i}\hessEigvec{i}{1}$};
        \draw (3.1,-2.5) +(-110:1.2) node[above right,scale=0.7,fill=white] {$\largeRadius(\beta)$};
    \end{tikzpicture}
        \caption{\label{fig:local_neighborhood}
        Local geometry of~$\Omega_\beta$ in the neighborhood of a critical point~$z_i$ close to the boundary and inside the domain~$\Omega_\beta$, as prescribed by Assumption~\eqref{hyp:locally_flat}. The relevant length scales are $\smallRadius(\beta)\ll \beta^{-\frac12} \ll \largeRadius(\beta)$, see~\eqref{eq:scaling_inequalities}.
        }
    \end{figure}
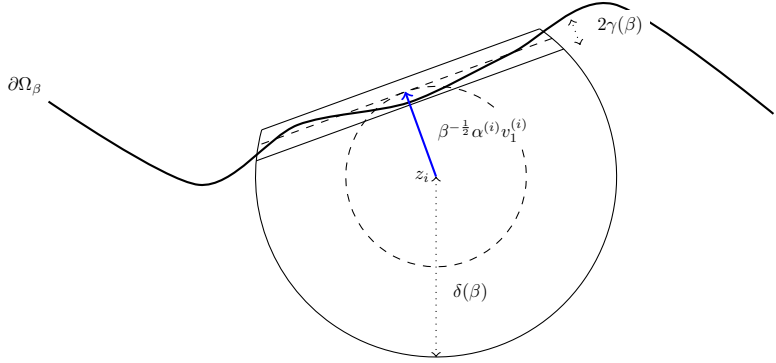
    Assumptions~\eqref{hyp:uniformly_bounded},~\eqref{hyp:alpha_exists},~\eqref{hyp:locally_flat} and~\eqref{hyp:scaling_deltai} are assumed to hold for the remainder of this work.
    \paragraph{Additional assumptions for Theorem~\ref{thm:eyring_kramers}.}
    For the purpose of deriving asymptotics for the metastable exit time~$\lambda_{1,\beta}^{-1}$, we will restrict our setting to the case of domains containing essentially a unique local minimum~$z_0$ far from the boundary (see~\eqref{hyp:one_minimum} below), which moreover contain sufficiently large sublevel sets of~$V$ (see~\eqref{hyp:energy_well} below). Let us make this precise.
    Formally, we first assume the following.
    \begin{hypothesis} 
        The point~$z_0$ is the only local minimum of~$V$ in~$\mathcal K$ which is far from the boundary: 
        \begin{equation}
            \tag{\bf EK1}
            \label{hyp:one_minimum}
           N_0 \geq 1,\qquad \alpha^{(0)} = +\infty,\qquad \forall\,1\leq i<N_0,\,\quad\alpha^{(i)}<+\infty.
        \end{equation}
    \end{hypothesis}
    Note that the basin of attraction~$\basin{z_0}$ (see Equation~\eqref{eq:basin}) is a natural candidate for the metastable state associated with the potential well around~$z_0$.
    It is also a convenient implicit definition from a numerical perspective, since determining whether~$x$ belongs to~$\basin{z_0}$ is as simple as integrating the gradient flow~$\phi_t(x)$ for a sufficiently long time.
    The definition~$\Omega_\beta=\basin{z_0}$ is actually commonly used in accelerated MD, see~\cite{V98}. However, in the context of the ParRep algorithm~(see~\cite{V98}), this definition is not expected to be optimal, and one of the motivations of this work is to rigorously show this and derive improved definitions for the metastable well associated with an energy minimum.
    
    Our analysis almost applies to this special case, with the slight caveat that~$\basin{z_0}$ is typically not a smooth domain, but is instead piecewise smooth (see the proof of Lemma~\ref{lemma:local_minimum} below).
    However, the points at which the boundary~$\partial\basin{z_0}$ fails to be smooth are critical points of~$V$ with index greater than~$1$, which are typically too high in energy to be visited by the dynamics~\eqref{eq:overdamped_langevin} in any reasonable amount of time.
    Thus, one can often circumvent this technical obstacle by considering a local regularization of the boundary excluding higher-index saddle points from~$\Omega_\beta$. In our geometric setting, this situation corresponds to the following parameter values:~$N_0=1$,~$N=N_0+N_1$,~$\epsLimit{0}=+\infty$ and~$\epsLimit{i}=0$ for~$1\leq i<N$.

    It is shown in~\cite{DGLLPN19,LLPN22} that when the domain coincides with the basin~$\basin{z_0}$ (up to a local regularization), the exit distribution starting from the QSD concentrates on the index-1 saddle points of lowest energy separating~$\basin{z_0}$ from the basin of attraction of another local minimum of~$V$. These so-called separating saddle points will also play a distinguished role in our analysis, as in~\cite{HS85,HKN04,BGK05,LPN21,LLPN22}.

    Indeed, we relate the asymptotic behavior of the eigenvalues with local perturbations of the domain near critical points, as formalized by the geometric assumptions~\eqref{hyp:alpha_exists} and~\eqref{hyp:locally_flat}.
    In this context, the asymptotic behavior of the smallest eigenvalue will be especially sensitive to the shape of the boundary near separating critical points with the lowest energy, which turn out to be index-$1$ saddle points.

    Let~$\mathcal M(V)$ denote the set of local minima of~$V$ over~$\R^d$, which is discrete since~$V$ is Morse.
    It is natural to introduce the following quantity
    \begin{equation}
        \label{eq:vstar}
        \Vstar = \underset{z\in \mathcal S(z_0)}{\inf}\, V(z),\qquad \mathcal S(z_0)=\partial \basin{z_0}\cap \left(\bigcup_{m\in\mathcal M(V)\setminus \{z_0\}} \partial\basin{m}\right),
    \end{equation}
    so that~$\Vstar-V(z_0)$ gives the height of the energy barrier separating~$z_0$ from some other basin of attraction for the steepest descent dynamics.
    The set~$\mathcal S(z_0)$ loosely coincides with~$\partial\basin{z_0}$, up to a submanifold consisting of the union of the stable manifolds for critical points separating~$\basin{z_0}$ from itself. We refer to~\cite[Theorem B.13]{MS14} or the proof of Lemma~\ref{lemma:local_minimum} below for more details on this point.
    In particular, it holds that
    \begin{equation}
    \label{eq:s_alt_rep}
    \mathcal S(z_0) = \partial{\overline{\basin{z_0}}}.
    \end{equation}

    We will make the following boundedness assumption.
    \begin{hypothesis}
        \begin{equation}
            \tag{\bf EK2}
            \label{hyp:energy_well_bounded}
            \text{The set }\basin{z_0}\cap \{V<\Vstar\}\text{ is bounded.}
        \end{equation}
    \end{hypothesis}
    In the case~$\Vstar<+\infty$ (which is of course the case of interest), a natural sufficient condition to ensure~\eqref{hyp:energy_well_bounded} is to assume growth conditions on~$V$ at infinity, which are standard in the theoretical study of the stochastic process~\eqref{eq:overdamped_langevin}.

    In turn, by the regularity of~$V$,~\eqref{hyp:energy_well_bounded} implies that~$\Vstar<+\infty$, and thus the set
    \begin{equation}
        \underset{z\in\mathcal S(z_0)}{\mathrm{Argmin}}\,V(z)
    \end{equation}
    is non-empty. In fact, by virtue of Lemma~\ref{lemma:local_minimum} below and the Morse property satisfied by~$V$, the infimum in~\eqref{eq:vstar} is attained at index-$1$ saddle points of~$V$. Let us denote by~$I_{\mathrm{min}}$ the indices associated with these low-energy saddle points, so that we may also write
    \begin{equation}
        \label{eq:vstar_bis}
        I_{\min} = \underset{\stackrel{N_0\leq i <N_0+ N_1}{z_i\in\mathcal S(z_0)}}{\mathrm{Argmin}}\,V(z_i),\qquad \Vstar = \underset{\stackrel{N_0\leq i <N_0+ N_1}{z_i\in\mathcal S(z_0)}}{\min}\,V(z_i).
    \end{equation}

    Notice that, for~$i\in I_{\min}$ such that~$\epsLimit{i} = +\infty$, the orientation convention for~$\hessEigvec{i}{1}$ has not yet been fixed. However, since in this case~$z_i \in \partial \basin{z_0} \cap \partial\left[\R^d\setminus \basin{z_0}\right]$, the stable manifold theorem implies that the outward normal~$\n_{\basin{z_0}}(z_i)$ to~$\partial\basin{z_0}$ is well-defined at~$z_i$, and is furthermore a unit eigenvector of~$\nabla^2 V(z_i)$ for the negative eigenvalue~$\hessEigval{i}{1}$.
    We therefore follow the convention that in this case~$\hessEigvec{i}{1}=\n_{\basin{z_0}}(z_i)$. 

    We require that the domains~$\Omega_\beta$ contain a small energetic neighborhood of the principal energy well~$\basin{z_0}\cap\{V<\Vstar\}$.
    \begin{hypothesis}
        There exist~$C_V,\beta_0>0$ such that, for all~$\beta>\beta_0$,
        \begin{equation}
            \tag{\bf EK3}
            \label{hyp:energy_well}
            \left[\basin{z_0}\cap\{V<\Vstar+C_V\largeRadius(\beta)^2\}\right]\setminus \bigcup_{i\in I_{\min}} B(z_i,\largeRadius(\beta)) \subset \Omega_\beta.
        \end{equation}
    \end{hypothesis}
    Assumption~\eqref{hyp:energy_well} serves as a counterpart to~\eqref{hyp:locally_flat}, mildly constraining the geometry of~$\partial\Omega_\beta$ away from the low-energy saddle points.
    We schematize its meaning in Figure~\ref{fig:energy_neighborhood}.

    For technical reasons, we finally require the quantity~$\delta(\beta)$ appearing in Assumption~\eqref{hyp:locally_flat} to be sufficiently small asymptotically, namely smaller than some positive constant~$\varepsilon_0(V,z_0)$ depending only on~$V$ and~$z_0$, whose expression is given in the proof of~Proposition~\ref{prop:suff_condition_delta} (see Section~\ref{subsec:local_energy_estimates} below).
    \begin{hypothesis}
        There exists~$\beta_0>0$ such that for all~$\beta>\beta_0$,
        \begin{equation}
            \tag{\bf EK4}
            \label{hyp:bound_delta}
            \largeRadius(\beta)<\varepsilon_0(V,z_0),
        \end{equation}
        where~$\varepsilon_0(V,z_0)$ is introduced in Proposition~\ref{prop:suff_condition_delta} below. Additionally, for each~$0\leq i<N$,~$z_i$ is the unique critical point of~$V$ in~$B(z_i,2\varepsilon_0(V,z_0))$.
    \end{hypothesis}
    The role of this hypothesis is to ensure that~$V$ is sufficiently well approximated by its second-order expansion on the $\delta(\beta)$ scale around each~$(z_i)_{i\in I_{\min}}$.
    Since~\eqref{hyp:locally_flat} only constrains the geometry on the scale~$\delta(\beta)$, Assumption~\eqref{hyp:bound_delta} is rather non-restrictive.

    \begin{remark}
        \label{rem:non_ssp}
            Let us note that, together with the regularity of~$\Omega_\beta$, Assumption~\eqref{hyp:energy_well} may force critical points which are low in energy to be far from the boundary, i.e.~$\epsLimit{i}=+\infty$. This may occur when the set
    \begin{equation}
        \label{eq:non_removable_points}
        \mathcal X(z_0) = \left\{\,1\leq i < N\middle|\,z_i\in\partial\basin{z_0}\setminus \mathcal{S}(z_0),\,V(z_i) < \Vstar\right\}
    \end{equation}
    is non-empty.
    Indeed, recall that~$\partial \basin{z_0}$ has Lebesgue measure 0 (see the proof of Lemma~\ref{lemma:stable_manifold} below). It follows that for~$i\in\mathcal X(z_0)$, there exists~$h>0$ such that almost all points in~$B(z_i,h)$ are contained in~$\basin{z_0}$. If not,~$z_i\in\overline{\basin{m}}$ for some local minimum $m\neq z_0$ (since $V$ is bounded from below), which is forbidden by the definitions of~$\mathcal X(z_0)$ and~$\mathcal S(z_0)$. By continuity, we may assume that~$B(z_i,h)\subset\{V<\Vstar\}$ and~$h<\varepsilon_0(V,z_0)$, so that, for sufficiently large~$\beta$, almost all points in~$B(z_i,h)$ are contained in~$\Omega_\beta$ by Assumption~\eqref{hyp:energy_well} (because~$h<\varepsilon_0(V,z_0) \implies B(z_i,h) \cap\bigcup_{j\in I_{\min}} B(z_j,\largeRadius(\beta))=\varnothing$ by Assumption~\eqref{hyp:bound_delta}). Thus~$z_i\in\overline{\Omega}_\beta$, but since~$\Omega_\beta$ is a smooth manifold with boundary, it must in fact hold that~$B(z_i,h)\subset \Omega_\beta$. This obviously implies~$\epsLimit{i}=+\infty$.
    
    Note that, since minima lie in the interior of their basin of attraction,~$i\geq N_0$ for any~$i\in\mathcal X(z_0)$, so that Assumption~\eqref{hyp:energy_well} can never lead to a situation in which~Assumption~\eqref{hyp:one_minimum} cannot be verified.
    \end{remark}

    In Figure~\ref{fig:basin}, we give a schematic representation of the sets~$\basin{z_0},\,\mathcal S(z_0)$ and~$I_{\min}$ in a two-dimensional situation in which~$\mathcal X(z_0)\neq\varnothing$.

    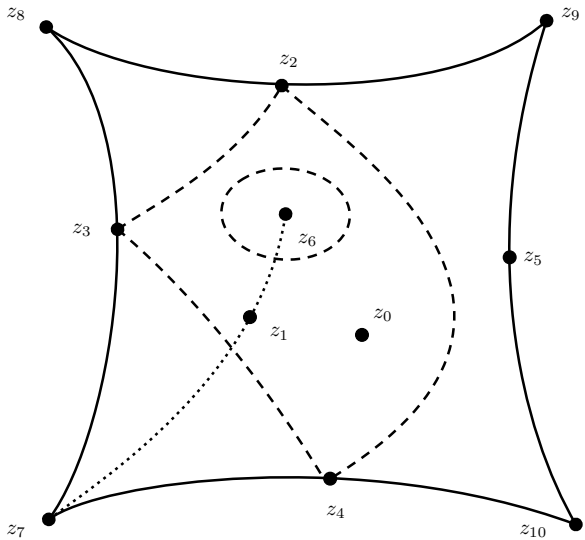
\begin{figure}
        \centering
    \begin{tikzpicture}[scale=1.25]
        \node (background) at (0,-0.16) {\includegraphics[width=0.75\linewidth]{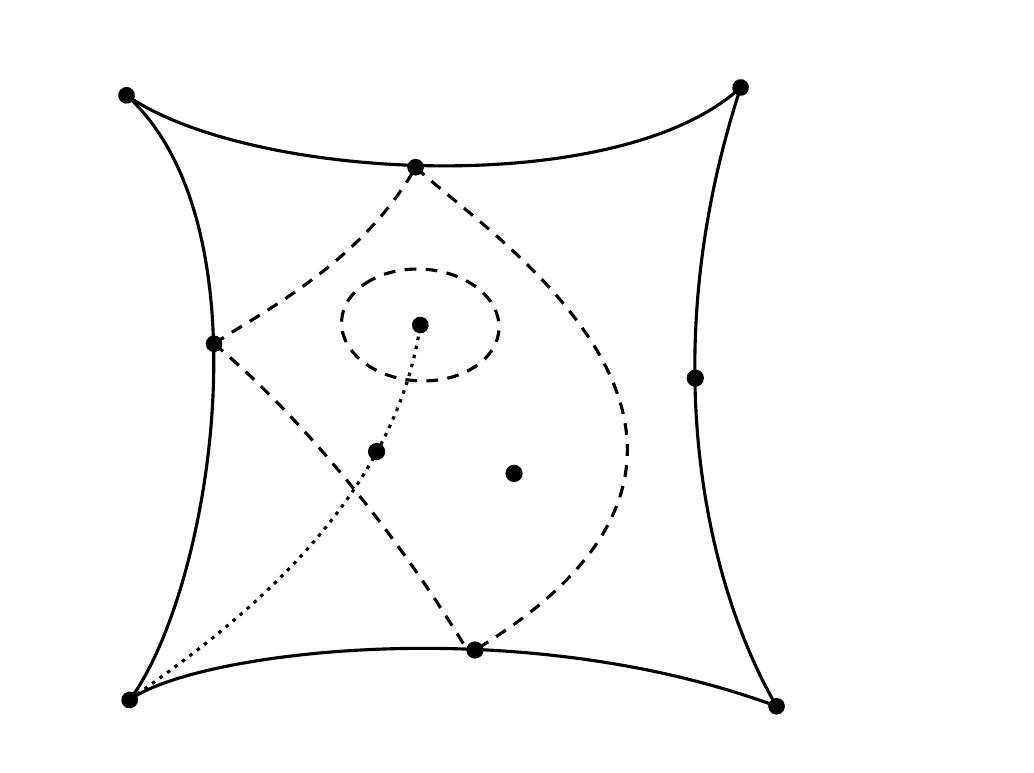}};
        \node[scale=0.8] at (0.2,-0.7) {$z_0$};
        \node[scale=0.8] at (-0.9,-0.9) {$z_1$};
        \node[scale=0.8] at (-0.8,2) {$z_2$};
        \node[scale=0.8] at (-3,0.2) {$z_3$};
        \node[scale=0.8] at (-0.3,-2.8) {$z_4$};
        \node[scale=0.8] at (1.8,-0.1) {$z_5$};
        \node[scale=0.8] at (-0.6,0.1) {$z_6$};
        \node[scale=0.8] at (-3.7,-3) {$z_7$};
        \node[scale=0.8] at (-3.7,2.5) {$z_9$};
        \node[scale=0.8] at (2.2,2.5) {$z_8$};
        \node[scale=0.8] at (1.8,-3) {$z_{10}$};
    \end{tikzpicture}
    \caption{Depiction of a basin~$\basin{z_0}$ in dimension~$d=2$ (see~\eqref{eq:basin}). In solid lines, the set~$\mathcal S(z_0)$ defined in~\eqref{eq:vstar}. The dotted line is the set~$\partial\basin{z_0}\setminus\mathcal S(z_0)$. The dashed line is the level set~$\{V=\Vstar\}$, with~$\Vstar$ defined in~\eqref{eq:vstar_bis}. There are eleven critical points including the minimum~$z_0$, five index-1 saddle points~$z_1,z_2,z_3,z_4$ and~$z_5$. The point $z_1$ is a non-separating saddle point. The remaining points are index-2 saddle points (local maxima). Here, the set~$I_{\min}$ defined in~\eqref{eq:vstar_bis} is given by~$I_{\min}=\{2,3,4\}$, and the set~$\mathcal X(z_0)$ defined in~\eqref{eq:non_removable_points} is given by~$\mathcal X(z_0)=\{1\}$, so that under~\eqref{hyp:energy_well}, it holds in particular that~$\epsLimit{1}=+\infty$. }
    \label{fig:basin}    
    \end{figure}

    \begin{figure}
        \centering
        \begin{tikzpicture}[scale=1.25]
            \node (background) at (0,-0.16) {\includegraphics[width=0.75\linewidth]{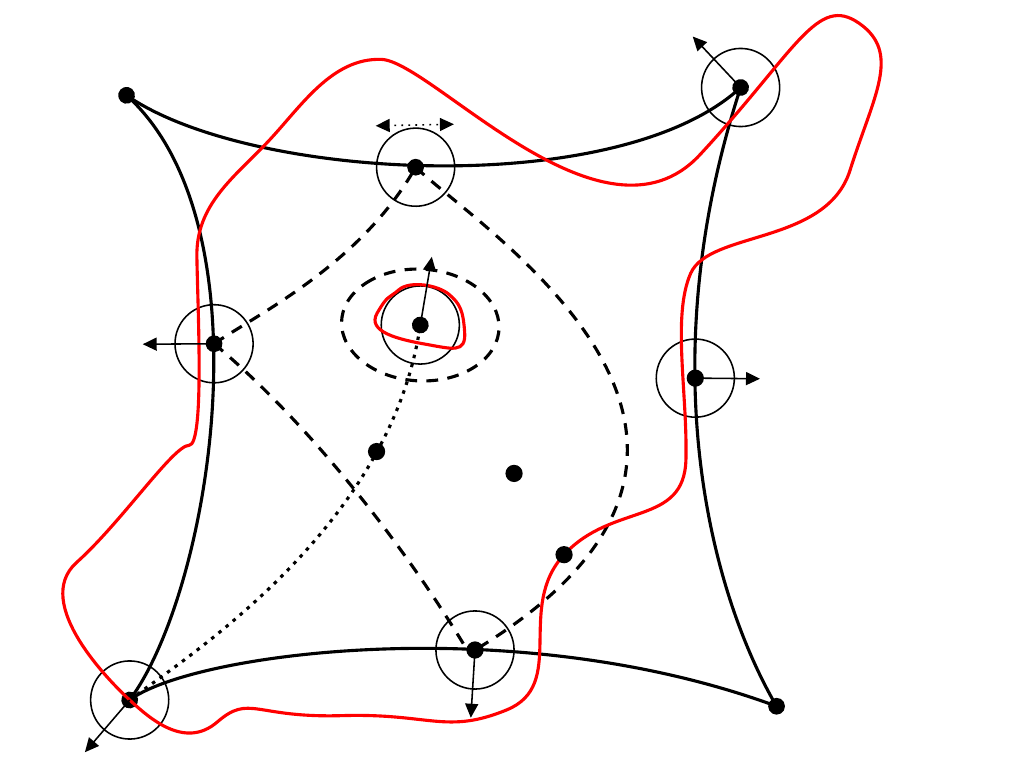}};
            \node[scale=0.8] at (0.4,-1.3) {$z^*$};
            \node[scale=0.8] at (-0.9,2.3) {$\largeRadius(\beta)$};
            \node[scale=0.8] at (-3.5,0.2) {$\hessEigvec{3}{1}$};
            \node[scale=0.8] at (-4,-1.0) {$\partial\Omega_\beta$};
            \node[scale=0.8] at (2.3,-0.1) {$\hessEigvec{5}{1}$};
            \node[scale=0.8] at (-0.6,1.2) {$\hessEigvec{6}{1}$};
            \node[scale=0.8] at (-4,-3.2) {$\hessEigvec{7}{1}$};
            \node[scale=0.8] at (1.8,3) {$\hessEigvec{8}{1}$};
            \node[scale=0.8] at (-0.8,-2.8) {$\hessEigvec{4}{1}$};
        \end{tikzpicture}
        \caption{Illustration of the hypothesis~\eqref{hyp:energy_well} around the basin depicted in Figure~\ref{fig:basin}. The red curve is the boundary of the boundary of a domain satisfying the geometric constraint~\eqref{hyp:locally_flat} (at a fixed value of~$\beta$), but violating~\eqref{hyp:energy_well}. The boundary crosses the level set~$\{V=\Vstar\}$ (see~\eqref{eq:vstar_bis}), and~$z^*$ is therefore a low-energy generalized saddle point~(see Section~\ref{subsec:genericity} below) for this domain. The critical points~$z_9$ and~$z_{10}$ are assumed not to be in the enclosing set~$\mathcal K$ (see Assumption~\eqref{hyp:uniformly_bounded}), so that~$N=9$ here. The points~$z_0,z_1,z_2$ and~$z_4$ are far from the boundary, while the others are close. Note that the orientation convention for~$\hessEigvec{4}{1}$ is fixed by the geometry of~$\basin{z_0}$.}
        \label{fig:energy_neighborhood}
    \end{figure}
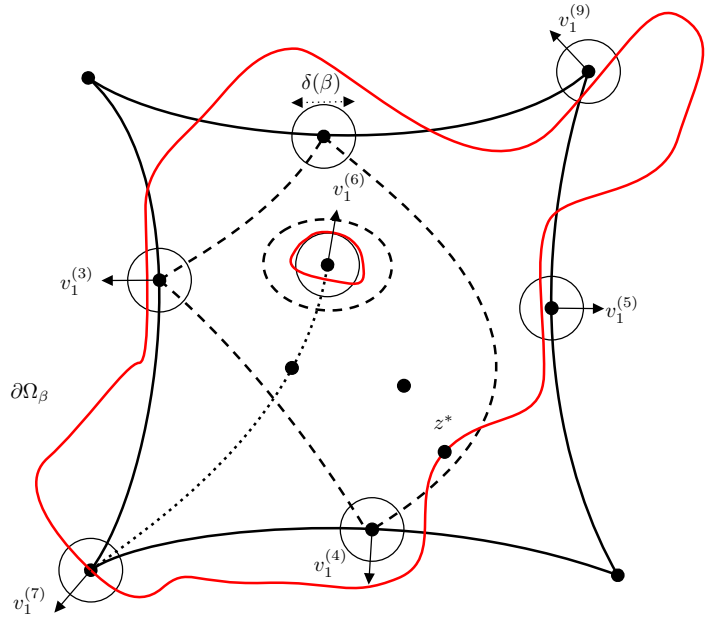

    \subsection{Genericity of the assumptions and comparison with previous work}
    \label{subsec:genericity}
    Let us briefly discuss how the assumptions listed in Sections~\ref{subsec:notation} to~\ref{subsec:geometric_assumptions} relate to previous works on the spectral asymptotics of the Witten Laplacian.
    We stress that our choice of geometric setting is heavily biased towards the optimization problem for~\eqref{eq:separation} mentioned in the introduction. As such, it would be possible to obtain Eyring--Kramers-type asymptotics in more general settings, but for which the separation of timescales~\eqref{eq:separation} would be asymptotically smaller than for those domains to which we restrict ourselves, and so we filter out many such cases through our choice of geometric framework.

    \paragraph{Euclidean setting.} First, we stress that our results are for the time being restricted to the case of parameter-dependent subsets of the Euclidean space $\R^d$, whereas many works from the semiclassical literature~\cite{HKN04,LP10,LN15,LPN21,DGLLPN19} consider the more general setting of Riemannian manifolds, with or without boundaries.
    The distinction is not entirely anecdotal, since the presence of a metric introduces another technical difficulty, which is typically taken care of by a suitable choice of local coordinates.

    Since the natural counterpart of the main geometric assumption~\eqref{hyp:locally_flat} would most likely involve expressing the asymptotic shape of the boundary in such a system of local coordinates, this would lead to conditions which would be difficult to verify in practice. We avoid this difficulty by considering for now the Euclidean case. 
    However, extending the results to the genuine Riemannian case would be of practical interest, since many atomistic trajectories evolve on manifolds (at least, flat tori), and allow one to analyze the case of effective metastable dynamics in free-energy wells, which involve multiplicative noise in the associated diffusion process, see for instance the discussion in~\cite[Section 2.3]{ZHS16}.

    \paragraph{On the choice of the local geometry.}
    Assumption~\eqref{hyp:locally_flat} is of course restrictive, since it states that the shape of the boundary may be roughly parametrized in the local neighborhood of a critical point~$z_i$ close to the boundary by a single scalar parameter~$\epsLimit{i}$. If~$z_i$ is far from the boundary, there is no particular assumption on the local shape of the boundary, except that it must be sufficiently far from~$z_i$, according to~\eqref{hyp:scaling_deltai}.
    However, all works on spectral asymptotics of the Witten Laplacian dealing with boundary conditions impose geometric restrictions at the boundary, either for technical reasons, or for the sake of obtaining analytical formul\ae.
    To assess the restrictiveness of the condition~\eqref{hyp:locally_flat}, and compare it to previous works, it is thus helpful to consider the standard case~$\Omega_\beta = \Omega$ for all $\beta>0$, where~$\Omega\subset\R^d$ is a fixed smooth, bounded and open domain.
    Let~$z_i\in\overline\Omega$ be a critical point. We may ignore all critical points in~$\R^d\setminus \overline\Omega$ by considering~$\mathcal K$ to be a sufficiently small closed neighborhood of~$\overline\Omega$ in~\eqref{hyp:uniformly_bounded}.
    We distinguish two cases.
    \begin{enumerate}
        \item{If~$z_i\in\Omega$, then~$\sigma_\Omega(z_i)>0$ and~$z_i$ is thus far from the boundary. Then,~\eqref{hyp:locally_flat} and~\eqref{hyp:scaling_deltai} hold locally without restriction by setting~$\delta(\beta) = \frac{\sigma_\Omega(z_i)}2,$ and~$\gamma(\beta)=0$ for all $\beta>0$.}
        \item{If~$z_i\in\partial\Omega$, then~$z_i$ is close to the boundary with~$\epsLimit{i}=0$. This is the case of critical points on the boundary analyzed in~\cite{LPN21,LLPN22}. In~\cite{LPN21}, it is assumed that for each $z_i\in\partial \Omega$ such that~$\Ind(z_i)=1$, the outward normal~$\n_{\Omega}(z_i)$ is an eigenvector for the unique negative eigenvalue of~$\Hess V(z_i)$. As we show in Lemma~\ref{lemma:geometry_equivalent} below, this requirement is equivalent to~\eqref{hyp:locally_flat} in the case of a fixed domain for the order-one saddle points lying on the boundary. In~\cite{LLPN22}, for the purpose of obtaining finer estimates on the normal derivative of low-lying eigenmodes, this condition is replaced with the much stronger requirement that~$\partial \Omega$ coincides with the stable manifold~$\mathcal{W}^+(z_i)$ in a neighborhood of~$z_i$.}
    \end{enumerate}

    \begin{lemma}
        \label{lemma:geometry_equivalent}
    Let~$\Omega\subset\R^d$ be a smooth open domain, and set~$\Omega_\beta=\Omega$ for all~${\beta>0}$. Then~\eqref{hyp:uniformly_bounded},~\eqref{hyp:alpha_exists},~\eqref{hyp:locally_flat} and~\eqref{hyp:scaling_deltai} hold if and only if~$\Omega$ is bounded, and for each~$z_i\in\overline\Omega$ such that~$\nabla V(z_i)=0$, either~$z_i\in\Omega$ or~$\n_\Omega(z_i)=\hessEigvec{i}{1}$.
    \end{lemma}
    \begin{proof}
            Assume that \eqref{hyp:uniformly_bounded}--\eqref{hyp:scaling_deltai} hold. Then~$\Omega$ is bounded, according to~\eqref{hyp:uniformly_bounded}. We only need to check the case of a critical point~$z_i\in\partial\Omega$, in which case~$\epsLimit{i}=0$. Since~$\Omega$ is smooth, $\n_\Omega(z_i) = -\nabla \sigma_{\Omega}(z_i) \neq 0$. Let $u\in\R^d$ with~$|u|=1$ and~$u^\intercal\hessEigvec{i}{1}=0$. According to~\eqref{hyp:locally_flat}, for~$\beta>\beta_0$, where~$\beta_0$ is chosen so that~$\beta^{-\frac12}<\largeRadius(\beta)$, 
            the ball $B(z_i+\beta^{-\frac12}u,2\smallRadius(\beta))$ intersects both~$\localNeighborhood[-]{i}(\beta)\subset \Omega$ and~$B(z_i,\largeRadius(\beta))\setminus \localNeighborhood[+]{i}(\beta)\subset \Omega^c$. It follows that
            $\d(z_i + \beta^{-\frac12}u,\partial\Omega) \leq 2\smallRadius(\beta)$, hence:
            \[ \left|u^\intercal\nabla \sigma_{\Omega}(z_i)\right| \leq \underset{\beta\to\infty}{\lim\,\sup}\,\frac{\d(z_i+\beta^{-\frac12}u,\partial\Omega)}{\beta^{-\frac12}} \leq \underset{\beta\to\infty}{\lim\,\sup}\,2\sqrt\beta\smallRadius(\beta)=0.\]
            Since~$u$ was an arbitrary unit vector orthogonal to~$\hessEigvec{i}{1}$, and~$\nabla\sigma_{\Omega}(z_i)\neq 0$,~$\hessEigvec{i}{1}$ is collinear with the unit outward normal~$\mathrm{n}_{\Omega}= -\nabla \sigma_{\Omega}$, so that they are in fact equal, according to the orientation convention imposed by~\eqref{hyp:locally_flat}.

            For the reverse implication, since~$\Omega$ is bounded,~\eqref{hyp:uniformly_bounded} holds with~$\mathcal K = \overline \Omega$. For a critical point~$z_i\in\Omega$,~\eqref{hyp:alpha_exists},~\eqref{hyp:locally_flat} hold locally with~$\epsLimit{i}=+\infty,\largeRadius(\beta) = \sigma_\Omega(z_i)/2$ and $\smallRadius(\beta)=0$.
            
            If~$z_i\in\partial\Omega$,~\eqref{hyp:alpha_exists} holds locally with~$\epsLimit{i}=0$, and by hypothesis~$\n_\Omega(z_i) = \hessEigvec{i}{1}$. We once again rely on the smoothness of~$\partial \Omega$ to write the Taylor expansion
            \[\sigma_{\Omega}(z_i+h) = \sigma_{\Omega}(z_i) + \nabla\sigma_{\Omega}(z_i)^\intercal h + R_i(h) = -h^\intercal\hessEigvec{i}{1} + R_i(h),\]
            where there exist~$C_i,h_{0,i}>0$ such that~$|R_i(h)| \leq C_i|h|^2$ for all~$|h|<h_{0,i}$. Set~$\largeRadius(\beta) = \beta^{\deltaScalingExp-\frac12}$,~$\smallRadius(\beta)=\beta^{\gammaScalingExp-1}$ with~$0<2\deltaScalingExp < \gammaScalingExp < \frac12$. Then~$\largeRadius,\smallRadius$ satisfy the scaling assumptions of~\eqref{hyp:locally_flat} and~\eqref{hyp:scaling_deltai}.
            Recalling the definition of the capped balls~\eqref{eq:capped_balls}, we have~$\sigma_\Omega(\localNeighborhood[-]{i}) \subset [\smallRadius(\beta) - C_i\delta(\beta)^2,+\infty)$ for $\beta>h_{0,i}^{\frac1{s-1/2}}$, and likewise~$\sigma_\Omega(B(z_i,\largeRadius(\beta))\setminus\localNeighborhood[+]{i}) \subset (-\infty,-\smallRadius(\beta) +C_i\delta(\beta)^2]$. It is thus clear that the inclusion~\eqref{hyp:locally_flat} holds locally around~$z_i$ for~$\beta>\beta_{0,i} := \max\{h_{0,i}^{\frac1{s-1/2}},C_i^{\frac1{t-2s}}\}$.
            
            Since~$V$ has finitely many critical points in~$\overline{\Omega}$, the functions~$\largeRadius$ and~$\smallRadius$ satisfy by construction the requirements of Assumption~\eqref{hyp:locally_flat} with~$\beta_0 := \underset{0\leq i < N}{\max}\,\beta_{0,i}$.

    \end{proof}
    Lemma~\ref{lemma:geometry_equivalent} shows that our setting can be understood as a generalization of that of~\cite{LPN21} in the case of temperature-dependent domains in Euclidean space.

    \begin{remark}
        In the general case, the geometry of the boundary~$\partial\Omega_\beta$ may be quite degenerate under~\eqref{hyp:locally_flat}, even close to critical points, as no particular restrictions on~$\partial\Omega_\beta$ are imposed in the strip~$\localNeighborhood[+]{i}(\beta)\setminus\localNeighborhood[-]{i}(\beta)$, except for its smoothness at each~$\beta>0$. For example, curvatures at any boundary point within the strip may diverge arbitrarily fast as~$\beta\to\infty$, or it may happen that~$z_i\in\partial\Omega_\beta$ for all~$\beta>0$ and some~$0\leq i < N$, but that~$\n_{\Omega_\beta}(z_i)$ fails to converge.
        In this respect, our analysis goes beyond the standard setting used for example in~\cite{LPN21}.
    \end{remark}

    \paragraph{On Assumptions~\eqref{hyp:one_minimum} and~\eqref{hyp:energy_well}.}
    A phenomenon of interest in the semiclassical approach to metastability is the interaction between potential wells, see~\cite{HS85,DS99,HKN04,BGK05,LN15,LLPN22}. We consider here, as in~\cite{DGLLPN19}, the case where the unique local minimum of~$V$ in~$\mathcal K$ which is far from the boundary is~$z_0$ (we stress however that this minimum need be neither global nor unique on~$\mathcal K$, provided the other local minima are close to the boundary), and thus there is in some sense only one potential well.
    The motivation for concentrating on this setting is the conclusion of Theorem~\ref{thm:eyring_kramers} below, which shows that the modification to the standard Eyring--Kramers formula arises when low-lying index-1 saddle points are close to the boundary, which therefore prevents any interaction between energy wells separated by these critical points.

    Besides, from the perspective of the shape optimization problem mentioned in the introduction, the more interesting situation occurs when both the first and second eigenvalues are sensitive to the position of the boundary. In the context of MD, domains containing multiple energy minima can be seen as energy superbasins, for which the relaxation timescale is related to the crossing of energy barriers internal to the domain, and which will therefore be insensitive to the position of the boundary.
    We believe however that the analysis can be extended to handle the case of multiple energy wells, by adapting standard arguments (see for instance~\cite{LPN21}) to our temperature-dependent setting.

    In the analysis of the Witten Laplacian with Dirichlet boundary conditions~\cite{HN06,LN15,DGLLPN19,LPN21}, the interesting phenomenon of so-called generalized saddle points (which are not genuine critical points) also plays a role. These are local minima~$z$ of~$V|_{\partial \Omega_\beta}$ for which the normal derivative of~$V$ is positive, i.e.~$\nabla V(z)\cdot \n_{\Omega_\beta}(z) > 0$. In particular, such points are not critical points of~$V$, and can be seen informally as index-one saddle points for the extension of~$V$ by~$-\infty$ outside of~$\Omega_\beta$.
    The role of Assumption~\eqref{hyp:energy_well} is to ensure that generalized saddle points are sufficiently high in energy so that their contribution to the smallest eigenvalue is negligible, and that the main contribution comes from intrinsic low-lying saddle points, indexed by~$I_{\min}$.

    Since the definition of generalized saddle points depends on the geometry of~$\partial\Omega_\beta$ which varies as~$\beta\to\infty$, one way to analyze their contribution would be to place strong geometric constraints on the domains, using analogs of Assumption~\eqref{hyp:locally_flat} and~Proposition~\ref{prop:laplace} in the~$1/\beta$ scaling around a predetermined limiting geometry. 
    However, we do not pursue this direction. Besides, the contribution of non-critical generalized saddle points to the pre-exponential factor is of order~$\sqrt\beta$ rather than~$1$ in the case of a usual saddle point (see for example~\cite{HN06}), while they do not contribute to the harmonic spectrum. From the point of view of maximizing the separation of timescales~\eqref{eq:separation}, it is therefore always asymptotically preferable to consider domains~$\Omega_\beta$ which do not contain generalized index-1 saddle points at the energy level~$\Vstar$.
 
    \section{Statement of the main results}
    \label{sec:main_results}
    We present in this section the main results of this work, giving the leading-order asymptotics of the eigenvalues of~$-\cL_\beta$ in the low-temperature limit.
    The resulting formulas are explicit in terms of easily computable functions, and give in particular quantitative estimates for crucial timescales related to the dynamics~\eqref{eq:overdamped_langevin} conditioned on non-absorption at the boundary,
    namely the metastable exit timescale~$\lambda_{1,\beta}(\Omega_\beta)^{-1}$, and the local relaxation timescale~$\left(\lambda_{2,\beta}(\Omega_\beta)-\lambda_{1,\beta}(\Omega_\beta)\right)^{-1}$.
    As noted in the introduction, this is of practical interest for assessing the efficiency of acceleration methods in MD, such as ParRep~\cite{V98,PUV15,LBLLP12} and ParSplice~\cite{PCWKV16}, as well as providing quantitative estimates for the decorrelation time.
    To estimate~$\lambda_{2,\beta}(\Omega_\beta)$, we make use of a harmonic approximation result (Theorem~\ref{thm:harm_approx}), which we state in Section~\ref{subsec:harm}.
    Under~\eqref{hyp:one_minimum}, the harmonic approximation shows that~$\lambda_{1,\beta}\to 0$ in the limit~$\beta\to\infty$, but without any estimate on the asymptotic rate of decay. To get useful estimates of the separation of timescales~\eqref{eq:separation}, we therefore derive finer asymptotics for~$\lambda_{1,\beta}(\Omega_\beta)$, extending the Eyring--Kramers formula (Theorem~\ref{thm:eyring_kramers}). This result is discussed in Section~\ref{subsec:ek}.
    In Section~\ref{subsec:takeaways}, we present in Corollary~\ref{corr:practical_implication} a consequence of the first two main results, and discuss its relevance to the optimization of the efficiency of parallel replica methods from accelerated MD, with respect to the definitions of the metastable states.

    \subsection{Harmonic approximation of the Dirichlet spectrum}
    \label{subsec:harm}
    We generalize the harmonic approximation of~\cite{S83,CFKS87} to the case of homogeneous Dirichlet conditions in a temperature-dependent domain, treating the case in which the asymptotic geometry near critical points is prescribed by the assumptions of Section~\ref{subsec:geometric_assumptions}.

    The first main result involves an auxiliary unbounded operator, defined by a direct sum of quantum harmonic oscillators
    \begin{equation}
        \label{eq:harm_approx_informal}
        H_{\beta,\alpha}^{\mathrm{H}} = \bigoplus_{i=0}^{N-1} H_{\beta,\epsLimit{i}}^{(i)},
    \end{equation}
    where each harmonic oscillator $H_{\beta,\epsLimit{i}}^{(i)}$, understood as some local model associated with the critical point $z_i$, is endowed with homogeneous Dirichlet boundary conditions at the boundary of the half-space~$z_i + \halfSpace{i}\left(\frac{\epsLimit{i}}{\sqrt\beta}\right)$, recalling the definition~\eqref{eq:half_space}.

    A precise definition of the operator~$H_{\beta,\alpha}^{\mathrm{H}}$ is given by the construction contained in Sections~\ref{subsec:local_oscillators}--\ref{subsec:global_harm_operator}.
    This construction shows that for any~$k\geq 1$, the $k$-th smallest eigenvalue of $H_{\beta,\alpha}^{\mathrm{H}}$ is given by~$\beta\lambda_{k,\alpha}^{\mathrm{H}}$ for some $\lambda_{k,\alpha}^{\mathrm{H}}\geq 0$ independent of $\beta$. We prove the following result.

    \begin{theorem}
        \label{thm:harm_approx}
        Assume that~\eqref{hyp:uniformly_bounded},~\eqref{hyp:alpha_exists},~\eqref{hyp:locally_flat} and~\eqref{hyp:scaling_deltai} hold.
        For any~$k\geq 1$, the following limit holds:
        \begin{equation}
            \label{eq:harm_limit}
            \underset{\beta\to\infty}{\lim}\,\lambda_{k,\beta} = \lambda_{k,\alpha}^{\mathrm{H}},
        \end{equation}
        where
        \begin{equation}
            \label{eq:global_alpha}
            \alpha = \left(\epsLimit{i}\right)_{0\leq i < N},
        \end{equation}
        the~$\epsLimit{i}$ are defined in~\eqref{hyp:alpha_exists} and~$\lambda_{k,\beta}$ is the~$k$-th Dirichlet eigenvalue of~$-\cL_\beta$ (where~$\cL_\beta$ is defined in~\eqref{eq:generator}) considered on~$\Lmu(\Omega_\beta)$.
    \end{theorem}

    Crucially, the limiting eigenvalue~$\lambda_{k,\alpha}^{\mathrm{H}}$ is expressible in terms of the eigenvalues of one-dimensional harmonic oscillators on the real half-line (with homogeneous Dirichlet boundary conditions) making their numerical approximation computationally feasible.
    
    The fact that the spectrum of $H_{\beta,\alpha}^{\mathrm{H}}$ scales linearly with $\beta$ is related to the expression~\eqref{eq:witten_laplacian} for the Witten Laplacian~$H_\beta$, which is unitarily equivalent to~$-\beta \cL_\beta$. In fact, the operator~$H_{\beta,\alpha}^{\mathrm{H}}$ should be understood as a harmonic approximation to~$H_\beta$, in the spirit of~\cite[Chapter 11]{CFKS87}.
    The proof of Theorem~\ref{thm:harm_approx} relies on the construction of approximate eigenvectors for~$H_\beta$, or so-called harmonic quasimodes, and is given in Section~\ref{sec:proof_harm_approx}.

    \subsection{A modified Eyring--Kramers formula}
    \label{subsec:ek}

    When~\eqref{hyp:one_minimum} is satisfied, Theorem~\ref{thm:harm_approx} implies that~$0=\lambda_{1,\alpha}^{\mathrm{H}}<\lambda_{2,\alpha}^{\mathrm{H}}$, which is insufficiently precise to obtain an asymptotic equivalent for the principal eigenvalue~$\lambda_{1,\beta}$ in the limit~${\beta\to +\infty}$.
    The second main result provides such an equivalent, under the additional assumptions discussed in Section~\ref{subsec:geometric_assumptions}.
    \begin{theorem}
        \label{thm:eyring_kramers}
        Assume that the hypotheses of Theorem~\ref{thm:harm_approx} hold, as well as~\eqref{hyp:one_minimum},~\eqref{hyp:energy_well_bounded},~\eqref{hyp:energy_well} and~\eqref{hyp:bound_delta}.
        The following estimate holds in the limit~$\beta\to+\infty$:
        \begin{equation}
            \label{eq:eyring_kramers}
            \lambda_{1,\beta} = \e^{-\beta(\Vstar-V(z_0))} \left[\sum_{i\in I_{\min}} \frac{|\hessEigval{i}{1}|}{2\pi\Phi\left(|\hessEigval{i}{1}|^{\frac12}\epsLimit{i}\right)}\sqrt{\frac{\det \nabla^2 V(z_0)}{\left|\det \nabla ^2 V(z_i)\right|}} \left(1 +\varepsilon_i(\beta)\right)\right],
        \end{equation}
        where
        \begin{equation}
            \label{eq:gaussian_cdf}
            \Phi(x) =  \displaystyle\frac1{\sqrt{2\pi}}\displaystyle\int_{-\infty}^x\e^{-\frac{t^2}2}\,\d t,\qquad \Phi(+\infty) = 1,
        \end{equation}
        and
        \begin{equation}
            \label{eq:eyring_kramers_error}
            \varepsilon_i(\beta) =\begin{cases}\O(\beta^{-1}\largeRadius(\beta)^{-2}),&\epsLimit{i}=+\infty,\\
                \O(\sqrt\beta\smallRadius(\beta)+\beta^{-1}\largeRadius(\beta)^{-2}+\beta^{-\frac12}),&\epsLimit{i}<+\infty.\end{cases}
        \end{equation}
    \end{theorem}
    According to Proposition~\ref{prop:qsd_spectral}, Theorem~\ref{thm:eyring_kramers} gives an asymptotic equivalent for the inverse of the mean exit time of the dynamics~\eqref{eq:overdamped_langevin} initially distributed according to the QSD in~$\Omega_\beta$.
     \begin{remark}
        By restricting oneself to the geometric setting in which the boundary is flat in a positive neighborhood of each~$(z_i)_{i\in I_{\min}}$ (that is~$\smallRadius(\beta)=0$ and~$\largeRadius(\beta)=\eta$ in~\eqref{hyp:locally_flat} for some small~$\eta>0$), one recovers the optimal error terms
        \[
        \varepsilon_i(\beta)=\begin{cases}
            \O(\beta^{-1}),&\epsLimit{i}=+\infty,\\
            \O(\beta^{-\frac12}),&\epsLimit{i}<+\infty,
        \end{cases}
        \]
        in accordance with the analysis performed in~\cite{LPN21}. On the other hand, aiming for maximum geometric flexibility, Assumptions~\eqref{hyp:locally_flat} and~\eqref{hyp:scaling_deltai} only imply that, in the limit~$\beta\to\infty$,
        \[\sqrt\beta\smallRadius(\beta)=\smallo(1),\qquad\beta^{-1}\largeRadius(\beta)^{-2}=\O(|\log\beta|^{-1}),\]
        and thus the remainder~\eqref{eq:eyring_kramers_error} is~$\O(|\log\beta|^{-1})$ in the case~$\epsLimit{i}=+\infty$, and not quantitative in the case~$\epsLimit{i}<+\infty$.
    \end{remark}

    \begin{remark}
    The modification with respect to the Eyring--Kramers formula obtained in the boundaryless case in~\cite{HS85,BGK05,HKN04} and in the case with boundary in~\cite{HN06,LPN21} concerns the prefactor, which depends on the asymptotic distances of the low-lying saddle points to the boundary.
    Let us fix a family of domains~$(\Omega_t)_{t\in[-1,1]}$ and~$i\in I_{\min}$ such that~$\sigma_{\Omega_t}(z_i)=at$ for all~$-1\leq t \leq 1$, i.e. the boundary crosses~$z_i$ in such a way that the geometry of the boundary is prescribed by~\eqref{hyp:locally_flat} (say with~$\epsLimit{i}(t)=a t\sqrt\beta$,~$\largeRadius(\beta)=\varepsilon>0$ and~$\smallRadius(\beta)=0$) for some large~$\beta$. One can informally see Theorem~\ref{thm:eyring_kramers} as an analysis of the transition in the pre-exponential factor of~$\lambda_{1,\beta}(\Omega_t)$ as the saddle point crosses the boundary.
    Assuming for simplicity that~$I_{\min}=\{i\}$, writing~$\lambda_{1,\beta}(\Omega_t)=C(t)\e^{-\beta(\Vstar-V(z_0))}$, and formally substituting~$\epsLimit{i}(t)$ in~\eqref{eq:eyring_kramers}, Theorem~\ref{thm:eyring_kramers} suggests that
    \[C(t)\approx \frac{|\hessEigval{i}{1}|}{2\pi\Phi\left(|\hessEigval{i}{1}|^{\frac12}at\sqrt\beta\right)}\sqrt{\frac{\det\,\nabla^2 V(z_0)}{|\det\,\nabla^2 V(z_i)|}}.\]
    We stress that this is only a formal interpretation, since the remainder term~\eqref{eq:eyring_kramers_error} depends non-uniformly on~$t$.
    The prefactor is halved as~$t$ goes from $0$ to $1$, with a sharp transition occurring on the scale~$\left(\beta|\hessEigval{i}{1}|\right)^{-\frac12}$. This is in accordance with the probabilistic interpretation of~$\lambda_{1,\beta}$, since half of the trajectories of the dynamics~\eqref{eq:overdamped_langevin} which reach~$z_i$ are expected to return to a small neighborhood of~$z_0$ before transitioning to another potential well.
    On the other hand, the asymptotic estimate, in the limit~$x\to-\infty$,
    \begin{equation}
        \label{eq:phi_tail_asymptotic}
        \Phi(x) = \frac{1+\O(|x|^{-2})}{|x|\sqrt{2\pi}}\e^{-\frac{x^2}2}
    \end{equation}
    suggests again that
    \[\lambda_{1,\beta}(\Omega_{-1})\approx a\sqrt{\frac{\beta}{2\pi}}|\hessEigval{i}{1}|^{\frac32}\sqrt{\frac{\det\,\nabla^2 V(z_0)}{|\det\,\nabla^2 V(z_i)|}}\e^{-\beta(\Vstar-\frac12|\hessEigval{i}{1}|a^2-V(z_0))}.\]
    Although the case~$\epsLimit{i}=-\infty$ is not covered in this work (we nevertheless expect that our setting can be extended to handle this case as well), we note that this heuristic is in agreement with the results of~\cite{HN06}. Indeed, writing~$z^*(a)=z_i-a\hessEigvec{i}{1}$, for small~$a$, this approximation corresponds to
    \[\lambda_{1,\beta}(\Omega_{-1})\approx \sqrt{\frac{\beta}{2\pi}}|\nabla  V(z^*(a))|\sqrt{\frac{\det\,\nabla^2 V(z_0)}{\det\,\nabla^2_{\partial \Omega_{-1}}V(z^*(a))}}\e^{-\beta (V(z^*(a))-V(z_0))},\]
    where~$\nabla^2_{\partial \Omega_{-1}}$ denotes the Hessian on the submanifold~$\partial \Omega_{-1}$. This corresponds to the standard asymptotic behavior in the case where~$z^*(a)\in \partial \Omega_{-1}$ is a so-called generalized saddle point (see for instance~\cite{HN06,LPN21}).
    \end{remark}

    The proof of Theorem~\ref{thm:eyring_kramers}, which relies on the construction of accurate quasimodes for the principal Dirichlet eigenvector inspired by~\cite{LPN21}, and a modified Laplace method (Proposition~\ref{prop:laplace}), is performed in Section~\ref{sec:proof_ek}.

    Before proving Theorems~\ref{thm:harm_approx} and~\ref{thm:eyring_kramers}, we briefly and informally discuss some implications of these results for the problem of maximizing the timescale separation~\eqref{eq:separation}.
    \subsection{Practical implications of the asymptotic analysis}
    \label{subsec:takeaways}
    In this section, we briefly and informally highlight the key implications of our results for the selection of metastable states and the estimation of associated timescales in MD simulations. For this purpose, we assume in this section that Assumptions~\eqref{hyp:uniformly_bounded}--\eqref{hyp:scaling_deltai} and~\eqref{hyp:one_minimum}--\eqref{hyp:bound_delta} are satisfied.
    
    Theorems~\ref{thm:harm_approx} and~\ref{thm:eyring_kramers} give asymptotic equivalents for the metastable exit rate~$\lambda_{1,\beta}$ and the convergence rate~$\lambda_{2,\beta}-\lambda_{1,\beta}$ to the QSD inside~$\Omega_\beta$ (also known as the decorrelation rate).
    In the low-temperature regime,~$\lambda_{1,\beta}$ converges exponentially fast to zero, and~$\lambda_{2,\beta}$ is asymptotically bounded from below.
    Explicitly, we obtain from the statement of Theorem~\ref{thm:eyring_kramers} and the proof of Theorem~\ref{thm:harm_approx} the following result.

    \begin{corollary}
        \label{corr:practical_implication}
        Suppose that Assumptions~\eqref{hyp:uniformly_bounded}--\eqref{hyp:scaling_deltai} and~\eqref{hyp:one_minimum}--\eqref{hyp:bound_delta} hold. Then
        \begin{equation}
            \label{eq:eigval_asymptotics}        
            \begin{aligned}
                \lambda_{1,\beta} &\widesim{\beta\to\infty}\e^{-\beta(\Vstar-V(z_0))} \left[\sum_{i\in I_{\min}} \frac{|\hessEigval{i}{1}|}{2\pi\Phi\left(|\hessEigval{i}{1}|^{\frac12}\epsLimit{i}\right)}\sqrt{\frac{\det \nabla^2 V(z_0)}{\left|\det \nabla ^2 V(z_i)\right|}}\right],\\
                \lambda_{2,\beta}&\widesim{\beta\to\infty}\min\left\{\underset{1\leq j\leq d}{\min}\,\hessEigval{0}{j},\underset{1\leq i<N}{\min}\,\left[\abs{\hessEigval{i}{1}}\mu_{0,\epsLimit{i}\sqrt{|\hessEigval{i}{1}|/2}}-\frac{\hessEigval{i}{1}}2 + \sum_{2\leq j\leq d}\abs{\hessEigval{i}{j}}\1_{\hessEigval{i}{j}<0}\right]\right\},
            \end{aligned}
        \end{equation}
        where~$\mu_{0,\epsLimit{i}\sqrt{|\hessEigval{i}{1}|/2}}$ is the principal eigenvalue of the canonical harmonic oscillator~$-\frac12\left(\partial_x^2-x^2\right)$ acting on the spatial domain~$\left(-\infty,\epsLimit{i}\sqrt{|\hessEigval{i}{1}|/2}\right)$ with homogeneous Dirichlet boundary conditions.
    \end{corollary}
    \begin{proof}
        The asymptotic equivalent for $\lambda_{1,\beta}$ follows immediately from Theorem~\ref{thm:eyring_kramers}, while the equivalent for~$\lambda_{2,\beta}$ follows from the proof of Theorem~\ref{thm:harm_approx}, namely the expression~\eqref{eq:lambda_2_expr} in Section~\ref{subsec:global_harm_operator} below.
    \end{proof}
    Both asymptotic equivalents in the formulas~\eqref{eq:eigval_asymptotics} can be easily computed in practice, given the knowledge of all the critical points of~$V$ inside~$\mathcal K$ and the eigenvalues of the Hessian~$\nabla^2 V$ at these points. The formulas~\eqref{eq:eigval_asymptotics} provide insight into how to tune acceleration methods such as ParRep in the low-temperature limit.
    \paragraph{Harmonic approximation of the decorrelation time.} A key choice in ParRep-like algorithms is the selection of a decorrelation time associated with a domain~$\Omega_\beta$ (which here we take for the sake of generality to be temperature-dependent). Heuristically, Proposition~\ref{prop:qsd_spectral} suggests that a natural decorrelation time is given by~$n_{\mathrm{corr}}/(\lambda_{2,\beta}-\lambda_{1,\beta})$, where~$n_{\mathrm{corr}}>0$ is a tolerance hyperparameter of our choosing.
    This choice is common in materials science, and~$\lambda_{2,\beta}$ is then usually further approximated by modeling the basin of attraction~$\basin{z_0}$ as a harmonic potential well, see for instance~\cite[Section 2.10]{PUV15}. In our setting, this corresponds to the approximation~$\lambda_{2,\beta}\approx \underset{1\leq j\leq d}{\min}\,\hessEigval{0}{j}$. The formula~\eqref{eq:eigval_asymptotics} shows that this approximation is however not valid in general (even in the case~$\Omega_\beta=\basin{z_0}$), but that it is possible to get a valid approximation by taking into account the eigenvalues of harmonic oscillators at the other critical points.
    \paragraph{Optimization of the asymptotic timescale separation~\eqref{eq:separation} with respect to~$\left(\epsLimit{i}\right)_{i\in I_{\min}}$.}
    The formulas~\eqref{eq:eigval_asymptotics} also provide insight into the problem of maximizing the separation of timescales ${(\lambda_{2,\beta}-\lambda_{1,\beta})/\lambda_{1,\beta}}$ with respect to the shape of the domain. Equivalently, we aim to maximize~$\lambda_{2,\beta}/\lambda_{1,\beta}$. As mentioned in the introduction, this optimization problem can be addressed using numerical methods, at least in cases where low-dimensional representations of the system can be used, see~\cite{BLS25b}. Here we focus on the optimization in the semiclassical limit.
    Obviously, there is a caveat in the fact that our geometric assumptions (see Section~\ref{subsec:geometric_assumptions}) restrict the class of domains with respect to which we optimize, and our results only give asymptotically optimal prescriptions in the limit~${\beta\to\infty}$ for the choice of~$(\epsLimit{i})_{0\leq i < N}$. We can nevertheless make some observations.

    First, the ratio~$\lambda_{2,\beta}/\lambda_{1,\beta}$ diverges at an exponential rate which is independent of the parameter~$\alpha$ in the limit~$\beta\to\infty$. However, the prefactor~$\e^{-\beta(\Vstar-V(z_0))}\lambda_{2,\beta}/\lambda_{1,\beta}$ converges to a finite limit, which is a function of~$\alpha$, and thus it is in fact this quantity that we wish to maximize.
    Second, the parameters~$\left(\epsLimit{i}\right)_{i\in\{0\}\cup\mathcal X(z_0)}$ are set to~$+\infty$ by our geometric assumptions, where we recall the definition~\eqref{eq:non_removable_points} of the set~$\mathcal X(z_0)$.
    Third, each critical point~$z_i$ with~$1\leq i<N$ contributes a lowest eigenvalue
    \[\lambda_i(\alpha^{(i)}) := \abs{\hessEigval{i}{1}}\mu_{0,\epsLimit{i}\sqrt{|\hessEigval{i}{1}|/2}}-\frac{\hessEigval{i}{1}}2 + \sum_{2\leq j\leq d}\abs{\hessEigval{i}{j}}\1_{\hessEigval{i}{j}<0}\]
    to the harmonic spectrum.
    These last two facts imply that there is an upper bound on the asymptotics of~$\lambda_{2,\beta}$, with respect to~$(\epsLimit{i})_{i\in I_{\min}}$, namely
    \begin{equation}
        \label{eq:upper_bound}
        \ell(z_0)  = \min\left\{\underset{1\leq j\leq d}{\min}\,\hessEigval{0}{j},\underset{i\in\mathcal X(z_0)}{\min}\,\lambda_i(+\infty)\right\}=\min\left\{\underset{1\leq j\leq d}{\min}\,\hessEigval{0}{j},\underset{i\in\mathcal X(z_0)}{\min}\,\sum_{1\leq j\leq d}\abs{\hessEigval{i}{j}}\1_{\hessEigval{i}{j}<0}\right\}>0.
    \end{equation}
    In many cases, the set~$\mathcal X(z_0)$ is empty, and so~$\ell(z_0)$ is in this case simply given by the smallest eigenvalue of the Hessian~$\nabla^2 V(z_0)$ at the minimum.
    If, for~$i\in I_{\min}$, it holds that~$\lambda_i(+\infty)=|\hessEigval{i}{1}|\geq\ell(z_0)$ (i.e. when the unstable mode at the saddle point is sharp enough), then the choice~$\epsLimit{i}=+\infty$ is in fact optimal, since it maximally decreases the asymptotic behavior of~$\lambda_{1,\beta}$ without affecting that of~$\lambda_{2,\beta}$.
    If for one or more~$z_i$ with~$i\in I_{\min}$, it holds that~$\lambda_i(+\infty)<\ell(z_0)$, then there is a genuine optimization problem, but which typically involves a small number of parameters.
    
    In the latter case, there exists indeed an optimal value~$\left(\alpha^{(i)\star}\right)_{i\in I_{\min}}\in (-\infty,+\infty]^{|I_{\min}|}$. To see this, we rely on the Gaussian tail estimate~\eqref{eq:phi_tail_asymptotic} and Lemma~\ref{lemma:principal_eigenvalue_dirichlet} to show that, as~$\alpha^{(i)}$ tends to~$-\infty$, the value~$\lambda_i(\epsLimit{i})$ will eventually become larger than~$\ell(z_0)$, whereas the prefactor for~$\lambda_{1,\beta}$ will tend to~$+\infty$. This implies that any maximizing sequence of parameters is bounded from below by some~$C>-\infty$ in each of its components.
    We endow~$(-\infty,+\infty]$ with the one-point compactified topology at~$+\infty$.
    In this topology, the claim then follows from compactness of~$[C,+\infty]^{|I_{\min}|}$, the continuity at~$+\infty$ of~$\Phi$ (defined in~\eqref{eq:gaussian_cdf}) and of the principal Dirichlet eigenvalue~$\theta\mapsto\mu_{0,\theta}$ of the one-dimensional Dirichlet oscillator (which will be obtained in~\eqref{eq:principal_eigenvalue_infinity} from Lemma~\ref{lemma:principal_eigenvalue_dirichlet} below).
    
    \paragraph{Optimization with respect to the other parameters.}
    The remaining tunable parameters are those not corresponding to low-lying separating saddle points. These are given by the vector of parameters~$\mathcal P(\alpha) =\left(\epsLimit{i},\,i\not\in \{0\}\cup\mathcal X(z_0)\cup I_{\min}\right)$. Indeed, for~$i\in \{0\}\cup\mathcal X(z_0)$, it necessarily holds that~$\epsLimit{i}=+\infty$ (see Remark~\ref{rem:non_ssp} above), and the case~$i\in I_{\min}$ is treated in the previous paragraph.
    Noticing that the asymptotic behavior of~$\lambda_{1,\beta}$ is only a function of~$(\epsLimit{i})_{i\in I_{\min}}$, it follows that the asymptotic optimization problem reduces to a maximization of~$\lim_{\beta\to\infty}\lambda_{2,\beta}$ with respect to the parameter~$\mathcal P(\alpha)$. By domain monotonicity (see Proposition~\ref{prop:comparison_principle}), all the quantities~$\lambda_i(\epsLimit{i})$
    are decreasing functions of~$\alpha^{(i)}$.
    In particular, the components of~$\mathcal P(\alpha)$ may be sent to~$-\infty$ without affecting the asymptotic prefactor~$\e^{-\beta(\Vstar-V(z_0))}\lambda_{2,\beta}/\lambda_{1,\beta}$.
    
    One option is to entirely disregard the corresponding critical points by not including them in~$\mathcal K$ from the start. This reduces the problem to that of finding an asymptotically optimal perturbation of the set~$\basin{z_0}\cap\{V<\Vstar\}$.

    However, in general, many values of~$\mathcal P(\alpha)$ will yield asymptotically optimal parameters. We now describe a full optimization procedure using the following steps.
    \begin{itemize}
        \item{First, solve the optimization problem with respect to~$(\epsLimit{i})_{i\in I_{\min}}$, following the procedure described in the previous paragraph, and neglecting the contribution of critical points which are not low-energy separating saddle points. In other words, find
        \begin{equation}
            \label{eq:reduced_opt_problem}
            \alpha^{\star}_{I_{\min}}\in\underset{\alpha\in (-\infty,+\infty]^{|I_{\min}|}}{\mathrm{Argmax}}\frac{\displaystyle\min\left\{\ell(z_0),\underset{i\in I_{\min}}{\min}\,\lambda_i(\epsLimit{i})\right\}}{\displaystyle\sum_{i\in I_{\min}} \frac{|\hessEigval{i}{1}|}{2\pi\Phi\left(|\hessEigval{i}{1}|^{\frac12}\epsLimit{i}\right)}\sqrt{\frac{\det \nabla^2 V(z_0)}{\left|\det \nabla ^2 V(z_i)\right|}}}.
        \end{equation}
        For each optimum~$\alpha^{\star}_{I_{\min}} =(\alpha^{(i)\star})_{i\in I_{\min}}$, there exists an associated optimal harmonic eigenvalue
        \[\lambda\left(\alpha^{\star}_{I_{\min}}\right) = \min\left\{\ell(z_0),\underset{i\in I_{\min}}{\min}\,\lambda_i(\alpha^{(i)\star})\right\}.\]}
        \item{Then, any value of~$\alpha^{(i)}\in\mathcal P(\alpha)$ for which $\lambda_i(\alpha^{(i)})$ is larger than or equal to~$\lambda(\alpha^\star_{I_{\min}})$ is optimal. In other words, the structure of the set of optimal parameters with respect to~$\mathcal P(\alpha)$ is particularly simple: it is the Cartesian product
    \[\prod_{i\not\in \{0\}\cup\mathcal X(z_0)\cup I_{\min}} (-\infty,\alpha^{(i)\star}],\qquad \text{where }\lambda_i(\alpha^{(i)\star})=\lambda(\alpha^{\star}_{I_{\min}}).\]}
    The full set of optimal~$\alpha$ can be deduced by taking a union over the set of optimizers~$\alpha^{\star}_{I_{\min}}$.
    \end{itemize}

    For a system which is reasonably isotropic (in the sense that the~$\hessEigval{i}{j}$ do not span many orders of magnitude), it is sensible to expect that for many saddle points~$z_i$ such that~$\alpha^{(i)}\in \mathcal P(\alpha)$, any~$\alpha^{(i)\star}\in(-\infty,+\infty]$ is optimal, particularly if~$\mathrm{Ind}(z_i)\gg 1$. For such points, the asymptotic separation of timescales is insensitive at leading order to the choice of~$\epsLimit{i}$.

    \paragraph{The effect of other minima.}
    In the particular case where~$z_i$ is a local minimum (i.e. $1\leq i<N_0$), it holds from Equation~\eqref{eq:principal_eigenvalue_infinity} (see Lemma~\ref{lemma:principal_eigenvalue_dirichlet} below)~that~$\underset{\epsLimit{i}\to+\infty}{\lim}\,{\lambda_i(\alpha^{(i)})}=0$, and thus~$\alpha^{(i)\star}$ is finite.
    This indicates that, in the low-temperature regime, the separation of timescales ultimately degrades when moving from a domain containing one minimum far from the boundary (such as a neighborhood of~$\basin{z_0}$) to one containing several, such as an energy superbasin. This suggests the existence of a locally optimal domain around~$\basin{z_0}$, which is indeed observed numerically in~\cite{BLS25b}.
    This observation also motivates a posteriori the choice of Assumption~\eqref{hyp:one_minimum}, which restricts the class of considered domains to the vicinity of~$\basin{z_0}\cap\{V<\Vstar\}$.

    \section{Proof of Theorem~\ref{thm:harm_approx}}\label{sec:proof_thm1}
    \label{sec:proof_harm_approx}
    In this section, we perform the construction of the harmonic approximation to the Witten Laplacian~\eqref{eq:global_harmonic_approximation_conj}, and give the proof of Theorem~\ref{thm:harm_approx}.
    The construction relies on the definition of local models for the Witten Laplacian~$H_\beta$ defined in~\eqref{eq:witten_laplacian}, and a family of approximate eigenmodes, or quasimodes, thereof.
    These quasimodes correspond in fact to exact eigenmodes of the harmonic approximation, or of a carefully chosen realization thereof, pointwise multiplied by a smooth cutoff function to localize the analysis.
    The harmonic approximation itself is obtained by considering a direct sum of local models around each critical point~$z_i$, which are quantum harmonic oscillators supplemented with appropriate Dirichlet boundary conditions, depending on the value of the limit~$\epsLimit{i} \in (-\infty,+\infty]$.
    
    In Section~\ref{subsec:local_oscillators}, we define formally the local models serving in the construction of the harmonic approximation, before discussing in Section~\ref{subsec:dirichlet_oscillators} their Dirichlet realization and obtaining the required spectral properties. In Section~\ref{subsec:global_harm_operator}, we define the global harmonic approximation and in Section~\ref{subsubsec:harm_quasimodes} the associated quasimodes, proving along the way crucial localization estimates.
    In Section~\ref{subsec:domain_extension}, we derive a key technical result related to the construction of an extended domain with a precise control on the shape of the boundary near critical points, and finally prove Theorem~\ref{thm:harm_approx} in Section~\ref{subsubsec:harm_approx_final_proof}.

    As the construction of the harmonic approximation involves many intermediate operators, we provide for the reader a summary of the notation used throughout this section in Table~\ref{tab:harm_notation} below.

        \begin{table}
        \makebox[\textwidth]{ 
        \hspace*{-\dimexpr\hoffset\relax} 
        \renewcommand{\arraystretch}{2.0}
        \begin{tabular}{|c|c|c|c|}
            \hline
            Operator & Spatial domain & Eigenstates & Eigenvalues\\
            \hline
            $\frac12(-\partial_x^2 + x^2)$ & $(-\theta,+\infty)$ & $v_{k,\theta}$ & $\mu_{k,\theta}$\\
            \hline
            $K_\theta^{(i)},\,d=1,\theta\in\R\cup\{+\infty\}$ & $(-\infty,\theta)$ & $w_{k,\theta}^{(i)}$,\,$k\geq 0$ & $\omega_{k,\theta}^{(i)}$,\,$k\geq 0$\\
            \hline
            $K_{\theta}^{(i)},\,d\geq 2,\theta\in\R\cup\{+\infty\}$ & $(-\infty,\theta)\times \R^{d-1}$ & $ \psi_{k,\theta}^{(i)},\,k\in\N^d$ or $\N^*$ & $\lambda_{k,\theta}^{(i)},\,k\in\N^d$ or~$\N^*$ \\
            \hline
            $H_{\beta,\theta}^{(i)},\,\theta\in\R\cup\{+\infty\}$ & $z_i + \halfSpace{i}\left(\frac{\theta}{\sqrt\beta}\right)$ & $\psi_{\beta,k,\theta}^{(i)}$,\,$k\geq 1$&$\beta\lambda_{k,\theta}^{(i)}$,\,$k\geq 1$\\
            \hline
            $\displaystyle H_{\beta,\alpha}^{\mathrm{H}},\,\alpha\in \left(\R\cup\{+\infty\}\right)^N$ & $\displaystyle\prod_{i=0}^{N-1}\left[z_i + \halfSpace{i}\left(\frac{\alpha_i}{\sqrt\beta}\right)\right]$ & $\psi_{\beta,k,\alpha}^{\mathrm{H}} = \psi_{\beta,n_k,\alpha_{i_k}}^{(i_k)},\,k\geq 1$ & $\beta\lambda_{k,\alpha}^{\mathrm{H}} = \beta \lambda_{n_k,\alpha_{i_k}}^{(i_k)},\,k\geq 1$\\
            \hline
        \end{tabular}
        }
        \caption{Notation used in Section~\ref{sec:proof_harm_approx}. Top row: one-dimensional Dirichlet oscillators from Section~\ref{subsec:dirichlet_oscillators}, see Equations~{\eqref{eq:ho_first_op}--\eqref{eq:holomorphic_eigensytem}}. Second row: rescaled local harmonic models from Section~\ref{subsec:dirichlet_oscillators} (one-dimensional case), see Equations~{\eqref{eq:local_ho_1d}--\eqref{eq:full_oscillator_domain}}. Third row: rescaled local harmonic models from Section~\ref{subsec:dirichlet_oscillators} (multi-dimensional case), see Equations~{\eqref{eq:full_ho_nd}--\eqref{eq:ki_enumeration}}. Fourth row: local harmonic models from Section~\ref{subsec:global_harm_operator}, see Equations~\eqref{eq:local_harm_domain}--\eqref{eq:harmonic_eigenmode}. Bottom row: global harmonic approximation from Section~\ref{subsec:global_harm_operator}, see Equations~{\eqref{eq:global_harmonic_approximation_conj}--\eqref{eq:harm_spectrum_enumeration}}.}
        \label{tab:harm_notation}
    \end{table}

    \subsection{Local harmonic models}\label{subsec:local_oscillators}
    The potential part of~$H_\beta = - \Delta + U_\beta $, given by~$U_\beta=\frac12\left(\beta^2\frac{|\nabla V|^2}2-\beta\Delta V\right)$ is, at dominant order in~$\beta$, composed of wells centered around the critical points of~$V$, which become steeper as~$\beta\to\infty$. 
    The purpose of the harmonic approximation is to approximate~$H_\beta$ using independent local models consisting of shifted harmonic oscillators centered around each one of these wells, with frequencies prescribed by the eigenvalues of the Hessian~$\nabla^2 V$ at~$z_i$.
    This very simple approximation is sufficient to estimate the first-order behavior of the bottom of the spectrum of~$-\cL_\beta$.

    Introduce
    \[\Sigma^{(i)} = \frac12\nabla^2 \left(\frac14|\nabla V|^2\right)(z_i)  = \frac12\left[\frac12 D^3 V \nabla V + \frac12 \left( \nabla^2 V\right)^2 \right](z_i) = \frac14 \left(\nabla^2 V\right)^2(z_i).\]
    We define local harmonic approximations to~$H_\beta$ around each critical point as
    \begin{equation}
        \label{eq:local_harmonic_approx}
        H_\beta^{(i)} = -\Delta + \beta^2 (x-z_i)^\intercal \Sigma^{(i)}(x-z_i) - \beta \frac{\Delta V(z_i)}2.
    \end{equation}

    \begin{remark}
    The operators~$H_\beta^{(i)}$ have a natural interpretation in terms of the original stochastic dynamics~\eqref{eq:overdamped_langevin}. Indeed, a direct computation shows that~$\beta^{-1}H_\beta^{(i)}$ is formally conjugate (up to an additive constant) to
    \begin{equation}
        \label{eq:generator_ou}
        -\cL_\beta^{(i)} = - \frac1\beta\Delta+x^\intercal\nabla^2 V(z_i) \nabla ,
    \end{equation}
    under the change of representation~$u\mapsto \e^{-\beta V^{(i)}/2}u(z_i+\cdot)$ to the flat~$L^2$ coordinates, where~${V^{(i)}(x)}$ is the local harmonic approximation to~$V$, namely $V^{(i)}(x)={V(z_i) + x^\intercal \frac{\nabla^2 V(z_i)}2 x}$. 
    Hence,~$\cL_\beta^{(i)}$ may be seen as the generator of a diffusion of the form~\eqref{eq:overdamped_langevin}, in which the potential has been replaced by its harmonic approximation around~$z_i$, so that the resulting stochastic process is a generalized Ornstein--Uhlenbeck process. We stress that this interpretation is merely formal, as the operator~$\cL_\beta^{(i)}$ is typically not well-behaved, since the measure~$\e^{-\beta V^{(i)}(x)}\,\d x$ is not even finite if~$i\geq N_0$, due to the presence of repulsive modes in the harmonic approximation.
    \end{remark}
    We then define the shifted harmonic oscillators:
    \begin{equation}
        \label{eq:rescaled_oscillator}
        \Ki{} = -\Delta  + x^\intercal \mathcal D^{(i)} x -\frac{\Delta V(z_i)}2,\end{equation}
    where~$\mathcal D^{(i)}$ denotes the diagonal matrix~$\mathcal D^{(i)}=\mathrm{diag}\left(\nu_j^{(i)2}/4\right)_{j=1,\dots,d}$.
    By dilation~$D_\lambda f(x) = \lambda^{d/2}f(\lambda x)$, translation~$T_b f(x) = f(x-b)$ and rotation~$\mathcal U^{(i)}$, where we recall the notation~\eqref{eq:eigvecs_unitary}, a direct computation shows that
    the Dirichlet realization of~$H_{\beta}^{(i)}$ on~$L^2(\Omega_\beta)$ is unitarily equivalent to that of~$\beta \Ki{}$ on~$L^2(\sqrt{\beta}U^{(i)\intercal}(\Omega_\beta-z_i))$:
    \begin{equation}
        \label{eq:harmonic_conjugation}
        H_{\beta}^{(i)} = T_{z_i}D_{\sqrt\beta}\mathcal U^{(i)}\left(\beta \Ki{}\right)\mathcal U^{(i)*}D_{1/\sqrt\beta}T_{-z_i}.
    \end{equation}
    This is precisely the construction performed in~\cite{CFKS87} on~$L^2(\R^d)$, up to our choice of conjugating~$H_\beta^{(i)}$ by~$\mathcal U^{(i)}$ to simplify the explicit form of tensorized eigenmodes.
    We next proceed in Section~\ref{subsec:dirichlet_oscillators} to compute the eigendecomposition for a family of self-adjoint realizations of the harmonic oscillators~$\Ki{}$, corresponding to specific Dirichlet boundary conditions in which the boundary is a hyperplane transverse to the eigendirection~$\hessEigvec{i}{1}$ whenever~$\epsLimit{i}<+\infty$. These operators in turn will serve as local approximations of~$H_\beta$ around each critical point, allowing the construction of approximate low-temperature quasimodes for~$H_\beta$.

    Before this, we recall standard results concerning the full one-dimensional harmonic oscillator, see for instance~\cite{T14}, and introduce some notation.
    The operator~$\frac12\left(-\partial_x^2+x^2\right)$ considered on~$L^2(\R)$, the canonical oscillator, which we denote by~${\mathfrak H}_{\infty}$, is self-adjoint as the Friedrichs extension of a positive quadratic form.
    We denote, for~$k\in\N$,
    \begin{equation}
        \label{eq:hermite_eigenfunction}
        v_{k,\infty}(x) = \frac{1}{\sqrt{2^k k! \sqrt\pi}}\e^{-\frac{x^2}2}H_k(x),\qquad H_k(x) = \e^{x^2}\partial_x^k\,\e^{-x^2},
    \end{equation}
    where~$H_k$ is the~$k$-th Hermite polynomial. The function~$v_{k,\infty}$ is the~$k$-th eigenstate of~${\mathfrak H}_\infty$, with
    \begin{equation}
        \label{eq:hermite_eigenproblem}
         \mathfrak{H}_{\infty}v_{k,\infty} = \mu_{k,\infty} v_{k,\infty},\qquad \mu_{k,\infty} = k + \frac12.
    \end{equation}
    The full harmonic oscillator will serve, as in~\cite[Chapter 11]{CFKS87}, as the base operator to construct local models for~$H_\beta$ associated with critical points which are far from the boundary, and will also be useful to capture the behavior of modes transverse to the first eigendirection of the Hessian for critical points which are close to the boundary.
    For the first eigendirection, however, we need to use another model, which is a harmonic oscillator with a Dirichlet boundary condition on the asymptotic hyperplane~$z_i+\partial \halfSpace{i}(\epsLimit{i}/\sqrt\beta)$.
    In fact, to handle the possibly irregular nature of~$\partial\Omega_\beta$ in the vicinity of critical points which are close to the boundary, we need to define these Dirichlet oscillators for a range of boundary conditions.

    \subsection{Dirichlet oscillators}\label{subsec:dirichlet_oscillators}
    In this section, we introduce the appropriate Dirichlet realizations for the harmonic oscillator, which serve as the basis for the construction of local models for~$H_\beta$ around critical points which are close to the boundary.
    We consider the following dense subspace of~$L^2(\R_+^*)$:
    \[\mathcal D(\widetilde{\mathfrak H_0}) = \left\{f:x^2 f\in L^2(\R_+^*),\,f\text{ is differentiable, }f'\text{ is absolutely continuous, }f''\in L^2(\R_+^*),\,f(0)=0\right\}.\]
    We first recall classical properties of the Dirichlet harmonic half-oscillator (see for instance~\cite[Chapter X.1]{RS75} or~\cite[Section 5.1.2]{BS12} for a closely related construction): the symmetric operator
   ~$$ \widetilde{\mathfrak H}_0 = \frac12(-\partial_x^2+x^2)$$
    with domain~$\mathcal D(\widetilde{\mathfrak H}_0)$ is essentially self-adjoint. Its closure, denoted by~$\mathfrak{H}_0$, has a complete family of eigenfunctions which are given explicitly by the odd states of the full harmonic oscillator:
    \begin{equation}
        \label{harmonic_half_oscillator}
        \mathfrak{H}_0 v_{2k+1,\infty} = \left(2k + \frac32\right) v_{2k+1,\infty},\qquad k\in \N.
    \end{equation}
    The family~$(\sqrt 2 v_{2k+1,\infty})_{k\geq 0}$ is an orthonormal eigenbasis for~$\mathfrak{H}_0$, where the factor~$\sqrt 2$ enforces normalization in~$L^2(\R_+)$. Furthermore,~$\mathfrak{H}_0$ has compact resolvent, with~$\mathcal D({\mathfrak{H}}_0) \subset H_0^1(\R_+)$.
    Our aim is to make precise, for~$\theta\in\R$, the spectrum of a self-adjoint realization of the canonical oscillator~$\frac12(-\partial_x^2+x^2)$ with Dirichlet boundary conditions on~$[-\theta,+\infty)$. Specifically, we show that the spectrum is well-defined, purely discrete, and depends continuously on the position of the boundary~$\theta$.
    To this end, we use analytic perturbation theory, noticing that by translation, the spectral properties of the operator
   ~\begin{equation}
    \label{eq:ho_first_op}
   \widetilde{\mathfrak H}_\theta = \frac12(-\partial_x^2+x^2),\qquad \mathcal D(\widetilde{\mathfrak H}_\theta) = T_{-\theta}\mathcal D(\widetilde{\mathfrak H}_0)
   \end{equation}
    can be deduced from those of the conjugate operator
   ~\begin{equation}
    \label{eq:perturbation_conjugate}
    T_{\theta} \widetilde{\mathfrak H}_\theta  T_{-\theta}=\frac12(-\partial_x^2+x^2) - \theta x +\frac{\theta^2}2,
   \end{equation}  
   with domain~$\mathcal D(\widetilde{\mathfrak H_0})$.
    Since the constant~$\frac{\theta^2}2$ only shifts the spectrum by an analytic function of~$\theta$, it is sufficient to study the operator
   ~$$\widetilde{\mathfrak{G}}_\theta = \widetilde{\mathfrak{H}}_0-\theta x.$$
   We show the following result.
   \begin{lemma}
    \label{lemma:spectral_analysis}
    The operator~$\widetilde{\mathfrak{G}}_\theta$ is essentially self-adjoint, and its closure~$\mathfrak{G}_\theta$ has compact resolvent. For any~$k\geq 0$, the normalized eigenpairs
    $\left(\widetilde\mu_{k,\theta},\widetilde v_{k,\theta}\right) \in \R\times L^2(\R_+)$
    are defined by
    \begin{equation}
        \label{eq:tilde_eigenrelations}
        \mathfrak{G}_\theta \widetilde v_{k,\theta} = \widetilde\mu_{k,\theta}\widetilde v_{k,\theta},\qquad \|\widetilde v_{k,\theta}\|^2_{L^2(\R_+)} =1,
    \end{equation}
    with the enumeration convention fixed by the condition that~$\widetilde v_{k,0}$ is an eigenstate of the half-harmonic oscillator:~$\widetilde v_{k,0} = \sqrt 2 v_{2k+1,\infty}$.
    Furthermore, the eigenpairs~$(\widetilde\mu_{k,\theta},\widetilde v_{k,\theta})$ can be chosen to be holomorphic functions of~$\theta$.
   \end{lemma}
   \begin{proof}
    We check that~$\theta x$ is~$\widetilde{\mathfrak H}_0$-bounded with relative bound~$0$. In the following, norms are on~$L^2(0,\infty)$. For~$\varphi \in \testfuncs([0,+\infty))$ and any~$M>0$, we compute:
    \begin{equation}
        \begin{aligned}
            \|\theta x\varphi\|^2 &= \theta^2\langle x^2\varphi,\varphi\rangle\\
            &\leq 2\theta^2 \left\langle \widetilde{\mathfrak H}_0\varphi,\varphi\right\rangle\\
            &\leq \theta^2 M^2\|\varphi\|^2 + \frac{\theta^2}{M^2}\|\widetilde{\mathfrak H}_0\varphi\|^2,
        \end{aligned}
    \end{equation}
    using Cauchy--Schwarz and Young inequalities in the last line.

    It follows that
    \begin{equation}
        \label{eq:relative_bound}
        \|\theta x \varphi\| \leq |\theta|M\|\varphi\| + \frac{2|\theta|}M \left\| \widetilde{\mathfrak H}_0 \varphi\right\|,
    \end{equation}
    so that the claim follows by taking~$M\to \infty$.

    Since~$\theta x$ is~$\widetilde{\mathfrak H}_0$-bounded with relative bound 0, by the Kato--Rellich theorem (see for instance~\cite[Theorem 6.4]{T14}), the operator~$\widetilde{\mathfrak{G}}_\theta=\widetilde{\mathfrak H}_0 - \theta x$ is essentially self-adjoint on~$\mathcal D({\widetilde{\mathfrak{H}}}_0)$, and its unique self-adjoint extension has domain~
    $\mathcal D(\mathfrak{G}_\theta) = \mathcal D({\mathfrak{H}_0})\subset H_0^1(\R_+)$ independently of~$\theta$.
    We denote by~$\mathfrak{G}_\theta$ the closure of~${\widetilde{\mathfrak{G}}}_\theta$.
    A straightforward consequence of the relative bound~\eqref{eq:relative_bound} (which extends to the closures of the operators at play) and the compactness of the resolvent of~$\mathfrak{H}_0$ is that, for fixed~$\theta\in \R$ and~$\mathrm{Im}\,\lambda \neq 0$, the resolvent~$(\mathfrak{G}_\theta-\lambda)^{-1}$ is a compact operator. Hence~${\mathfrak G}_\theta$ also has a compact resolvent, and therefore a purely discrete spectrum. Since this spectrum is manifestly bounded from below, it consists of isolated eigenvalues of finite multiplicity tending to~$+\infty$.
    Standard results of perturbation theory~(see~\cite[Chapter VII]{K95}) apply. In particular, we get from~\eqref{eq:relative_bound} and~\cite[Theorems VII.2.6 and VII.3.9]{K95} that~${\mathfrak G}_\theta$ defines a self-adjoint holomorphic family of type (A) for~$\theta\in\R$, and that there exist, for every~$k\in\N$, holomorphic functions of~$\theta$~$\widetilde\mu_{k,\theta}$,~$\widetilde v_{k,\theta}$ satisfying~\eqref{eq:tilde_eigenrelations}.
    \end{proof}

    Let us denote by~$\mathfrak{H}_\theta$ the self-adjoint operator on~$L^2([-\theta,+\infty))$ obtained by translating back and appropriately shifting the spectrum by~$\theta^2/2$ (recalling~\eqref{eq:perturbation_conjugate}):
    \begin{equation}
        \label{eq:ho_1d}
    \mathfrak{H}_\theta := T_{-\theta}{\mathfrak{G}}_\theta T_\theta + \frac{\theta^2}2.
    \end{equation}
    We compute its eigenpairs
    \begin{equation}
        \label{eq:ho_1d_eigenpairs}
    \mu_{k,\theta} = \widetilde{\mu}_{k,\theta}+\frac{\theta^2}2,\qquad v_{k,\theta} = T_{-\theta}\widetilde{v}_{k,\theta},
    \end{equation}
    satisfying the relation:
    \begin{equation}
        \label{eq:holomorphic_eigensytem}
        {\mathfrak{H}}_\theta v_{k,\theta} = \mu_{k,\theta} v_{k,\theta},
    \end{equation}
    where~$(v_{k,\theta})_{k\in \N}$ is an orthonormal basis of~$L^2([-\theta,+\infty))$. Moreover, the enumeration of these eigenpairs is fixed by the convention chosen for the harmonic half-oscillator, namely:
    \[v_{k,0} = \sqrt 2 v_{2k+1,\infty},\qquad \mu_{k,0} = 2k + \frac32,\]
    and the~eigenvalues $\mu_{k,\theta}$ depend holomorphically, hence continuously, on~$\theta$.

    We provide the following estimates on the principal eigenvalue~$\mu_{0,\theta}$ in the regimes~$\theta\to\pm\infty$, which are useful for the application of our results to asymptotic shape optimization (see Section~\ref{subsec:takeaways}).
    \begin{lemma}
        \label{lemma:principal_eigenvalue_dirichlet}
        It holds that
        \begin{equation}
            \label{eq:principal_eigenvalue_dirichlet}
            \mu_{0,\theta} = \frac{\theta^2}2 + \O(|\theta|^{2/3})\text{ in the limit }\theta\to-\infty,
        \end{equation}
        \begin{equation}
            \label{eq:principal_eigenvalue_infinity}
            \text{and }\underset{\theta\to +\infty}{\lim}\,\mu_{0,\theta}= \mu_{0,\infty}= \frac12.
        \end{equation}
    \end{lemma}
    \begin{proof}
        By unitary transformation, we equivalently estimate the principal Dirichlet eigenvalue of~$\mathfrak{H}_\theta$ on~$L^2(-\theta,+\infty)$. 
        
        Let us show~\eqref{eq:principal_eigenvalue_dirichlet}. Note first that, since~$x^2/2$ is bounded from below by~$\theta^2/2$ on~$(-\theta,+\infty)$, it clearly holds that~$\mu_{0,\theta}\geq \theta^2/2$, so there is a trivial lower bound. The proof of the estimate~\eqref{eq:principal_eigenvalue_dirichlet} given below does not rely on a variational argument, but rather on a connection with the asymptotics of Hermite polynomials.
        Let, for any~$n\geq 1$,~$\zeta_n \geq 0$ be the largest root of the~$n$-th Hermite polynomial~$H_n$, where we recall the definition~\eqref{eq:hermite_eigenfunction} of the eigenmodes for the full harmonic oscillator~$\mathfrak{H}_\infty$.
        Then, it holds that~$\mu_{0,-\zeta_n} = n + \frac12$. Indeed, the restriction of the~$n$-th eigenfunction~$v_{n,\infty}$ to the nodal domain~$(\zeta_n,+\infty)$ is a signed eigenfunction of~$\mathfrak{H}_{-\zeta_n}$, with eigenvalue~$\mu_{n,\infty}=n+\frac12$. By standard arguments, it must in fact be a principal eigenfunction. Indeed, if~$u$ minimizes the quadratic form, so does~$|u|$, hence there exists a signed principal eigenfunction. If~$v$ is an eigenfunction for some higher eigenvalue, it therefore cannot be signed without violating the orthogonality condition.

        From the domain monotonicity property of Dirichlet eigenvalues (see Proposition~\ref{prop:comparison_principle}), it holds for~$ -\zeta_{n+1}\leq \theta\leq -\zeta_{n}$, that~$ n + \frac32\geq \mu_{0,\theta}\geq n+\frac12$. Therefore,~$ n(\theta)+\frac32\geq\mu_{0,\theta}\geq n(\theta) + \frac12$, where~$n(\theta) = \max\,\{n\geq 1:\,\theta\leq -\zeta_n \}$.
        
        In~\cite[Theorem 6.32]{S39}, Szeg\"o gives the estimate~$\zeta_n = \sqrt{2n+1} + \O(n^{-1/6})$, from which we get\newline
        ~${\sqrt{2n(\theta)+1}(1+\smallo(1))\leq |\theta|}$ and~$\theta^2/C \leq n(\theta)\leq C\theta^2$ for some~$C>0$. Using Szeg\"o's estimate once again,
        \[\sqrt{2n(\theta)+1}+ \O(|\theta|^{-1/3})\leq |\theta|\leq \sqrt{2n(\theta)+3} + \O(|\theta|^{-1/3}),\]
        from which~$n(\theta) = \frac{\theta^2}2 + \O(|\theta|^{2/3})$ and the final estimate~\eqref{eq:principal_eigenvalue_dirichlet} easily follows.

        We now show~\eqref{eq:principal_eigenvalue_infinity}. The domain monotonicity principle yields the lower bound~$\mu_{0,\theta}\geq \mu_{0,\infty} = \frac12$.
        Let us show an asymptotic upper bound. We introduce~$\chi_\theta:\R\to\R$ a~$\mathcal C^\infty$ cutoff function such that
        \[\1_{(-\theta+1,+\infty)} \leq \chi_\theta\leq \1_{(-\theta,+\infty)},\]
        and denote by~$\overline{\chi}_\theta = \sqrt{1-\chi_\theta^2}$. We may choose~$\chi_\theta$ such that~$\|\partial_x \chi_\theta\|_{L^\infty(\R)},\|\partial_x\overline{\chi}_\theta\|_{L^\infty(\R)}\leq C$ for some~$C>0$ independent of~$\theta$. This may be enforced simply by setting~$\chi_\theta = \chi_0(\cdot +\theta)$ for a suitably chosen~$\chi_0$.      
        
        Consider the trial quasimode~$u_\theta=\chi_\theta v_{0,\infty}\in \mathcal D(\mathfrak{H}_\theta)$, where~$v_{0,\infty}(x) = \pi^{-1/4}\e^{-\frac{x^2}2}$ is defined in~\eqref{eq:hermite_eigenfunction}.
        It first holds that
        \[1 = \|v_{0,\infty}\|^2_{L^2(\R)} = \|u_\theta\|^2_{L^2(\R)} + \|\overline{\chi}_\theta v_{0,\infty}\|^2_{L^2(\R)},\]
        using~$\chi_\theta^2+\overline{\chi}_\theta^2 = 1$. Then, since~$\overline{\chi_\theta}\leq \1_{(-\infty,-\theta+1)}$, we obtain~$\|\overline{\chi_\theta}v_{0,\infty}\|_{L^2(\R)}^2 = \O\left(\e^{-(\theta-1)^2}\right)=\O\left(\e^{-\theta^2/2}\right)$ as~$\theta\to+\infty$ by a Gaussian tail bound. Hence,~$\|u_\theta\|_{L^2(-\theta,+\infty)}^2=\|u_\theta\|^2_{L^2(\R)} = 1 + \O(\e^{-\theta^2/2})$.

        The IMS localization formula~(see~\cite{S83,CFKS87}) then gives
        \[\frac12=\langle \mathfrak{H}_\infty v_{0,\infty},v_{0,\infty}\rangle_{L^2(\R)} = \left\langle \mathfrak{H}_\infty u_\theta,u_\theta \right\rangle_{L^2(\R)} + \left\langle \mathfrak{H}_\infty \overline{\chi}_\theta v_{0,\infty},\overline{\chi}_\theta v_{0,\infty}\right\rangle_{L^2(\R)}-\frac12\left\|\sqrt{\left(\partial_x\chi_\theta\right)^2 + \left(\partial_x\overline\chi_\theta\right)^2} v_{0,\infty}\right\|^2_{L^2(\R)}.\]
        Since~$\supp\, \partial_x\chi_\theta,\,\supp\, \partial_x \overline{\chi}_\theta \subset (-\theta,-\theta+1)$ and these derivatives are uniformly bounded in~$\theta$, it holds that~$\left\|\sqrt{\left(\partial_x\chi_\theta\right)^2 + \left(\partial_x\overline\chi_\theta\right)^2} v_{0,\infty}\right\|^2_{L^2(\R)} = \O(\e^{-\theta^2/2})$ by the same Gaussian estimate.
        Finally,
        \[\langle \mathfrak{H}_\infty \overline{\chi}_\theta v_{0,\infty},\overline{\chi}_\theta v_{0,\infty}\rangle_{L^2(\R)} = \frac12\|\partial_x(\overline{\chi_\theta}v_{0,\infty})\|^2_{L^2(\R)} + \frac12\|x\overline{\chi_\theta}v_{0,\infty}\|^2_{L^2(\R)}=\O(\e^{-\theta^2/3}),\]
        using again~$\supp\,\partial_x\overline{\chi}_\theta\subset (-\theta,-\theta+1)$,~$\overline{\chi}_\theta\leq \1_{(-\infty,-\theta+1)}$ and~$\|\partial_x\overline{\chi}_\theta\|^2_{L^\infty(\R)}\leq C$. The loss of a multiplicative constant in the exponent is due to the absorption of the~$\frac{x^2}2$ term from the potential.

        It follows that
        \[\mu_{0,\theta} \leq \frac{\langle \mathfrak{H}_\theta u_\theta,u_\theta\rangle_{L^2(-\theta,+\infty)}}{\|u_\theta\|^2_{L^2(-\theta,+\infty)}}=\frac{\langle \mathfrak{H}_\infty u_\theta,u_\theta\rangle_{L^2(\R)}}{\|u_\theta\|^2_{L^2(\R)}} = \frac{1}{2}\left(1+\O(\e^{-\theta^2/3})\right),\]
        and~\eqref{eq:principal_eigenvalue_infinity} follows upon taking the~$\lim\sup$ as~$\theta\to+\infty$ on both sides of this inequality.        
    \end{proof}
    Note that the strategies to show~\eqref{eq:principal_eigenvalue_dirichlet} and~\eqref{eq:principal_eigenvalue_infinity} can easily be adapted to treat the asymptotics of higher eigenvalues~$\mu_{k,\theta}$ of the Dirichlet oscillators~$\mathfrak{H}_\theta$, using respectively estimates on the~$k$-th largest root of the Hermite polynomials, and the trial family of quasimodes~$\left\{\chi_\theta v_{0,\infty},\dots,\chi_\theta v_{k,\infty}\right\}$.
    We now construct the local harmonic oscillators entering the harmonic approximation to the Witten Laplacian~\eqref{eq:witten_laplacian}, by considering tensorized eigenmodes of one-dimensional oscillators. 

    \paragraph{One-dimensional case.}
    In view of the change of variables~$z=\sqrt{\frac{|\nu_1^{(i)}|}2}x$ which is such that~$\frac{\nu_1^{(i)2}}4 x^2-\partial_x^2 = |\nu_1^{(i)}|\frac12\left(z^2-\partial_z^2\right)$, we denote, for~$k\geq 0$ and~$\theta\in \R\cup\{+\infty\}$, 
    \begin{equation}
        \label{eq:local_ho_1d}
    K^{(i)}_\theta = |\hessEigval{i}{1}|R D_{|\hessEigval{i}{1}/2|^{\frac12}}\mathfrak{H}_{\theta|\hessEigval{i}{1}/2|^{1/2}}D_{|\hessEigval{i}{1}/2|^{-\frac12}}R - \frac{\hessEigval{i}{1}}2,
    \end{equation}
    where~$R u(x) = u(-x)$ denotes the reflection operator, which accounts for the orientation convention chosen for~$\hessEigvec{i}{1}$. The spectrum of~$K^{(i)}_\theta$ is explicit in terms of the~$\mu_{k,\theta}$, with the following expressions for the eigenpairs:
    \begin{equation}
        \label{eq:hermite_eigenfunction_scaled}
        \wi{k}{\theta}(x) = \left(\frac{|\nu_1^{(i)}|}{2}\right)^{\frac14}v_{k,\theta(|\hessEigval{i}{1}|/2)^{\frac12}}\left(-\sqrt{\frac{|\nu^{(i)}_1|}2}x\right) ,\qquad \omegai{k}{\theta} = |\nu^{(i)}_1|\mu_{k,\theta(|\hessEigval{i}{1}|/2)^{\frac12}} - \frac{\nu^{(i)}_1}2.
    \end{equation}
    It is immediate, from the construction performed above, that~$(\wi{k}{\theta})_{k\geq 0}$ forms a complete orthonormal eigenbasis for the Dirichlet realization of the oscillator~$\Ki{\theta}$ on~$L^2((-\infty,\theta))$, hence~$\Ki{\theta}$ is self-adjoint with
        \begin{equation}
            \label{eq:full_oscillator_domain}
            \mathcal D(\Ki{\theta}) \subset H_0^1(-\infty,\theta).
        \end{equation}
        
    \paragraph{Multidimensional case.}
    The higher-dimensional case is obtained by considering a separable Schr\"odinger operator acting on the first coordinate as a one-dimensional Dirichlet oscillator, and on the~$(d-1)$ transverse coordinates as full one-dimensional (scaled) harmonic oscillators.
    More precisely, we define~$K_{\theta}^{(i)}$ as the closure of the essentially self-adjoint operator (see e.g.~\cite[Chapter X]{RS75} for background on self-adjoint extensions)
    \begin{equation}
        \label{eq:full_ho_nd}
    \left[|\hessEigval{i}{1}|R D_{|\hessEigval{i}{1}/2|^{\frac12}}\mathfrak{H}_{\theta|\hessEigval{i}{1}/2|^{1/2}}D_{|\hessEigval{i}{1}/2|^{-\frac12}}R\right]_{1} \otimes\,\mathbb{I}_{1}+\sum_{j=2}^d \left[|\hessEigval{i}{j}|D_{|\hessEigval{i}{j}/2|^{\frac12}}\mathfrak{H}_{\infty}D_{|\hessEigval{i}{j}/2|^{-\frac12}} \right]_{j}\otimes\,\mathbb{I}_{j} - \frac{\Delta V(z_i)}{2},
    \end{equation}
    where we define, given~$d$ Hilbert spaces~$(\mathcal H_j)_{j=1,\dots,d}$ and an unbounded operator~$A$ on~$\mathcal H_j$, the operator $A_{j}\otimes\mathbb{I}_{j}$ acting on~$\bigotimes_{j=1}^d \mathcal H_j$ via~${A_{j}\otimes\mathbb{I}_{j}(f_1\otimes\dotsm\otimes f_j) = f_1\otimes\dotsm\otimes f_{j-1}\otimes Af_j \otimes f_{j+1}\otimes \dotsm\otimes f_d}$.
    We naturally identify~$K_{\theta}^{(i)}$ with a self-adjoint operator on~$L^2((-\infty,\theta)\times \R^{d-1})$. Its eigendecomposition is explicit, and enumerated by~$n = (n_1,\dots,n_d)\in\N^d$:
    \begin{equation}
        \label{eq:full_harmonic_eigenstates}
        \lambdai{n}{\theta} = \omegai{n_1}{\theta} + \sum_{j=2}^{d}\left[|\nu_j^{(i)}|\mu_{n_j,\infty} - \frac{\nu_j^{(i)}}2\right],\qquad\psii{n}{\theta}(x)= \wi{n_1}{\theta}(x_1)\prod_{j=2}^d \left[\left(\frac{|\nu_j^{(i)}|}{2}\right)^{\frac14}v_{n_j,\infty}\left(\sqrt{\frac{|\nu^{(i)}_j|}2}x_j\right)\right].
    \end{equation}
    Moreover, the domain satisfies the inclusion
    \begin{equation}
        \label{eq:full_harmonic_domain}
        \mathcal D(\Ki{\theta}) \subset H_0^1\left[(-\infty,\theta)\times\R^{d-1}\right].
    \end{equation}
    It is often more convenient to enumerate the spectrum of~$\Ki{\theta}$ with integers instead. Here we slightly abuse notation, and again write the spectrum
    \begin{equation}
        \label{eq:ki_enumeration}
        \mathrm{Spec}(\Ki{\theta}) = \left(\lambdai{n}{\theta}\right)_{n\geq 1}
    \end{equation}
    in a non-decreasing sequence of eigenvalues. The indexing convention used will be clear from the context.

    When the need arises, we will consider~$\psii{n}{\theta}$ as an element of~$L^2(\R^d)$ by extending it by zero outside of~$(-\infty,\theta)\times\R^{d-1}$.
    A crucial tool in our analysis is the following pointwise decay estimate for these harmonic Dirichlet eigenmodes.
    \begin{lemma}
        For any~$n\in\N^d$,~$\theta\in(-\infty,\infty]$ and~$0\leq i< N$, there exists a constant~$C_{i,n,\theta}>0$ such that the following inequality holds for every~$x\in\R^d$:
        \begin{equation}
            \label{eq:exponential_decay}
            |\psii{n}{\theta}(x)| \leq C_{i,n,\theta}\,\e^{-\frac{|x|^2}{C_{i,n,\theta}}}.
        \end{equation}
    \end{lemma}
    \begin{proof}
        The proof relies on a probabilistic estimate obtained in~\cite{C78}, and a reflection argument.
        We consider the following anti-symmetrization of the eigenmode~$\psii{n}{\theta}$:
        \begin{equation}
            \widetilde{\psi}^{(i)}_{n,\theta}(x) = \psii{n}{\theta}(x)\1_{x_1\leq \theta} - \psii{n}{\theta}(\iota_\theta x)\1_{x_1>\theta},
        \end{equation}
        where~$\iota_\theta(x) = (2\theta-x_1,\dots,x_d)$ denotes the reflection with respect to the~$\{x_1 = \theta\}$ hyperplane. Then, it is easy to check that~$\widetilde{\psi}^{(i)}_{n,\theta}$
        is also an eigenmode (for the same eigenvalue~$\lambdai{n}{\theta}$) of the Schr\"odinger operator associated with the symmetrized potential:
        \begin{equation}
            \widetilde K^{(i)}_\theta = -\Delta + \widetilde W^{(i)} ,\quad \widetilde W^{(i)}(x) = W^{(i)}(x)\1_{x_1\leq \theta} + W^{(i)}(\iota_\theta x)\1_{x_1>\theta},
        \end{equation}
        where~$W^{(i)}(x)=x^\intercal \mathcal D^{(i)}x$ (recall that~$\mathcal D^{(i)}=\mathrm{diag}\left(\nu_j^{(i)2}/4\right)_{1\leq j\leq d} $).
        Note that there exist~$\varepsilon>0$ and a compact set~$B\subset \R^d$ such that
        \begin{equation}
            \label{eq:lemma1_minorization}
            \widetilde W^{(i)} \geq \varepsilon|x|^2,\qquad \forall\,x\in \R^d\setminus B,
        \end{equation}
        owing to the strict positivity of~$\mathcal D^{(i)}$ (since~$z_i$ is a non-degenerate critical point).
        We consider~$K^{(i)}_{\theta,\mathrm{sym}}$ to be the self-adjoint operator obtained by the Friedrichs extension of the lower-bounded quadratic form associated with~$\widetilde K^{(i)}_\theta$.
        Then, it immediately follows from~\cite[Proposition 3.1]{C78} and the lower bound~\eqref{eq:lemma1_minorization} that the pointwise estimate~\eqref{eq:exponential_decay} holds for~$\widetilde\psi^{(i)}_{n,\theta}$ and some constant~$C_{i,n,\theta}>0$.
        The proof is concluded, in view of the inequality~$|\psii{n}{\theta}(x)|\leq |\widetilde\psi^{(i)}_{n,\theta}(x)|$ for all~$x\in\R^d$.
    \end{proof}

    \subsection{Global harmonic approximation}\label{subsec:global_harm_operator}
    We now define global harmonic approximations to~$H_\beta$ defined in~\eqref{eq:witten_laplacian}. Because of the geometric flexibility afforded by Assumption~\eqref{hyp:locally_flat}, we will in fact use this harmonic approximation for a variety of Dirichlet boundary conditions. Each of these boundary conditions will be encoded with a vector of extended real numbers~$\alpha' = (\alpha'_i)_{0\leq i <N} \in (-\infty,\infty]^{N}$.
    In this context, a distinguished role is played by~$\alpha'=\alpha$, where we recall the definition~\eqref{eq:global_alpha}.
    Its components correspond to the asymptotic signed distance of each critical point to the boundary, on the scale~$\beta^{-\frac12}$, in view of Assumption~\eqref{hyp:alpha_exists}.
    
    For general~$\alpha'$, we define local oscillators~$H^{(i)}_{\beta,\alpha'_i}$ from the definition~$\Ki{\alpha'_i}$ (in the~$d$-dimensional case) using the unitary equivalence~\eqref{eq:harmonic_conjugation}.
    In particular, the domain of~$H_{\beta,\alpha'_i}^{(i)}$ is given by
    \begin{equation}
        \label{eq:local_harm_domain}
        \mathcal D(H^{(i)}_{\beta,\alpha'_i}) = T_{z_i}D_{\sqrt\beta}\mathcal U^{(i)}\mathcal D\left(K_{\alpha'_i}^{(i)}\right)\subset H_0^1\left[z_i + \halfSpace{i}\left(\frac{\alpha'_i}{\sqrt\beta}\right)\right].
    \end{equation}
    We denote, for~$n\in\N^d$,
    \begin{equation}
        \label{eq:harmonic_eigenmode}
        \psii{\beta,n}{\alpha_i'}(x) = \beta^{\frac d4}\psii{n}{\alpha_i'}\left(\sqrt\beta U^{(i)\intercal}(x-z_i)\right),\qquad\beta\lambda^{(i)}_{n,\alpha_i'}
    \end{equation}
    the eigenpair of~$H^{(i)}_{\beta,\alpha_i'}$ associated with~$\left(\psii{n}{\alpha_i'},\lambda^{(i)}_{n,\alpha_i'}\right)$ under this correspondence. 
    Notice that we introduce a prefactor~$\beta$ in the definition of the eigenvalues of~$H_{\beta,\alpha_i'}^{(i)}$, which is related to the fact that~$-\cL_\beta$ is unitarily equivalent to~$H_\beta/\beta$, see~\eqref{eq:witten_laplacian}.
    Note the~$\beta^{\frac d4}$ factor in~\eqref{eq:harmonic_eigenmode}, which accounts for~$L^2$ normalization.

    The global approximation is formed by a direct sum of these local oscillators:
    \begin{equation}
        \label{eq:global_harmonic_approximation_conj}
        H_{\beta,\alpha'}^{\mathrm H} = \bigoplus_{i=0}^{N-1} H^{(i)}_{\beta,\alpha'_i},\qquad \mathcal D(H_{\beta,\alpha'}^{\mathrm H}) = \prod_{i=0}^{N-1} \mathcal D(H^{(i)}_{\beta,\alpha'_i}),
    \end{equation}
    hence the harmonic spectrum is given by
    \begin{equation}
        \label{eq:full_harmonic_approximation_spectrum}
        \mathrm{Spec}(H_{\beta,\alpha'}^{\mathrm H}) = \left\{\beta\lambdai{n}{\alpha'_i}\right\}_{\substack{0\leq i<N\\n\in\N^d}}.
    \end{equation}
    
    Let us specify the convention we use to enumerate the various spectra at play.
    First, we enumerate the spectrum of~$\Ki{\alpha_i'}$ in non-decreasing order, according to the convention~\eqref{eq:ki_enumeration}, with corresponding eigenmodes~$(\psii{m}{\alpha_i'})_{m\geq 1}$.
    We then enumerate the full harmonic spectrum in non-decreasing order, by defining two integer-valued sequences
    \begin{equation}
        \label{eq:eigenstate_enumeration}
        (n_j)_{j\geq 1}\in \left(\N^*\right)^{\N^*},\qquad (i_j)_{j\geq 1}\in \{0,\dots,N-1\}^{\N^*},
    \end{equation}
    defined by the condition that the~$j$-th smallest eigenvalue of~$H_\beta^{\mathrm H}$, counted with multiplicity, is given by
    \begin{equation}
        \label{eq:harm_spectrum_enumeration}
        \beta\lambda_{j,\alpha'}^{\mathrm H} = \beta\lambda^{(i_j)}_{n_j,\alpha'_{i_j}},
    \end{equation}
    where we first defer to the ordering on~$\{0,\dots,N-1\}$ and then to the ordering on each~$\mathrm{Spec}(\Ki{\alpha_i'})$ to resolve ambiguities due to degenerate eigenvalues.
    We note that this choice is arbitrary, since the ordering convention plays no particular part in the analysis.
    For convenience, we also define, for each~$0\leq i <N$, the function which gives for~$n\geq 1$ the number of states with energy lower than or equal to~$\lambda_{n,\alpha'}^{\mathrm H}$ and localized around~$z_i$:
    \begin{equation}
        \label{eq:def_ni}
        \mathfrak{N}_{i}(n) = \#\{ 1 \leq j \leq n: i_j = i\}.
    \end{equation}
    To lighten the notation, we have omitted the dependence of~$\mathfrak{N}_i$,~$i_j$ and~$n_j$ in~$\alpha'$, which will be clear from the context.
    We also note the following equalities, valid by definition for any~$n\geq 1$:
    \begin{equation}
        \label{eq:enumeration_relations}
        \underset{0\leq i<N}{\max}\,\lambda^{(i)}_{\mathfrak{N}_i(n),\alpha_i'}  = \lambda_{n,\alpha'}^{\mathrm H},\qquad \underset{0\leq i<N}{\min}\,\lambda^{(i)}_{\mathfrak{N}_i(n)+1,\alpha_i'} = \lambda_{n+1,\alpha'}^{\mathrm H},\qquad \sum_{i=0}^{N-1} \mathfrak{N}_i(n) = n.
    \end{equation}

    In particular, an expression for the second eigenvalue of the harmonic approximation is available.
    \begin{remark}
    In the case where the bottom eigenvalue is associated with a local minimum~$z_0$, i.e.~$\mathfrak{N}_0(1)=1$,~$\mathfrak{N}_i(1)=0$ for~$0< i < N$, then it holds that
    \begin{equation}
    \lambda_{2,\alpha'}^{\mathrm{H}} = \min\left\{\underset{1\leq j\leq d}{\min}\,\lambda_{e_j,\alpha_0'}^{(0)},\underset{0<i<N}{\min}\,\lambda_{(0,\dots,0),\alpha_i'}^{(i)}\right\},
    \end{equation}
    where~$e_j\in\N^d$ is the~$j$-th canonical basis element.

    In the particular case~$\alpha_0'=+\infty$, this gives, using the expression~\eqref{eq:full_harmonic_eigenstates} and~$\mu_{0,\infty}=1/2$,
   
    \begin{equation}
    \label{eq:lambda_2_expr}
    \begin{aligned}
     \lambda_{2,\alpha'}^{\mathrm{H}} &= \min\left\{\underset{1\leq j\leq d}{\min}\,\hessEigval{0}{j},\underset{0<i<N}{\min}\,\omega_{0,\alpha_i'}^{(i)}+\sum_{j=2}^d |\hessEigval{i}{j}|\1_{\hessEigval{i}{j}<0}\right\}\\
     &= \min\left\{\underset{1\leq j\leq d}{\min}\,\hessEigval{0}{j},\underset{0<i<N}{\min}\,|\nu^{(i)}_1|\mu_{0,\alpha'_i(|\hessEigval{i}{1}|/2)^{\frac12}} - \frac{\nu^{(i)}_1}2+\sum_{j=2}^d |\hessEigval{i}{j}|\1_{\hessEigval{i}{j}<0}\right\},
    \end{aligned}
    \end{equation}
    using equation~\eqref{eq:hermite_eigenfunction_scaled} in the last line.
    \end{remark}
    
    Finally, we note that the analyticity of the map~$\alpha \mapsto \mu_{k,\alpha}$ obtained for all~$k\in \N$ in Lemma~\ref{lemma:spectral_analysis} implies the continuity of the mapping~$\alpha' \mapsto \lambda_{n,\alpha'}^{\mathrm H}$ for any~$n\geq 1$.

    \subsection{Construction of harmonic quasimodes and associated localization estimates}\label{subsubsec:harm_quasimodes}
    Approximate eigenmodes of~$H_\beta$ may be obtained by localizing the eigenmodes of the harmonic approximation around the corresponding critical point, in such a way that the Dirichlet boundary conditions in~$\Omega_\beta$ are met.
    We consider, for~$\theta\in\R\cup\{+\infty\}$, so-called quasimodes of the form
    \begin{equation}
        \label{eq:harm_quasimode}
        \widetilde\psi^{(i)}_{\beta,n,\theta} = \frac{\chi_\beta^{(i)}\psi_{\beta,n,\theta}^{(i)}}{\|\chi_\beta^{(i)}\psi_{\beta,n,\theta}^{(i)}\|_{L^2(\Omega_\beta)}},
    \end{equation}
    where the harmonic mode~$\psi_{\beta,n,\theta}^{(i)}$ of~$H_{\beta,\theta}^{(i)}$, defined in~\eqref{eq:harmonic_eigenmode}, is multiplied by the cutoff function~$\chi_\beta^{(i)}$ and normalized in~$L^2(\Omega_\beta)$.
    The role of~$\chi_\beta^{(i)}$ is to localize the quasimode in the vicinity of~$z_i$. To this effect, we fix a reference~$\testfuncs(\R)$ cutoff function~$\chi$ such that
    \begin{equation}
        \label{eq:reference_cutoff}
        \1_{[-\frac12,\frac12]} \leq \chi \leq \1_{[-1,1]}.
    \end{equation}
    We furthermore require that
    \begin{equation}
        \label{eq:partition_of_unity_condition}
        \left\|\frac{\d}{\d x}\left[\sqrt{1-\chi^2}\right]\right\|_{L^\infty(\R)} < +\infty.
    \end{equation}
    Let us next define the localized cutoff function:
    \begin{equation}
        \label{eq:cutoff}
        \chi_\beta^{(i)}(x) = \chi\left(\largeRadius(\beta)^{-1}|x-z_i|\right).
    \end{equation}
    Recalling that~$\sqrt\beta\largeRadius(\beta)\xrightarrow{\beta\to\infty} + \infty$, we have that~$\supp\,\chi_\beta^{(i)}$ is contained in a ball around~$z_i$ whose radius is large with respect to~$\frac1{\sqrt\beta}$, while~$\supp\,\nabla \chi_\beta^{(i)}$ is contained in a hyperspherical shell around~$z_i$:
    \begin{equation}
        \label{eq:cutoff_supports}
        \supp\,\chi_\beta^{(i)} \subset B(z_i,\largeRadius(\beta)),\qquad\supp\,\nabla\chi_\beta^{(i)} \subset B(z_i,\largeRadius(\beta))\setminus B\left(z_i,\frac12\largeRadius(\beta)\right).
    \end{equation}
    In fact we will assume in the proof of Theorem~\ref{thm:harm_approx} that~\eqref{hyp:locally_flat} is satisfied with $\delta(\beta)\ll \beta^{-\frac13}$ (this comes at no cost to generality), so that the support of~$\chi_\beta^{(i)}$ is localized around $z_i$.
    We have the bounds
    \begin{equation}
        \label{eq:linf_bound_nabla_chi}
        \left\|\partial^p\chi_\beta^{(i)}\right\|_\infty = \|\partial^p \chi\|_\infty \largeRadius(\beta)^{-|p|},
    \end{equation}
    for any multi-index~$p\in\N^d$, and thus~$\|\partial^p \chi_\beta^{(i)}\|_\infty = \smallo\left(\beta^{\frac{|p|}2}\right)$.
    
    We stress that, assuming that~$\beta$ is large enough for~\eqref{hyp:locally_flat} to hold, although~$\chi_\beta^{(i)}$ does not necessarily vanish on~$\partial \Omega_\beta$ when~$\epsLimit{i}<+\infty$, we still have~$\widetilde\psi_{\beta,k,\theta}^{(i)}\in H_0^1(\Omega_\beta)$ (the form domain of $Q_\beta$) provided~$\theta<\epsLimit{i} - \sqrt\beta\smallRadius(\beta)$ and for~$\beta$ large enough.
    
    The following result records some crucial localization estimates.
    \begin{lemma}
        \label{lemma:localization}
        We consider, for~$n\in\N^d$, eigenvectors $\psi_{\beta,n,\theta}^{(i)}$ of~$H^{(i)}_{\beta,\theta}$ normalized in~$L^2\left(z_i+\halfSpace{i}(\frac{\theta}{\sqrt\beta})\right)$, extended by~0 to~$\R^d$, and define the associated quasimodes~$\widetilde\psi_{\beta,n,\theta}^{(i)}$ according to~\eqref{eq:harm_quasimode}.
        Then, for any~$n,m\in\N^d$, there exist~$\beta_0>0$ and constants~$M_{i,n,\theta},M_{i,n,m,\theta}>0$, independent of~$\beta$, such that the following estimates hold for any~$\beta>\beta_0$.
        \begin{enumerate}[]
            \item{\begin{equation}\label{eq:loc_eqa}\left\|\left(1-\chi_\beta^{(i)}\right)\psi_{\beta,n,\theta}^{(i)}\right\|_{L^2(\R^d)} = \O\left(\e^{-\frac{\beta\largeRadius(\beta)^2}{M_{i,n,\theta}}}\right),\end{equation}}
            \item{\begin{equation}\label{eq:loc_eqb}\left|\left\langle \widetilde\psi_{\beta,n,\theta}^{(i)},\widetilde\psi_{\beta,m,\theta}^{(i)}\right\rangle_{L^2(\R^d)}-\delta_{nm}\right| =\O\left(\e^{-\frac{\beta\largeRadius(\beta)^2}{M_{i,n,m,\theta}}}\right),\end{equation}} 
            \item{\begin{equation}\label{eq:loc_eqc}\left|\left\langle H_\beta^{(i)}(1-\chi_\beta^{(i)})\psi_{\beta,n,\theta}^{(i)},(1-\chi_\beta^{(i)})\psi_{\beta,m,\theta}^{(i)}\right\rangle_{L^2(\R^d)}\right| = \O\left(\beta\e^{-\frac{2\beta\largeRadius(\beta)^2}{M_{i,n,m,\theta}}}\right).\end{equation}}
        \end{enumerate}
        If the scaling~\eqref{hyp:scaling_deltai} holds, the upper bounds decay superpolynomially in~$\beta$.
    \end{lemma}
    \begin{proof}
        We begin by proving~\eqref{eq:loc_eqa}.
        Changing coordinates with~$y = \sqrt\beta U^{(i)\intercal}(x-z_i)$, we get, in view of~\eqref{eq:harmonic_eigenmode},
        \begin{equation}
            \begin{aligned}
                \left\|(1-\chi_\beta^{(i)})\psi_{\beta,n,\theta}^{(i)}\right\|^2_{L^2(\R^d)} &= \int_{\R^d}\left(1-\chi\right)\left(\frac1{\sqrt\beta\largeRadius(\beta)}\left|U^{(i)}y\right|\right)\psi^{(i)}_{n,\theta}(y)^2\,\d y\\
                &\leq \int_{\R^d\setminus B\left(0,\frac12 \sqrt\beta\largeRadius(\beta)\right)}\psi^{(i)}_{n,\theta}(y)^2\,\d y\\
                &\leq |\mathbb S^{d-1}|C_{i,n,\theta}^2\int_{\frac12 \sqrt\beta\largeRadius(\beta)}^\infty\,s^{d-1}\e^{-\frac{2s^2}{C_{i,n,\theta}}}\,\d s\\
                &= \O\left(\int_{\frac12 \sqrt\beta\largeRadius(\beta)}^\infty\,\e^{-\frac{s^2}{C_{i,n,\theta}}}\,\d s\right),
            \end{aligned}
        \end{equation}
        where we used respectively the lower bound in~\eqref{eq:reference_cutoff} and the pointwise exponential decay estimate~\eqref{eq:exponential_decay} to obtain the first and second inequalities. Applying, for~$t>1$, a standard Gaussian tail bound
        $${\int_t^\infty \e^{-s^2/2}\,\d s \leq t^{-1}\int_t^\infty s\e^{-s^2/2}\,\d s \leq \e^{-t^2/2}}$$
        yields the desired bound~\eqref{eq:loc_eqa} for~$\beta$ sufficiently large, since~$\sqrt\beta\largeRadius(\beta)\to +\infty$, and where we set~$M_{i,n,\theta} = 2 C_{i,n,\theta}$.

        Unless otherwise specified, the norms and inner products in the remainder of the proof are on~$L^2(\R^d)$.
        To show~\eqref{eq:loc_eqb}, we first note that, since
       ~$\|\psi^{(i)}_{\beta,n,\theta}\|=1$ and~$0\leq \chi_\beta^{(i)}\leq 1$, a triangle inequality gives
        \[0\leq 1-\left\|\chi_\beta^{(i)}\psi^{(i)}_{\beta,n,\theta}\right\|\leq \left\|(1-\chi_\beta^{(i)})\psi^{(i)}_{\beta,n,\theta}\right\|,\]
        so that
        \begin{equation}
            \label{eq:lemme2_normalization}
            1\leq \frac1{\left\|\chi_\beta^{(i)}\psi^{(i)}_{\beta,n,\theta}\right\|} \leq \frac1{1-\left\|(1-\chi_\beta^{(i)})\psi^{(i)}_{\beta,n,\theta}\right\|}\leq 1 + 2\left\|(1-\chi_\beta^{(i)})\psi^{(i)}_{\beta,n,\theta}\right\|
        \end{equation}
        for~$\beta$ sufficiently large, where we use~\eqref{eq:loc_eqa} to obtain the final inequality.

        For convenience, we write~$\psi^{(i)}_{\beta,n,\theta} = \psi_n$,~$\psi^{(i)}_{\beta,m,\theta}=\psi_m$ and~$\chi_\beta^{(i)}=\chi$. Then,
        \[\left\langle \chi \psi_n,\chi \psi_m\right\rangle = \left\langle \psi_n,\psi_m\right\rangle - 2\left\langle \psi_n,(1-\chi)\psi_m\right\rangle + \left\langle\psi_n,(1-\chi)^2\psi_m\right\rangle.\]
        Thus, by Cauchy--Schwarz inequalities, and using the orthonormality of~$\psi_n$ and~$\psi_m$, we obtain:
       \[ \left|\left\langle \chi\psi_n,\chi\psi_m\right\rangle - \delta_{n,m}\right| \leq 2 \|\psi_n\|\|(1-\chi)\psi_m\| + \|\psi_n\|\|(1-\chi)^2\psi_m\| \leq 3\|(1-\chi)\psi_m\|,\]
       since~$(1-\chi)^2\leq (1-\chi)$. It follows by symmetry that
       \begin{equation}
       \label{eq:lemme2_normalization_bis}
       \left|\left\langle \chi\psi_n,\chi\psi_m\right\rangle - \delta_{n,m}\right|\leq \frac32\left(\|(1-\chi)\psi_n\|+\|(1-\chi)\psi_m\|\right).
       \end{equation}
       Denoting by~$\widetilde \psi_n = \chi\psi_n\|\chi\psi_n\|^{-1}$ and~$\widetilde \psi_m = \chi\psi_m\|\chi\psi_m\|^{-1}$, we obtain
       \[\begin{aligned}
        \left|\left\langle \widetilde \psi_n,\widetilde \psi_m\right\rangle -\delta_{nm}\right|&\leq\left|\left\langle \chi \psi_n,\chi\psi_m\right\rangle - \delta_{n,m}\right| + \left|1-\left(\left\|\chi\psi_n\right\|\left\|\chi\psi_m\right\|\right)^{-1}\right|\left|\left\langle \chi \psi_n,\chi\psi_m\right\rangle\right|,
       \end{aligned}\]
       and it follows from~\eqref{eq:lemme2_normalization},~\eqref{eq:lemme2_normalization_bis} that
       \[\left|\left\langle \widetilde \psi_n,\widetilde \psi_m\right\rangle -\delta_{nm}\right|=\O\left(\|(1-\chi)\psi_n\|+\|(1-\chi)\psi_m\|\right)\]
       in the limit~$\beta\to\infty$. The estimate~\eqref{eq:loc_eqa} then implies~\eqref{eq:loc_eqb} with~$M_{i,n,m,\theta} = \max\{M_{i,n,\theta},M_{i,m,\theta}\}$.

        For~\eqref{eq:loc_eqc}, we start with an algebraic computation for a Schrödinger operator~$H = -\Delta+V$ with domain~$\mathcal D(H)\subset H^1(\R^d)$, two eigenstates~$u,v$ with respective eigenvalues~$\lambda_u,\lambda_v$, and~$\eta\in\mathcal{C}^2(\R^d)$ with uniformly bounded derivatives.
        Using the relation
        \[H(\eta u) = \eta H u - 2\nabla \eta\cdot \nabla u - u \Delta \eta = \lambda_u \eta u - 2 \nabla \eta \cdot \nabla u  - u \Delta \eta,\]
        we get by integrating against~$\eta v$,
            \begin{align*}
                \left\langle H\eta u,\eta v\right\rangle = \lambda_u\left\langle u,\eta^2 v\right\rangle - \left\langle u,\left(\eta\Delta\eta\right)v\right\rangle -2 \left\langle \nabla u,\left(\eta\nabla\eta\right)v\right\rangle.
            \end{align*}
            Thus, by symmetry and an integration by parts,
            \begin{equation}
                \label{eq:energy_localization_formula}
                \left\langle H\eta u,\eta v\right\rangle = \left\langle u,\left(\frac{\lambda_u+\lambda_v}2\eta^2-\eta\Delta\eta\right)v\right\rangle -\frac12\left\langle \nabla[uv],\nabla[\eta^2]\right\rangle_{\mathcal D'(\R^d)\times \mathcal D(\R^d)} = \left\langle u,\left(\frac{\lambda_u+\lambda_v}2\eta^2+|\nabla\eta|^2\right)v\right\rangle.
            \end{equation}
        
        We apply~\eqref{eq:energy_localization_formula} to~$H = H_{\beta,\theta}^{(i)}$ with~$\eta =1-\chi_\beta^{(i)}$, denoting~$\lambda_u=\beta\lambda_n$ and~$\lambda_v=\beta\lambda_m$ for some~$\lambda_n,\lambda_m$ independent of~$\beta$ (recall from Table~\ref{tab:harm_notation} that the eigenvalues of~$H_{\beta,\theta}^{(i)}$ are linear in~$\beta$).
        Noting that the function~$\frac{\lambda_u+\lambda_v}2\eta^2+|\nabla\eta|^2$ is supported in~$S^{(i)}_\beta:=\R^d\setminus B\left(z_i,\frac12\largeRadius(\beta)\right)$ and using~\eqref{eq:linf_bound_nabla_chi}, there exists a constant~$C_{n,m}>0$ such that
        \[\left\|\frac{\lambda_u+\lambda_v}{2}(1-\chi_\beta^{(i)})^2+|\nabla(1-\chi_\beta^{(i)})|^2\right\|_{L^\infty(\R^d)}\leq C_{n,m}\left(\beta+\largeRadius(\beta)^{-2}\right).\]

        A Cauchy--Schwarz inequality then implies that
        \begin{equation}
            \begin{aligned}
                \left|\left\langle H_\beta^{(i)}(1-\chi_\beta^{(i)})\psi_{\beta,n,\theta}^{(i)},(1-\chi_\beta^{(i)})\psi_{\beta,m,\theta}^{(i)}\right\rangle \right| &\leq C_{n,m}\left(\beta+\largeRadius(\beta)^{-2}\right)\|\psi^{(i)}_{\beta,n,\theta}\|_{L^2\left(S^{(i)}_\beta\right)}\|\psi^{(i)}_{\beta,m,\theta}\|_{L^2\left(S^{(i)}_\beta\right)},
            \end{aligned}
        \end{equation}
        and by the same arguments as those leading to~\eqref{eq:loc_eqa}, we have
        \[\|\psi^{(i)}_{\beta,n,\theta}\|_{L^2\left(S^{(i)}_\beta\right)}\|\psi^{(i)}_{\beta,m,\theta}\|_{L^2\left(S^{(i)}_\beta\right)} = \O\left(\e^{-\frac{2\beta\largeRadius(\beta)^2}{M_{i,n,m,\theta}}}\right),\]
        which implies~\eqref{eq:loc_eqc} upon using the scaling $\largeRadius(\beta)\gg \beta^{-\frac12}$.
    \end{proof}

    The following lemma, adapting~\cite[Equations 11.5--7]{CFKS87} to the Dirichlet context, justifies the local approximation of~$H_\beta$ by~$H_\beta^{(i)}$ around~$z_i$ to the order~$\beta$.
    \begin{lemma}
        \label{lemma:taylor_bound}
        Fix~$0\leq i <N$,~$u\in L^2(\Omega_\beta)$, and assume that~$\largeRadius(\beta)<\beta^{\deltaScalingExp-\frac12}$ for some~$0<\deltaScalingExp<\frac16$ in~\eqref{hyp:locally_flat}.
        The formal operator~$H_\beta-H_{\beta}^{(i)}$ extends to a bounded operator on~$L^2(\Omega_\beta)$, and
        there exist~$C>0$ and~$\beta_0>0$ such that, for all~$\beta>\beta_0$, the following estimate holds:
        \begin{equation}
            \label{eq:taylor_bound_witten}
            \left\|(H_\beta-H_{\beta}^{(i)})\chi_\beta^{(i)} u\right\|_{L^2(\Omega_\beta)} \leq C\beta^{3\deltaScalingExp+\frac12}\|\chi_\beta^{(i)}u\|_{L^2(\Omega_\beta)} = \smallo(\beta)\|\chi_\beta^{(i)}u\|_{L^2(\Omega_\beta)}.
        \end{equation}
    \end{lemma}
    \begin{proof}
        Since~$H_\beta-H_\beta^{(i)}$ is a multiplication operator by a smooth function over a bounded domain, it is bounded in~$L^2(\Omega_\beta)$, and we therefore only need to control the~$L^\infty$-norm of the difference in the potential parts
       ~$$U_\beta - \beta^2(x-z_i)^\intercal \Sigma^{(i)}(x-z_i) + \beta \frac{\Delta V(z_i)}2$$ on~$\Omega_\beta \cap\mathrm{supp}\,\chi_\beta^{(i)} \subset B(z_i,\largeRadius(\beta))$.
        We estimate separately the two contributions
       ~$$ \beta^2\left(\frac{|\nabla V(x)|^2}4 - (x-z_i)^\intercal \Sigma^{(i)}(x-z_i) \right) - \frac\beta2(\Delta V(x) - \Delta V(z_i)).$$
        Using a second-order Taylor expansion around~$z_i$, there exist~$\beta_0>0$ and $C_1>0$ depending only on~$V$ and~$i$ such that, for all~$\beta>\beta_0$ and every~$x\in \mathrm{supp}\,\chi_\beta^{(i)}$, 
       ~$$\left|\frac{|\nabla V(x)|^2}4 - (x-z_i)^\intercal \Sigma^{(i)}(x-z_i)\right| \leq C_1 |x-z_i|^3 \leq C_1\beta^{3\deltaScalingExp-\frac32}.$$
        For the Laplacian term, since~$V$ is~$C^\infty$ and~$\Delta V$ is thus locally Lipschitz, we have, for~$\beta$ large enough, and for some constant~$C_2>0$,
       ~$$ \left| \Delta V(x) - \Delta V(z_i)\right| \leq C_2|x-z_i|\leq C_2 \beta^{\deltaScalingExp-\frac12}.$$
        By gathering these estimates and setting~$C = \max\{C_1,C_2\}$,
       ~$$\left\|(H_\beta-H_\beta^{(i)})\chi_\beta^{(i)} u\right\|_{L^2(\Omega_\beta)} \leq C\max\{\beta^{3\deltaScalingExp+\frac12},\beta^{\deltaScalingExp+\frac12}\}\|\chi_\beta^{(i)}u\|_{L^2(\Omega_\beta)},$$
        which yields the desired bound. Note that~$\beta^{3s+\frac12} = \smallo(\beta)$ since~$0<\deltaScalingExp<\frac16$.
    \end{proof}

    \subsection{Local perturbations of the boundary}
    \label{subsec:domain_extension}
    In the proofs of both the harmonic approximation (Theorem~\ref{thm:harm_approx}) and of the modified Eyring--Kramers formula (Theorem~\ref{thm:eyring_kramers}), we make use of the following technical result, which guarantees the existence of local extensions and contractions of the domains~$\Omega_\beta$, whose geometry around each critical point close to the boundary is precisely that of a half-space.
    More precisely, the existence of such an extension is used in the proof of Theorem~\ref{thm:harm_approx} to obtain a lower bound on the spectrum, using a domain monotonicity result (see~Proposition~\ref{prop:comparison_principle} below).
    In the proof of Theorem~\ref{thm:eyring_kramers}, both the extension and contraction are used to provide asymptotic bounds on the principal eigenvalue.
    Besides, the construction of approximate eigenmodes is greatly simplified on these perturbed domains.
    \begin{proposition}
        \label{prop:domain_extension}
        Let~$(\Omega_\beta)_{\beta\geq 0}$ be a family of smooth domains satisfying~\eqref{hyp:locally_flat}, and let~$\rho$ be a non-negative function such that
        \begin{equation}
            \label{eq:shift_scaling}
            \smallRadius(\beta)<\rho(\beta) = \smallo(\largeRadius(\beta)).
        \end{equation}
        Then, there exist~$\beta_0>0$, and for each~$\beta>\beta_0$, bounded, smooth, open domains~$\Omega_{\beta,\rho}^\pm$ with
        $${\Omega_{\beta,\rho}^-\subseteq \Omega_\beta\subseteq\Omega_{\beta,\rho}^+},$$ such that
        \begin{equation}
            \label{eq:inclusion_perturbed_domain}
           B\left(z_i,\frac12\largeRadius(\beta)\right) \cap \Omega_{\beta,\rho}^\pm = B\left(z_i,\frac12\largeRadius(\beta)\right) \cap \left[z_i + \halfSpace{i}\left(\frac{\epsLimit{i}}{\sqrt\beta}\pm \rho(\beta)\right)\right]
        \end{equation}
        for all~$0\leq i < N$ for which~$z_i$ is close to the boundary (i.e. $\epsLimit{i}<+\infty$).
    \end{proposition}
    \begin{proof}
        To simplify the presentation, we provide the construction assuming that~$z_i$ is the single critical point of~$V$ close to the boundary. Since the construction is local, and critical points of~$V$ are isolated because of the Morse property, our proof can be straightforwardly adapted to the case of multiple critical points.
        
        Because~$\Omega_\beta$ may be viewed as the positive superlevel set of the signed distance function, namely
        \[\Omega_\beta = \sigma_{\Omega_\beta}^{-1}(0,+\infty),\]
        a natural approach is to construct the extended domain~$\Omega_{\beta,\rho}^+$ as the positive superlevel set of a local perturbation of~$\sigma_{\Omega_\beta}$ around each~$z_i$ close to the boundary.
        This is roughly the construction we perform. However, making this precise requires some technicalities to ensure the regularity of the boundary level set.
        For visual reference, the construction of the extension~$\Omega_{\beta,\rho}^+$ is sketched in Figures~\ref{fig:ext_domain_full} and~\ref{fig:ext_domain_zoom}.

        We consider the signed distance functions~$f_\beta$ to~$\partial\Omega_\beta$,~$g_\beta$ to~$\partial B(0,\largeRadius(\beta))$, and~$h_\beta$ to the half-space~$\partial \halfSpace{i}\left(\frac{\epsLimit{i}}{\sqrt\beta}+ \rho(\beta)\right)$.
        Since~$\Omega_\beta = f_\beta^{-1}(0,+\infty)$ and $B(0,\largeRadius(\beta))=g_\beta^{-1}(0,+\infty)$ are smooth and bounded domains, there exists~$r_\beta>0$ such that~$f_\beta,g_\beta$ are~$C^\infty$ respectively on~$\{|f_\beta|<r_\beta\}$ and~$\R^d\setminus\{0\}$ (see for example~\cite[Lemma 14.16]{GT01}) whereas~$h_\beta$ as an affine map is smooth on~$\R^d$.
        
        The zero level-set of the function~$f_\beta\lor(g_\beta\land h_\beta)$ coincides with the boundary of the set
        $$\widetilde{\Omega}_{\beta,\rho}^+ :=\Omega_\beta \cup \left[ B(0,\largeRadius(\beta))\cap \halfSpace{i}\left(\frac{\epsLimit{i}}{\sqrt\beta}+ \rho(\beta)\right)\right].$$
        It satisfies~\eqref{eq:inclusion_perturbed_domain}, the inclusion~$\Omega_\beta\subseteq \widetilde{\Omega}_{\beta,\rho}^+$ and is bounded, but is generally not smooth.

        To enforce the regularity of the extended domain, we work with smooth versions of the~$\min$ and~$\max$ functions, constructed as follows. Take any~$\varepsilon>0$ and let~$a_\varepsilon^\pm\in \mathcal C^\infty(\R)$ be such that~$|x| \leq a^+_\varepsilon(x)\leq |x| + \varepsilon$, and~$|x|-\varepsilon\leq a^-_\varepsilon(x)\leq |x|$, with moreover~$a^\pm_\varepsilon(x) = |x|$ for all~$|x|\geq\varepsilon$.
        Define
        \[m^{\pm}_\varepsilon(x,y) = \frac{x+y\pm a^\pm_\varepsilon(x-y)}2,\]
        which are smooth functions on~$\R^2$ satisfying the inequalities:
        \begin{equation}
            \label{eq:smoothmax_ineq}
            x\land y \leq m^{-}_\varepsilon(x,y) \leq x\land y + \frac\varepsilon2,\qquad x\lor y \leq m^{+}_\varepsilon(x,y) \leq x\lor y + \frac\varepsilon2,
        \end{equation}
        with moreover
        \begin{equation}
        \label{eq:smoothmax_eq}
        \forall\,(x,y)\in\R^2\text{ such that }|x-y|\geq \varepsilon,\qquad m^{-}_\varepsilon(x,y) = x\land y,\qquad m^{+}_\varepsilon(x,y) = x\lor y.
    \end{equation}
        We introduce a temperature-dependent parameter~$\varepsilon_0(\beta)>0$, which will be reduced several times in the following proof. In order not to overburden the notation, we sometimes omit the dependence of~$\varepsilon_0$ on~$\beta$. We then define:
        $$\sigma_{\varepsilon_0} = m_{\varepsilon_0}^{+}(f_\beta,m_{\varepsilon_0}^{-}(g_\beta,h_\beta)).$$
        From~\eqref{eq:smoothmax_ineq}, we have by construction that:
        \begin{equation}
            \label{eq:sdf_ineq}
           f_\beta\lor (g_\beta\land h_\beta) \leq \sigma_{\varepsilon_0} \leq f_\beta\lor (g_\beta\land h_\beta) + \varepsilon_0(\beta).
        \end{equation}
        We will prove that, if~$\largeRadius(\beta)>\epsLimit{i}/\sqrt\beta + \rho(\beta)$ (which holds in the limit~$\beta\to\infty$ according to~\eqref{hyp:locally_flat} and~\eqref{eq:shift_scaling}), then for a sufficiently small~$\varepsilon_0(\beta)$, the function~$\sigma_{\varepsilon_0}$ is smooth on a small outward neighborhood~$\sigma_{\varepsilon_0}^{-1}(-\varepsilon_0(\beta),0)$ of its zero level set.
        Let us first write~$\sigma_{\varepsilon_0} = m_{\varepsilon_0}^{+}(f_\beta,\psi_{\varepsilon_0})$, where we define $\psi_{\varepsilon_0}=m_{\varepsilon_0}^{-}(g_\beta,h_\beta)$. The claimed regularity follows from the following observations.
        The function $\psi_{\varepsilon_0}$ is smooth on~$\R^d$ for~$\varepsilon_0(\beta)<\largeRadius(\beta)-\epsLimit{i}/\sqrt\beta -\rho(\beta)$: by the regularities of~$m_{\varepsilon_0}^{-}$,~$g_\beta$ and~$h_\beta$, it is enough to check that~$\psi_{\varepsilon_0}$ is smooth at~$0$. But~$h_\beta(0) = \frac{\epsLimit{i}}{\sqrt\beta}+\rho(\beta)$ and~$g_\beta(0)=\largeRadius(\beta)$, so for~$0<\varepsilon_0(\beta)<\largeRadius(\beta)-\epsLimit{i}/\sqrt\beta -\rho(\beta)$ it holds $h_\beta(0)<g_\beta(0)-\varepsilon_0(\beta)$ and~$\psi_{\varepsilon_0}$ thus coincides with~$h_\beta$ in a neighborhood of~$0$ by~\eqref{eq:smoothmax_eq}, and is therefore smooth at~$0$.

        Furthermore, the rightmost inequality in~\eqref{eq:smoothmax_ineq} implies the inclusion $${\sigma_{\varepsilon_0}^{-1}(-\varepsilon_0(\beta),0)\subset (f_\beta\lor \psi_{\varepsilon_0})^{-1}(-3\varepsilon_0(\beta)/2,0)},$$ and so it suffices to show the smoothness of~$\sigma_{\varepsilon_0}$ over this larger set. To achieve this, let~$x\in (f_\beta\lor \psi_{\varepsilon_0})^{-1}(-3\varepsilon_0(\beta)/2,0)$, and distinguish between four cases.
        \begin{itemize}
            \item{{\bf Case }~$\psi_{\varepsilon_0}(x)>f_\beta(x)+\varepsilon_0(\beta)$. Then,~$\sigma_{\varepsilon_0}$ coincides with~$\psi_{\varepsilon_0}$ in a neighborhood of~$x$, and is smooth at~$x$.}
            \item{{\bf Case}~$f_\beta(x)>\psi_{\varepsilon_0}(x)+\varepsilon_0(\beta)$. Then,~$\sigma_{\varepsilon_0}$ coincides with~$f_\beta$ in a neighborhood of~$x$, and is smooth at~$x$ provided~$3\varepsilon_0(\beta)/2<r_\beta$.}
            \item{{\bf Case}~$\psi_{\varepsilon_0}(x)\leq f_\beta(x)\leq \psi_{\varepsilon_0}(x) + \varepsilon_0(\beta)$. Then,~$f_\beta(x)=f_\beta(x)\lor \psi_{\varepsilon_0}(x) \in\left( -3\varepsilon_0(\beta)/2,0\right)$, thus~$f_\beta$ is smooth at~$x$, and likewise~$\sigma_{\varepsilon_0}$, provided~$3\varepsilon_0(\beta)/2<r_\beta$.}
            \item{{\bf Case}~$\psi_{\varepsilon_0}(x)-\varepsilon_0(\beta) \leq f_\beta(x) < \psi_{\varepsilon_0}(x)$. Then~$\psi_{\varepsilon_0}(x)\in\left(-3\varepsilon_0(\beta)/2,0\right)$ and~$f_\beta(x)\in \left(-5\varepsilon_0(\beta)/2,0\right)$. It follows that~$f_\beta$ and therefore~$\sigma_{\varepsilon_0}$ are smooth at~$x$ provided~$5\varepsilon_0(\beta)/2 < r_\beta$. }
        \end{itemize}
        We now assume here and in the following that~$\varepsilon_0(\beta) < \min\{\frac{2}{5}r_\beta,\largeRadius(\beta)-\frac{\epsLimit{i}}{\sqrt\beta}-\rho(\beta)\}$, and that~$\beta$ is large enough so that~$\varepsilon_0(\beta)$ can be chosen to be positive.
        The regularity of~$\sigma_{\varepsilon_0}$ on~$\sigma_{\varepsilon_0}^{-1}(-\varepsilon_0(\beta),0)$ allows for the construction of a smooth extended domain, as follows. By Sard's theorem~\cite{S42}, there exists~$\varepsilon_1(\beta) < \varepsilon_0(\beta)$ such that the level set~$\sigma_{\varepsilon_0}^{-1}\{-\varepsilon_1(\beta)\}$ contains no critical points of~$\sigma_{\varepsilon_0}$, and hence defines a smooth hypersurface of~$\R^d$ by the implicit function theorem.
        Therefore, the superlevel set~${\Omega_{\beta,\rho,\varepsilon_1}^{+} := \sigma_{\varepsilon_0}^{-1}(-\varepsilon_1(\beta),+\infty)}$ is a smooth open domain satisfying:
        \[(f_\beta \lor (g_\beta\land h_\beta))^{-1}(-\varepsilon_1(\beta),+\infty) \subseteq \Omega_{\beta,\rho,\varepsilon_1}^{+}\subseteq(f_\beta \lor (g_\beta\land h_\beta))^{-1}(-\varepsilon_1(\beta)-\varepsilon_0(\beta),+\infty),\]
        where we used the inequality~\eqref{eq:sdf_ineq}. It follows that~$\Omega_{\beta,\rho,\varepsilon_1}^+$ is bounded and contains~$(f_\beta \lor (g_\beta\land h_\beta))^{-1}(0,+\infty)=\widetilde{\Omega}_{\beta,\rho}^+\supseteq \Omega_\beta  $. Note that our construction (and in particular Sard's theorem) implies the existence of an appropriate~$\varepsilon_1(\beta)<\varepsilon_0(\beta)$ for any~$\varepsilon_0(\beta)>0$ sufficiently small, and so the existence of a smooth extension~$\Omega_\beta\subseteq\Omega_{\beta,\rho,\varepsilon_1}^+$ remains valid upon further reduction of~$\varepsilon_0(\beta)$.

        At this point, we will show that we can fix~$\beta_0>0$ to be sufficiently large and~$\varepsilon_0(\beta)$ sufficiently small so that the inclusion
        \begin{equation}\label{eq:strip_domain_inclusion}\O_{\beta,\varepsilon_0} := B\left(0,\largeRadius(\beta)-2\varepsilon_0(\beta)\right) \cap \halfSpace{i}\left(\frac{\epsLimit{i}}{\sqrt\beta}+\rho(\beta)+\varepsilon_0(\beta)\right)\setminus \overline{E}^{(i)} \left(\frac{\epsLimit{i}}{\sqrt\beta}+\rho(\beta)\right)\subset \left\{f_\beta+\varepsilon_0(\beta) < h_\beta < g_\beta-\varepsilon_0(\beta)\right\}
        \end{equation}
        holds for all~$\beta>\beta_0$, where~$\overline{E}^{(i)}$ denotes the closure of~$\halfSpace{i}$. The purpose of~\eqref{eq:strip_domain_inclusion} is to locate a set on which the geometry of the perturbed domain may be adjusted to enforce~\eqref{eq:inclusion_perturbed_domain}.
        Let us check that taking~$\varepsilon_0(\beta) < \frac{\rho(\beta)-\smallRadius(\beta)}2<\frac{\largeRadius(\beta)}{2}$ suffices to ensure~\eqref{eq:strip_domain_inclusion}.
        Let~$x\in\O_{\beta,\varepsilon_0}$: from~$x\in B(0,\largeRadius(\beta)-2\varepsilon_0(\beta))$, we get~$g_\beta(x)>2\varepsilon_0(\beta)$, and from~$x\in \halfSpace{i}\left(\frac{\epsLimit{i}}{\sqrt\beta}+\rho(\beta)+\varepsilon_0(\beta)\right)\setminus \overline{E}^{(i)} \left(\frac{\epsLimit{i}}{\sqrt\beta}+\rho(\beta)\right)$, we get $-\varepsilon_0(\beta)<h_\beta(x)<0$, whence~$\O_{\beta,\varepsilon_0(\beta)}\subset \{ -\varepsilon_0(\beta) < h_\beta < g_\beta-\varepsilon_0(\beta)\}$.
        From~$x\in B(0,\largeRadius(\beta)-2\varepsilon_0(\beta))\setminus\overline{E}^{(i)}\left(\frac{\epsLimit{i}}{\sqrt\beta}+\rho(\beta)\right)$, we deduce by~\eqref{hyp:locally_flat} that~$x\not\in \Omega_\beta$, with furthermore:
        $$d(x,\partial\Omega_\beta \setminus B(0,\largeRadius(\beta)))>2\varepsilon_0(\beta),$$
        $$d\left(x,\partial\Omega_\beta\cap B(0,\largeRadius(\beta))\right)>d\left(x,\partial\halfSpace{i}\left(\frac{\epsLimit{i}}{\sqrt\beta} + \smallRadius(\beta)\right)\right)>2\varepsilon_0(\beta),$$
        where the last inequality is obtained using~$\rho(\beta)-\smallRadius(\beta)>2\varepsilon_0(\beta)$ and~$x\not\in\overline{E}^{(i)}\left(\frac{\epsLimit{i}}{\sqrt\beta}+\rho(\beta)\right)$.
        Thus~$f_\beta(x)<-2\varepsilon_0(\beta)$, so that~$\O_{\beta,\varepsilon_0}\subset\{f_\beta<-2\varepsilon_0(\beta)\}$, and~$$\O_{\beta,\varepsilon_0} \subset \{f_\beta<-2\varepsilon_0(\beta)\}\cap\{-\varepsilon_0(\beta)<h_\beta<g_\beta-\varepsilon_0(\beta)\},$$ from which~\eqref{eq:strip_domain_inclusion} follows easily.

        We assume without loss of generality, upon taking~$\beta_0$ once again larger, that for all~$\beta>\beta_0$, it holds that~$\largeRadius(\beta)>\frac{\epsLimit{i}}{\sqrt\beta}+\rho(\beta)$,~$\largeRadius(\beta)/4>\rho(\beta)-\smallRadius(\beta)>0$, and~$\varepsilon_0(\beta)< \min\,\left\{\frac25 r_\beta,\largeRadius(\beta)-\frac{\epsLimit{i}}{\sqrt\beta}-\rho(\beta),\frac{\rho(\beta)-\smallRadius(\beta)}2\right\}$. 
        The previous construction can still be performed for this potentially smaller~$\varepsilon_0(\beta)$, upon accordingly updating~$\varepsilon_1(\beta)$.
        Furthermore, from the inclusion~\eqref{eq:strip_domain_inclusion} and~\eqref{eq:smoothmax_eq}, the boundary of the so-constructed set~$\Omega_{\beta,\rho,\varepsilon_1(\beta)}^+$ locally coincides with the hyperplane:
        $$\partial \Omega_{\beta,\rho,\varepsilon_1(\beta)}^{+} \cap \O_{\beta,\varepsilon_0} = \partial\halfSpace{i}\left(\frac{\epsLimit{i}}{\sqrt\beta}+\rho(\beta) + \varepsilon_1(\beta)\right)\cap \O_{\beta,\varepsilon_0}.$$

        The last step in the construction consists in indenting this hyperplane locally in~$\O_{\beta,\varepsilon_0}$, so that it coincides with~$\partial \halfSpace{i}\left(\frac{\epsLimit{i}}{\sqrt\beta}+\rho(\beta)\right)$ inside~$B\left(0,\frac12\largeRadius(\beta)\right)$.
        
        This indentation amounts to setting:
        $$\Omega_{\beta}^+ = \left[\Omega_{\beta,\rho,\varepsilon_1(\beta)}^{+}\setminus \O_{\beta,\varepsilon_0}\right] \cup \left[\hessPassage{i}\mathcal H^{(i)}_\beta \cap \O_{\beta,\varepsilon_0}\right],$$
        where~$\mathcal H^{(i)}_\beta$ is the hypograph:
        \[\mathcal H^{(i)}_\beta = \left\{(x,x')\in \R\times \R^{d-1} : x < \frac{\epsLimit{i}}{\sqrt\beta}+\rho(\beta) + \varepsilon_1(\beta)\eta(|x'|)\right\},\]     
        with~$\eta \in \mathcal C^\infty(\R)$, $0\leq \eta\leq 1$ is chosen such that:
        
        \begin{equation}
            \label{eq:indentation_cutoff}
            \begin{cases}
            \eta(|x'|)= 0,&\text{for } |x'|^2\leq \frac{\largeRadius(\beta)^2}4-\left(\frac{\epsLimit{i}}{\sqrt\beta}+\rho(\beta)+\varepsilon_1(\beta)\right)^2,\\
            \eta(|x'|) = 1, &\text{for }|x'|^2 > \left(\largeRadius(\beta)-3\varepsilon_0(\beta)\right)^2 -\left(\frac{\epsLimit{i}}{\sqrt\beta}+\rho(\beta)+\varepsilon_1(\beta)\right)^2.
        \end{cases}
        \end{equation}
        In the second line, one could possibly replace~$3\varepsilon_0(\beta)$ by~$(2+t)\varepsilon_0(\beta)$ for some other~$t>0$. This is related to the fact that~$\O_{\beta,\varepsilon_0}\subset B\left(0,\largeRadius(\beta)-2\varepsilon_0(\beta)\right)$ by~\eqref{eq:strip_domain_inclusion}.

        The requirement~\eqref{eq:indentation_cutoff} places additional constraints on~$\varepsilon_0(\beta)$ (and thus on~$\varepsilon_1(\beta))$. Namely, one must ensure that
        $\frac{\largeRadius(\beta)^2}4>\left(\frac{\epsLimit{i}}{\sqrt\beta}+\rho(\beta)+\varepsilon_1(\beta)\right)^2$, and that~$\left(\largeRadius(\beta)-3\varepsilon_0(\beta)\right)^2 >\frac{\largeRadius(\beta)^2}4$, which lead to the condition~$0<\varepsilon_1(\beta)<\varepsilon_0(\beta) < \min\left\{\frac{\largeRadius(\beta)}{2}-\frac{\epsLimit{i}}{\sqrt\beta}-\rho(\beta),\frac{\largeRadius(\beta)}{6}\right\}$.
        It is then easily checked that the first condition on~$\eta$ in~\eqref{eq:indentation_cutoff} ensures the property~\eqref{eq:inclusion_perturbed_domain}, and the second the smoothness of~$\Omega_\beta^+$.

        Therefore, choosing~$\beta_0>0$ sufficiently large, $\varepsilon_0(\beta)$ sufficiently small and~$0<\varepsilon_1(\beta)<\varepsilon_0(\beta)$ such that~$\Omega_{\beta,\varepsilon_1}^+$ is smooth, we have shown, since~$\O_{\beta,\varepsilon_0}$ is bounded and disjoint from~$\Omega_\beta$~(which follows from~\eqref{hyp:locally_flat}), the inclusions~$\Omega_\beta \subset \Omega_\beta^+$ and~\eqref{eq:inclusion_perturbed_domain}, as well as the boundedness of~$\Omega_\beta^+$.

        To construct the included domains~$\Omega_\beta^-\subseteq \Omega_\beta$, one can perform precisely the same construction, working instead on the open complement~$\R^d\setminus \overline{\Omega_\beta}$, which satisfies a symmetric version of Assumption~\eqref{hyp:locally_flat} for each~$z_i$ such that~$\epsLimit{i}<+\infty$. Denoting the resulting extension by~$\R^d\setminus\overline{\Omega_\beta}\subseteq\Omega_\beta^{-,\mathrm{c}}$, which by construction satisfies the condition

        $$B\left(z_i,\frac12\largeRadius(\beta)\right)\cap \Omega_\beta^{-,\mathrm{c}}= B\left(z_i,\frac12\largeRadius(\beta)\right)\cap\R^d\setminus\left[z_i+\overline{\halfSpace{i}}\left(\frac{\epsLimit{i}}{\sqrt\beta}-\rho(\beta)\right)\right]$$
        for each~$i$ such that~$\epsLimit{i}<+\infty$, we define
        \[\Omega_\beta^{-} := \R^d\setminus \overline{\Omega_\beta^{-,\mathrm{c}}},\]
        and indeed recover equation~\eqref{eq:inclusion_perturbed_domain} for~$\Omega_\beta^{-}\subset\Omega_\beta$, which is also clearly bounded.
    \end{proof}
    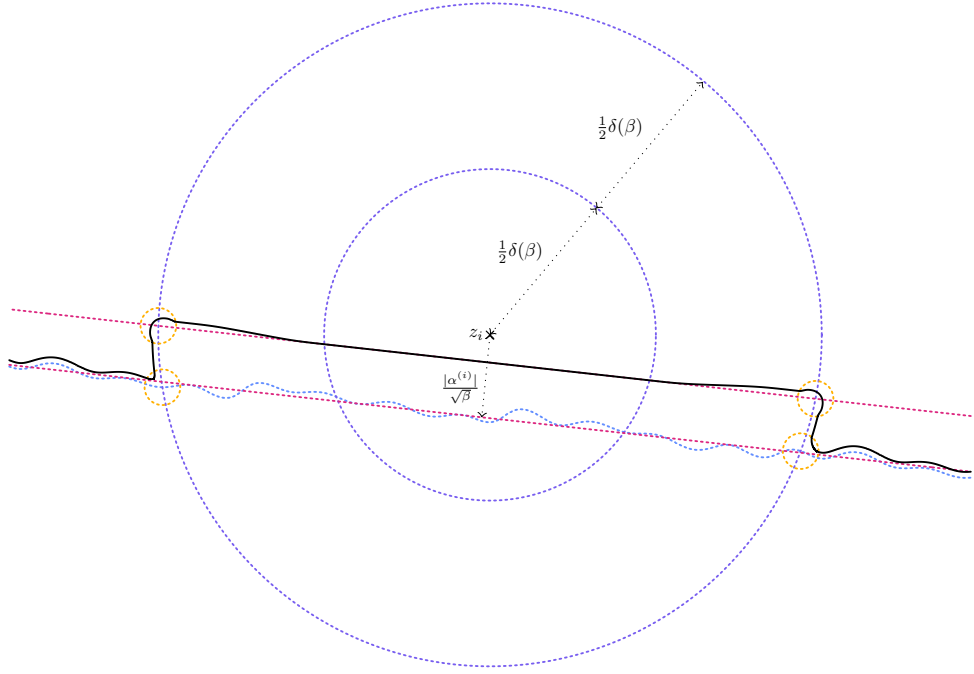
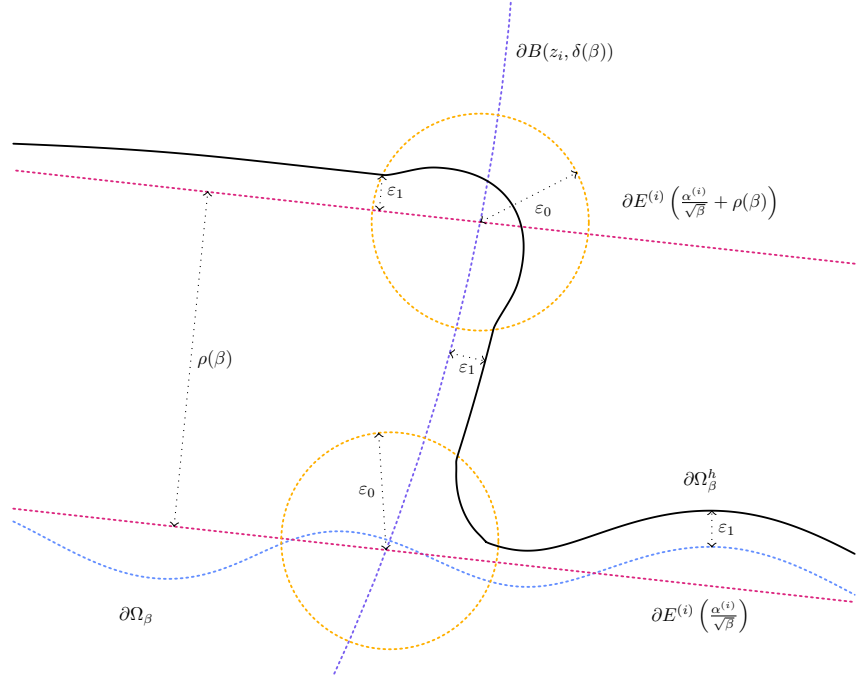
\begin{figure}
        \subfigure[Schematic representation of the boundary of the extended domain~$\partial \Omega_\beta^\shift$ in the case~$\epsLimit{i}<0$.]{
            \label{fig:ext_domain_full}
        \begin{tikzpicture}
            \node (background) at (0,0) {\includegraphics[width=\linewidth]{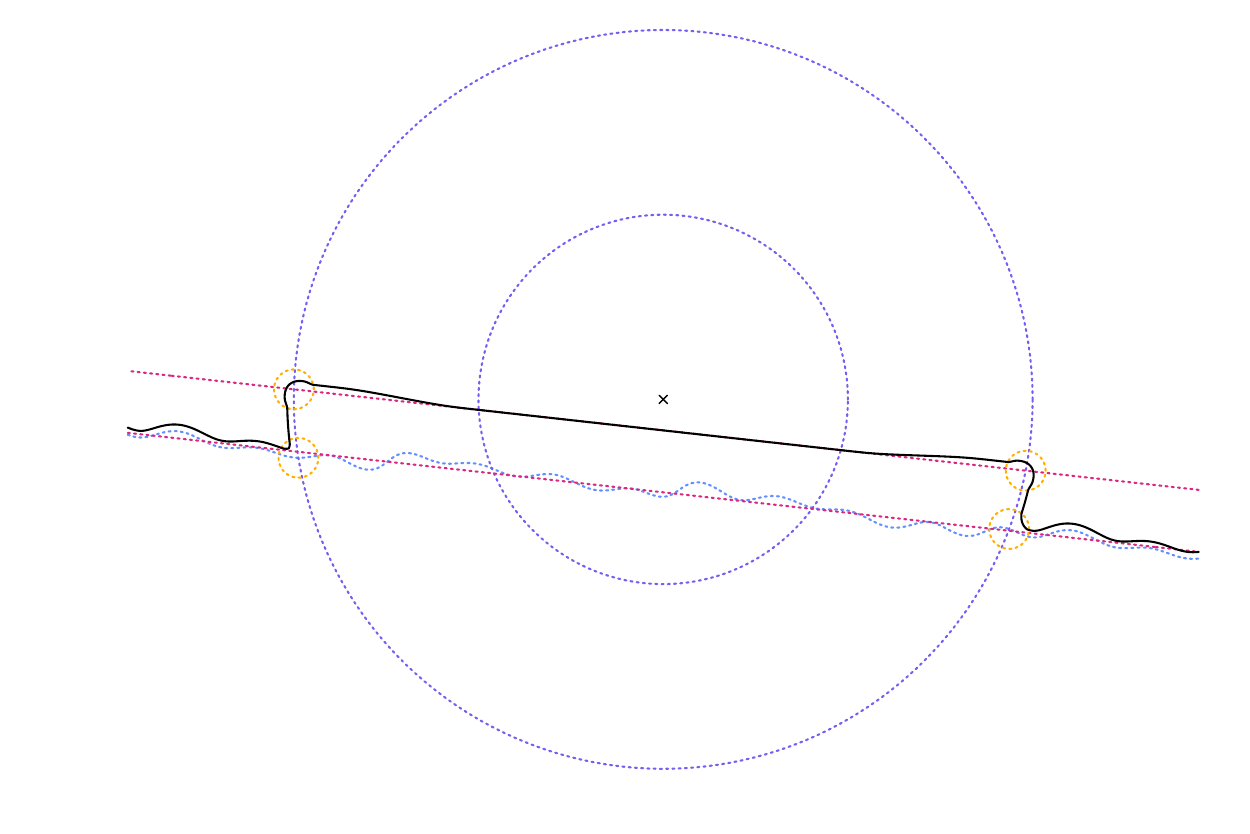}};
            \node (zi) at (0.457,0.21) {};
            \draw (zi) node[left,scale=0.7] {$z_i$};
            \draw[<->,dotted,thin] ($(zi)+(0,0)$)--($(zi)+(1.414,1.68)$);
            \draw[<->,dotted,thin] ($(zi)+(1.414,1.68)$)-- ($(zi)+(2.828,3.36)$);
            \draw[<->,dotted,thin] ($(zi)+(0,0)$)--($(zi)+(-0.11,-1.1)$);
            \draw (zi)+(0,1.1) node[right,scale=0.7] {$\frac12\largeRadius(\beta)$};
            \draw ($(zi)+(1.414,1.68)+(0.707,0.84)$) node[above left,scale=0.7] {$\frac12\largeRadius(\beta)$};
            \draw ($(zi)+(-0.05,-0.7)$) node[left,scale=0.7]{$\frac{|\epsLimit{i}|}{\sqrt\beta}$};
        \end{tikzpicture}
        }
    \subfigure[Enlarged view of the construction in the rightmost region of interest.]{
        \begin{tikzpicture}
            \label{fig:ext_domain_zoom}
            \node (background) at (0,0) {\includegraphics[width=\linewidth]{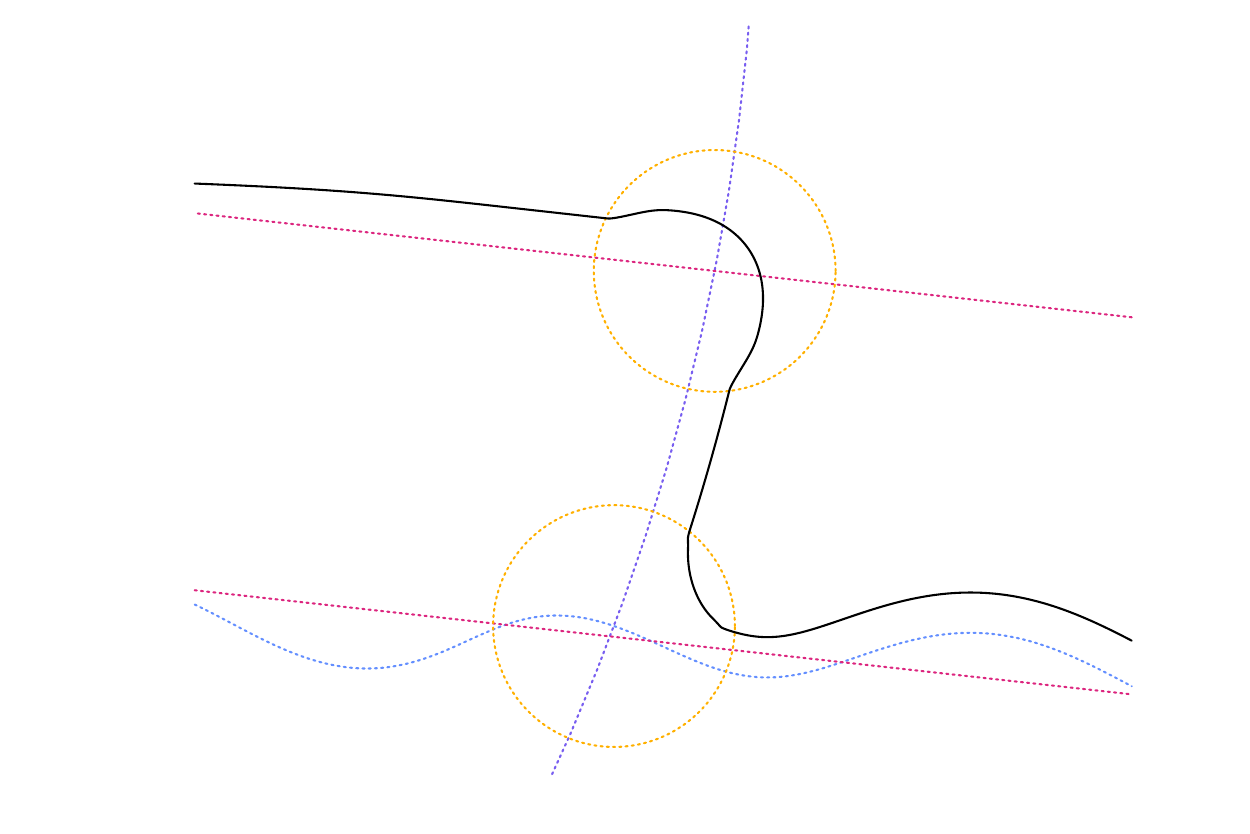}};
            \draw (-3.5,-3.5) node[scale=0.7] {$\partial \Omega_\beta$};
            \draw (4,-1.7) node[scale=0.7] {$\partial \Omega_\beta^\shift$};
            \draw (4,-3.5) node[scale=0.7] {$\partial\halfSpace{i}\left(\frac{\epsLimit{i}}{\sqrt\beta}\right)$};
            \draw (4,2) node[scale=0.7] {$\partial\halfSpace{i}\left(\frac{\epsLimit{i}}{\sqrt\beta}+\rho(\beta)\right)$};
            \draw (2.2,4.0) node[scale=0.7] {$\partial B(z_i,\largeRadius(\beta))$};
            \draw[<->,dotted,thin] (1.07,1.74)--(2.36,2.4);
            \draw ($(1.07,1.74)+(0.645,0.15)$) node[right,scale=0.7] {$\varepsilon_0$};
            \draw[<->,dotted,thin] (0.66,0)--(1.145,-0.11);

            \node (A) at (4.15,-2.59){};
            \node (B) at (-0.23,1.87){};
            \node (C) at (-3,-2.315){};
            \node (D) at (-0.18,-2.625){};

            \draw[<->,dotted,thin] ($(A)+(0,0)$)--($(A)+(0,0.497)$);
            \draw (0.902,-0.055) node[scale=0.7,below] {$\varepsilon_1$};
            \draw ($(A)+(0,0.235)$) node[scale=0.7,right] {$\varepsilon_1$};

            \draw[<->,dotted,thin] ($(B)+(-0.04,0.02)$)--($(B)+(0,0.497)$);
            \draw ($(B)+(-0.02,0.259)$) node[scale=0.7,right] {$\varepsilon_1$};

            \draw[<->,dotted,thin] ($(C)+(0,0)$)--($(C)+(0.45,4.467)$);
            \draw ($(C)+(0.225,2.234)$) node[scale=0.7,right] {$\rho(\beta)$};

            \draw[<->,dotted,thin] ($(D)+(0,0)$)--($(D)+(-0.1,1.565)$);
            \draw ($(D)+(-0.05,0.783)$) node[scale=0.7,left] {$\varepsilon_0$};

        \end{tikzpicture}
    }
        \caption{Schematic representation of the extended domain~$\Omega_\beta^+$ satisfying~\eqref{eq:inclusion_perturbed_domain}, depicted here in the vicinity of~$z_i$, a critical point close to the boundary, but outside the domain $\Omega_\beta$.}
    \end{figure}

    \subsection{Conclusion of the proof of Theorem~\ref{thm:harm_approx}}
    \label{subsubsec:harm_approx_final_proof}
    We conclude this section by giving the proof of the harmonic approximation in the case of a temperature-dependent boundary. We show that the first-order~$\beta\to\infty$ asymptotics of any given eigenvalue at the bottom of the spectrum of~$H_{\beta}$ are given by the corresponding eigenvalue of the global harmonic approximation~$H_{\beta,\alpha}^{\mathrm H}$ defined in~\eqref{eq:global_harmonic_approximation_conj}.
    The proof extends the results of~\cite{CFKS87,S83} to the case of moving Dirichlet boundary conditions, allowing the computation of first-order spectral asymptotics for eigenstates localized (in the Witten representation) around critical points which are close to the boundary.
    To show~\eqref{eq:harm_limit}, we study the eigenvalues of the Witten Laplacian~\eqref{eq:witten_laplacian}, since these are equal to those of~$-\cL_\beta$ up to a factor of~$\beta$.
    The (standard) strategy we follow is to construct approximate eigenvectors for~$H_\beta$, which (roughly speaking) consist of eigenvectors of each of the local oscillators~$H_{\beta,\epsLimit{i}}^{(i)}$, localized around the critical points~$z_i$ by appropriate cutoff functions.
    The main technical novelty compared to previous works, besides the construction of critical harmonic models performed in Section~\ref{subsec:dirichlet_oscillators}, is the technique used to ensure that the quasimodes belong to the form domain of the Dirichlet Witten Laplacian on~$\Omega_\beta$, relying namely on Proposition~\ref{prop:domain_extension}, rather than on the fact that the domain is locally diffeomorphic to a half-space, as in~\cite{HN06} and~\cite{LLPN22}.

    Once we have constructed valid quasimodes, we obtain coarse estimates similar to those of~\cite{CFKS87,S83}, at the level of the quadratic form and Courant--Fischer variational principles, allowing us to compute the limit spectrum of~$\cL_\beta$ as~$\beta\to\infty$, explicitly in terms of the spectra of the Dirichlet oscillators of Section~\ref{subsec:dirichlet_oscillators}.
    The constant~$\beta_0>0$ will be increased a finite number of times in the following proof, without changing notation.
    \begin{proof}[Proof of Theorem~\ref{thm:harm_approx}]
        Without loss of generality, we assume in this proof that~$\delta(\beta)<\beta^{s-\frac12}$ in~\eqref{hyp:locally_flat} for some~$0<s<\frac16$, so that the assumptions of Lemma~\ref{lemma:taylor_bound} are satisfied.
        \subsubsection{{Step 1: Upper bound on~$\lambda_{k,\beta}$}}
        The proof of the upper bound proceeds in two steps. First, we choose an appropriate realization of the harmonic approximation~\eqref{eq:global_harmonic_approximation_conj} so that the associated quasimodes are in the form domain~$H_0^1(\Omega_\beta)$. The second step is essentially identical to the analysis performed in~\cite[Theorem 11.1]{CFKS87}, given the localization estimates of Lemma~\ref{lemma:localization}. We include it for the sake of completeness.
        \medskip

        \noindent
        {\underline{\bf Step 1a: Perturbation of the local oscillators.}\newline}
        We fix~$k\geq 1$. We construct families of quasimodes~$\{\varphi_1,\dots,\varphi_k\}$ for~$H_\beta$ associated with the first~$k$ harmonic eigenvalues~$\beta\lambda_{1,\alpha}^{\mathrm H},\dots,\beta\lambda_{k,\alpha}^{\mathrm H}$.
        Recall Table~\ref{tab:harm_notation} for notation used in the definition of the harmonic approximation.
        The functions~$\varphi_{j}$ are of the general form~\eqref{eq:harm_quasimode}. We must however be careful with the choice of realization for the formal operator~$H_\beta^{(i_j)}$ defined in~\eqref{eq:local_harmonic_approx}, to ensure that the quasimodes are in the form domain of~$H_\beta$.
        In fact we need, for each~$0\leq i<N$, to distinguish between two cases:
        \begin{enumerate}[a)]
            \item{If~$z_i$ is far from the boundary, then by Assumption~\eqref{hyp:locally_flat},~$\chi_\beta^{(i)}$ is supported inside~$\Omega_\beta$, and thus the associated quasimodes are indeed in~$H_0^1(\Omega_\beta)$ (and even in the domain of~$\cL_\beta$).}
            \item{If~$z_i$ is close to the boundary, then by the third condition in Assumption~\eqref{hyp:locally_flat},~$\localNeighborhood[-]{i}(\beta)\subseteq B(z_i,\largeRadius(\beta)) \cap \Omega_\beta$ for $\beta$ large enough. In this case, the construction presented below yields again a quasimode in~$H_0^1(\Omega_\beta)$.}
        \end{enumerate}
        Let~$0\leq i<N$, and~$\shift>0$. Then, Assumption~\eqref{hyp:locally_flat} implies that there exists~$\beta_0>0$ such that, for all $\beta>\beta_0$, we have the inclusion:
        \[B(z_i,\largeRadius(\beta))\cap \left[z_i + \halfSpace{i}\left(\frac{\epsLimit{i}-\shift}{\sqrt\beta}\right)\right]\subseteq \localNeighborhood[-]{i}(\beta) \subseteq \Omega_\beta.\]
        Defining the vector
        \[\alpha^{\shift,-} = (\epsLimit{i} - \shift\1_{\epsLimit{i} < +\infty})_{0\leq i<N},\]
        we consider the perturbed harmonic approximation corresponding to the operator~$H_{\beta,\alpha^{\shift,-}}^{\mathrm H}$, as defined in~\eqref{eq:global_harmonic_approximation_conj}.
        By construction, for all~$\beta>\beta_0$, eigenmodes~$\psi_{\beta,n_j,\alpha^{\shift,-}_{i_j}}^{(i_j)}$ of~$H_{\beta,\alpha^{\shift,-}_{i_j}}^{(i_j)}$ are supported in
        $z_{i_j} + \halfSpace{i_j}\left(\frac{\epsLimit{i_j}-\shift}{\sqrt\beta}\right)$, 
        as noted in~\eqref{eq:local_harm_domain}, and it follows that the associated quasimode~$\widetilde\psi_{\beta,n_j,\alpha^{\shift,-}_{i_j}}^{(i_j)}$, as defined by~\eqref{eq:harm_quasimode}, belongs to~$H_0^1(\O_{i_j}^{-}(\beta)) \subset H_0^1(\Omega_\beta)$.
        \medskip

        \noindent{\underline{\bf Step 1b: Energy upper bound.}\newline}
        The aim of this step is to show the upper bound for all~$k\geq 1$:
        \begin{equation}
            \label{eq:harm_final_upper_bound}
            \underset{\beta\to\infty}{\overline\lim}\,\lambda_{k,\beta} \leq \lambda_{k,\alpha}^{\mathrm{H}}.
        \end{equation}
        For a fixed~$k\geq 1$, we consider the quasimodes
        \[\varphi_j = \widetilde\psi_{\beta,n_j,\alpha^{\shift,-}_{i_j}}^{(i_j)}\qquad \text{for }1\leq j\leq k.\]
        Note that~$(\varphi_j)_{1\leq j \leq k}$ span a~$k$-dimensional subspace of~$H_0^1(\Omega_\beta)$, since the Gram matrix~$$\left(\left\langle \varphi_j,\varphi_{j'}\right\rangle_{L^2(\Omega_\beta)}\right)_{1\leq j,j'\leq k}$$ is quasi-unitary, i.e. can be written in the form~$I + \smallo(1)$ as~$\beta\to\infty$. Indeed, since the~$\varphi_j$ are normalized in~$L^2(\Omega_\beta)$, it suffices to check that the off-diagonal entries of the Gram matrix vanish in this limit.
        Fixing~$1\leq j,j'\leq k$, this is clear for~$i_j \neq i_{j'}$, since the corresponding test functions~$\chi_\beta^{(i_j)}$,~$\chi_\beta^{(i_{j'})}$ have disjoint supports. If~$i_j = i_{j'}$, the statement is an immediate consequence of the quasi-orthogonality estimate~\eqref{eq:loc_eqb}.
        
        By the Min-Max Courant--Fischer principle, it suffices to show that
        \begin{equation}
            \label{eq:thm1_ub}
            \forall\,u\in \mathrm{Span}\{\varphi_j\}_{\,1\leq j \leq k},\qquad Q_\beta(u) \leq (\beta\lambda_{k,\alpha^{\shift,-}}^{\mathrm{H}}+\smallo(\beta))\|u\|_{L^2(\Omega_\beta)}^2,
        \end{equation}
        to conclude
        \[\underset{\beta\to\infty}{\overline\lim}\,\lambda_{k,\beta} \leq \lambda_{k,\alpha^{\shift,-}}^{\mathrm{H}}.\]
        Since the map~$\shift\mapsto\lambda_{k,\alpha^{\shift,-}}^{\mathrm{H}}$ is continuous and~${\alpha^{\shift,-} \xrightarrow{\shift\to 0} \alpha}$, 
        and therefore~$\lambda_{k,\alpha^{\shift,-}}^{\mathrm{H}}\xrightarrow{\shift\to 0} \lambda_{k,\alpha}^{\mathrm H}$, the latter inequality implies~\eqref{eq:harm_final_upper_bound}.

        Let us prove~\eqref{eq:thm1_ub}. Unless otherwise specified, in Step 1, norms and inner products are on~$L^2(\Omega_\beta)$.
        For a fixed~$1\leq j\leq k$, first consider~$u = \chi_\beta^{(i_j)} \psi_{\beta,n_j,\alpha^{\shift,-}}^{(i_j)} \in \mathrm{Span}\{\varphi_j\}$. Recall that~$H_{\beta,\alpha_{i_j}^{\shift,-}}^{(i_j)}\psi_{\beta,n_j,\alpha^{\shift,-}}^{(i_j)} = \beta \lambda_{n_j,\alpha^{\shift,-}}^{(i_j)}\psi_{\beta,n_j,\alpha^{\shift,-}}^{(i_j)}$. For convenience, we drop the indices and superscripts in the following computation, so that~${u=\chi\psi}$, with~$H \psi = \beta\lambda\psi$.
        By our choice of~$(\varphi_j)_{1\leq j\leq k}$, we have~${\lambda \in \left\{\lambda_{j,\alpha^{\shift,-}}^{\mathrm H}\right\}_{1\leq j\leq k}}$, and so~$\lambda\leq \lambda_{k,\alpha^{\shift,-}}^{\mathrm H}$.
        \[Q_\beta(u)=\left\langle H_\beta u,u\right\rangle= \left\langle H u,u\right\rangle + \left\langle (H_\beta - H)u,u\right\rangle,\]
        and since~$\left|\left\langle (H_\beta-H)u,u\right\rangle\right| \leq C_{j,\alpha,\shift}\beta^{3\deltaScalingExp+\frac12}\|u\|^2 = \smallo(\beta)\|u\|^2$ for some constant~$C_{j,\alpha,\shift}>0$ by Lemma~\ref{lemma:taylor_bound}, the only remaining task is to estimate the first term.
        Expanding the quadratic form, we get:
        \begin{equation}\begin{aligned}
            \label{eq:quad_form_expansion}
            \left\langle H\chi\psi,\chi\psi\right\rangle &= \left\langle H\psi,\psi\right\rangle - 2\left\langle H\psi,(1-\chi)\psi\right\rangle + \left\langle H(1-\chi)\psi,(1-\chi)\psi\right\rangle\\
            & = \beta\lambda\|\psi\|^2 - 2\left\langle H\psi,(1-\chi)\psi\right\rangle + \left\langle H(1-\chi)\psi,(1-\chi)\psi\right\rangle\\
            &\leq \beta\lambda\|u\|^2 + \O(\beta\|(1-\chi)\psi\|) + \left\langle H(1-\chi)\psi,(1-\chi)\psi\right\rangle,\\
        \end{aligned}
    \end{equation}
        where we used a Cauchy--Schwarz inequality and the estimate (see~\eqref{eq:loc_eqa})
        \[\|\psi\|^2 = \|u\|^2 + 2\left\langle u,(1-\chi)\psi\right\rangle + \|(1-\chi)\psi\|^2 = \|u\|^2 + \O(\|(1-\chi)\psi\|).\]
        We next use the localization estimates given in Lemma~\ref{lemma:localization} to control the two rightmost terms in the last line of~\eqref{eq:quad_form_expansion}. Here, we make crucial use of the hypothesis~\eqref{hyp:scaling_deltai},
        which gives the superpolynomial decay of the bound~\eqref{eq:loc_eqa}, implying that~$\beta\|(1-\chi)\psi\|=\smallo(\beta)$, and similarly~$\left\langle H(1-\chi)\psi,(1-\chi)\psi\right\rangle = \smallo(\beta)$, using~\eqref{eq:loc_eqc}.
        Finally,~\eqref{eq:loc_eqb} implies that~$\|u\|^2 = 1 + \smallo(1)$, so that we may write~$Q_\beta(u)\leq \|u\|^2(\beta\lambda + \smallo(\beta))$.
        Since~$\lambda\leq \lambda_{k,\alpha^{\shift,-}}^{\mathrm H}$, this implies the upper bound~\eqref{eq:thm1_ub} for this particular choice of~$u$, and thus for any~$u\in \mathrm{Span}\{\varphi_j\}$, as we now show.

        For~$u = \sum_{j=1}^k g_j \varphi_j \in \mathrm{Span}\{\varphi_j\}_{1\leq j\leq k}$, in view of the previous estimate, it is enough to show that the cross terms
        \[g_jg_{j'}\left\langle H_{\beta,\alpha_{i_j}^{\shift,-}}^{(i_j)}\varphi_j,\varphi_{j'}\right\rangle\]
        are small whenever~$i_j = i_{j'}$ with~$j\neq j'$, since the terms for which~$i_j \neq i_{j'}$ vanish for~$\beta$ large enough as the corresponding cutoff functions have disjoint supports.
        To check that the non-zero terms decay superpolynomially in~$\beta$, we denote, for convenience~$H_{\beta,\alpha_{i_j}^{\shift,-}}^{(i_j)}=H$,~$\chi = \chi_\beta^{(i_j)}$, and~$\psi = \psi^{(i_j)}_{\beta,n_j,\alpha^{\shift,-}}$, $\psi' = \psi^{(i_{j'})}_{\beta,n_{j'},\alpha^{\shift,-}}$, so that~$\varphi_j = \chi \psi/Z$,~$\varphi_{j'} = \chi\psi'/Z'$, where~$Z,Z'$ are normalizing constants ensuring that the quasimodes have unit~$L^2(\Omega_\beta)$-norm. With this notation,
        \[\begin{aligned}
            \left\langle H\chi\psi,\chi\psi'\right\rangle &= \left\langle H\psi,\psi'\right\rangle - \left\langle H\psi,(1-\chi)\psi'\right\rangle - \left\langle H(1-\chi)\psi,\psi'\right\rangle + \left\langle H(1-\chi)\psi,(1-\chi)\psi'\right\rangle\\
            &= \left\langle H\psi,\psi'\right\rangle - \beta\lambda\left\langle \psi,(1-\chi)\psi'\right\rangle - \beta\lambda'\left\langle (1-\chi)\psi,\psi'\right\rangle + \left\langle H(1-\chi)\psi,(1-\chi)\psi'\right\rangle\\
            &= 0 + \smallo(\beta)\|\psi\|\|\psi'\|,
        \end{aligned}
        \]
        where we used again Cauchy--Schwarz inequalities and the superpolynomial decay of the estimates~\eqref{eq:loc_eqa} and~\eqref{eq:loc_eqc} under~\eqref{hyp:scaling_deltai}, as well as the orthogonality relation~$\left\langle \psi,\psi'\right\rangle = 0$.
        Since~$Z,Z' = 1 +\smallo(1)$ by~\eqref{eq:loc_eqb}, it follows that~\eqref{eq:thm1_ub} holds, which concludes the proof of~\eqref{eq:harm_final_upper_bound}.

        \subsubsection{{Step 2: Lower bound on~$\lambda_{k,\beta}$}}
        We show the lower bound in~\eqref{eq:harm_limit}, namely that for all~$k\geq 1$,
        \begin{equation}
            \label{eq:harm_final_lower_bound}
            \underset{\beta\to\infty}{\underline\lim}\,\lambda_{k,\beta} \geq \lambda_{k,\alpha}^{\mathrm{H}}.
        \end{equation}
        As in the proof of the upper bound, we proceed in two steps.
        The first step is to construct an appropriate extension of the domain~$\Omega_\beta$ around critical points which are close to the boundary, relying on Proposition~\ref{prop:domain_extension}.
        The second step is similar to what is done in related proofs of~\cite{S83,CFKS87} in the boundaryless case, and we include it for completeness.
        \medskip

        \noindent{\underline{\bf Step 2a: Perturbation of the domain.}\newline}
        Our analysis requires us to consider, in addition to perturbed harmonic eigenvalues, a perturbed domain~$\Omega_\beta^\shift$ which contains~$\Omega_\beta$, parametrized by some small~$\shift>0$.
        We construct~$\Omega_\beta^\shift$ to be smooth and bounded, so that the Dirichlet realization of~$H_\beta$ in~$\Omega_\beta^\shift$ is self-adjoint on~$L^2(\Omega_\beta^\shift)$, with compact resolvent and form domain~$H_0^1(\Omega_\beta^\shift)$.
        This domain is constructed by a direct application of Proposition~\ref{prop:domain_extension} with $\rho(\beta)=\frac{\shift}{\sqrt\beta}$, that is, $$\Omega_\beta^\shift = \Omega_{\beta,h/\sqrt\beta}^{+}.$$ Crucially, it satisfies the inclusion~\eqref{eq:inclusion_perturbed_domain}, which corresponds to the requirement that, locally around each~$z_i$ close to the boundary,~$\partial\Omega_\beta^\shift$ coincides precisely with a hyperplane located at a distance~$\beta^{-\frac12}(\epsLimit{i}+\shift)$ away from the saddle in the direction~$\hessEigvec{i}{1}$, and normal to the same vector.
        In the remainder of the proof, norms and inner products are by default on~$L^2(\Omega_\beta^\shift)$.
        We denote by~$\lambda^\shift_{k,\beta}:=\lambda_{k,\beta}(\Omega_\beta^\shift)$ its~$k$-th eigenvalue, which is positive with finite multiplicity, and by~$H_\beta^\shift$ the Dirichlet realization of~$H_\beta$ on~$\Omega_\beta^\shift$, with~$Q_\beta^\shift$ its associated quadratic form.
        By the comparison principle given in Proposition~\ref{prop:comparison_principle}, it holds that $\lambda_{k,\beta}(\Omega_\beta)\geq \lambda_{k,\beta}(\Omega_\beta^\shift)$, and so it is sufficient to prove~\eqref{eq:harm_final_lower_bound} with~$\lambda_{k,\beta}=\lambda_{k,\beta}^\shift$.
        \medskip

        \noindent{\underline{\bf Step 2b: Energy lower bound.}\newline}
        From now on, $\alpha^{\shift,+}$ denotes the vector~$(\epsLimit{i}+\shift\1_{\epsLimit{i}<+\infty})_{0\leq i<N}$ and~$\varphi_{j}$, for~$1\leq j \leq k-1$, are $k-1$ harmonic quasimodes, associated with the Dirichlet harmonic approximation~$H^{\mathrm{H}}_{\beta,\alpha^{\shift,+}}$.
        However, due to the construction performed in the previous step, one must slightly adjust the definition~\eqref{eq:harm_quasimode}, and more precisely the support of the cutoff function.
        For this, we set
        \[\varphi_j = \frac{\eta_\beta^{(i_j)}\psi_{\beta,n_j,\alpha^{\shift,+}_{i_j}}^{(i_j)}}{\|\eta_\beta^{(i_j)}\psi_{\beta,n_j,\alpha^{\shift,+}_{i_j}}^{(i_j)}\|},\qquad \forall\,1\leq j\leq k-1,\]
        where~$\eta_\beta^{(i)}(x-z_i) = \chi\left(2|x-z_i|/\delta(\beta)\right)$.
        By construction, each~$\varphi_j$ is supported in~$B(z_{i_j},\frac12\largeRadius(\beta))$, and thus belongs to~$H_0^1(\Omega_\beta^\shift)$, by the inclusion~\eqref{eq:inclusion_perturbed_domain} and the fact that~$\psi_{\beta,n_j,\alpha^{\shift,+}_{i_j}}^{(i_j)}$ has support in~$z_{i_j}+\halfSpace{i_j}\left(\frac{\epsLimit{i_j}+\shift}{\sqrt\beta}\right)$ whenever~$\epsLimit{i_j}<+\infty$.
        We note that this choice of quasimodes amounts to rescaling~$\largeRadius(\beta)$ by a factor~$\frac12$, and consequently has no impact on the superpolynomial decay of the estimates of Lemma~\ref{lemma:localization}, nor on the conclusion of Lemma~\ref{lemma:taylor_bound}, which we will use freely in the remainder of the proof, upon replacing~$\chi_\beta^{(i)}$ by~$\eta_\beta^{(i)}$,~$\Omega_\beta$ by~$\Omega_\beta^\shift$ and~$H_\beta$ by~$H_\beta^\shift$.
        As in Step 1b, for~$\beta$ large enough, the~$\varphi_j$ span a $(k-1)$-dimensional subspace of~$L^2(\Omega_\beta^\shift)$.
      
        This time, we rely on the Courant--Fischer principle in its Max-Min form: it suffices to show that, for any~$u \in H_0^1(\Omega_\beta^\shift)\cap\mathrm{Span}\{\varphi_j\}_{ 1\leq j \leq k-1}^\perp$, the following inequality holds:
       ~$$ Q^\shift_\beta(u) \geq (\beta\lambda_{k,\alpha^{\shift,+}}^{\mathrm{H}} + \smallo(\beta))\|u\|^2.$$
        This will indeed imply:
        \[\underset{\beta\to\infty}{\underline\lim}\,\lambda_{k,\beta}\geq\underset{\beta\to\infty}{\underline\lim}\,\lambda^\shift_{k,\beta}\geq \lambda_{k,\alpha^{\shift,+}}^{\mathrm{H}},\]
        and the desired lower bound~\eqref{eq:harm_final_lower_bound} follows by continuity of the harmonic eigenvalues with respect to the boundary position~$\alpha^{\shift,+}$, taking the limit~$\shift\to 0$.
        Hence, let~$u\in H_0^1(\Omega_\beta^\shift)$ be orthogonal to~$\varphi_j$ for every~$1\leq j \leq k-1$. The IMS localization formula (see for instance~\cite[Chapter 3]{CFKS87}) gives:
        \begin{equation}
        \label{eq:ims_formula}
        Q^\shift_\beta(u) = \sum_{i=0}^{N} \left[Q^\shift_\beta\left(\eta_\beta^{(i)}u\right) - \left\|u\nabla \eta_\beta^{(i)}\right\|^2\right],
        \end{equation}
        with~$\eta_\beta^{(N)} = \sqrt{\mathbbm{1}_{\Omega_\beta^\shift}-\sum_{i=0}^{N-1}\eta_\beta^{(i)2}}.$
        We first estimate, for~$0\leq i < N$, the terms
        \begin{equation}
            Q_\beta^\shift\left(\eta_\beta^{(i)}u\right)-\left\langle H_{\beta,\alpha^{h,+}_i}^{(i)}\eta_\beta^{(i)}u,\eta_\beta^{(i)}u\right\rangle = \left\langle (H_\beta^\shift-H_{\beta,\alpha_i^{h,+}}^{(i)})\eta_\beta^{(i)}u,\eta_\beta^{(i)}u\right\rangle = \smallo(\beta)\|u\|^2,
        \end{equation}
        using Lemma~\ref{lemma:taylor_bound}.
        On the other hand, the assumption that~$u$ is $L^2(\Omega_\beta^\shift)$-orthogonal to the~$\varphi_j$ for~$1\leq j \leq k-1$ implies, for all $0\leq i<N$ and all $1\leq n\leq\mathfrak N_i(k-1)$, that~$\eta_\beta^{(i)}u$ is $L^2\left[z_{i} + \halfSpace{i}\left(\frac{\epsLimit{i}+\shift}{\sqrt\beta}\right)\right]$-orthogonal to~$\psi_{\beta,n,\alpha_{i}^{\shift,+}}^{(i)}$.
        Here, we made crucial use of the property~\eqref{eq:inclusion_perturbed_domain} of the extended domain.
        Furthermore,~$\eta_\beta^{(i)}u$ is in the form domain of the self-adjoint operator~$H_{\beta,\alpha^{\shift,+}_{i}}^{(i)}$, and thus the Courant--Fischer principle implies:
        \begin{equation}
            \left\langle H_{\beta,\alpha^{\shift,+}_{i}}^{(i)}\eta_\beta^{(i)}u,\eta_\beta^{(i)}u\right\rangle \geq \beta\lambda_{\mathfrak N_i(k-1)+1,\alpha^{\shift,+}_{i}}^{(i)}\|\eta_\beta^{(i)}u\|^2\geq \beta \lambda_{k,\alpha^{\shift,+}}^{\mathrm H}\|\eta_\beta^{(i)}u\|^2,
        \end{equation}
        where we used the identity~\eqref{eq:enumeration_relations}.
        The crude~$L^\infty$ bound~\eqref{eq:linf_bound_nabla_chi} implies,~since~$\largeRadius(\beta)^{-1} = \smallo\left(\sqrt\beta\right)$, that~$\left\|u\nabla\eta_\beta^{(i)}\right\|^2 = \smallo(\beta)\|u\|^2$.
        At this point, we have shown that 
        \[\sum_{i=0}^{N-1} Q^\shift_\beta\left(\eta_\beta^{(i)}u\right) - \left\|u\nabla \eta_\beta^{(i)}\right\|^2\geq \beta\lambda_{k,\alpha^{\shift,+}}^{\mathrm{H}}\sum_{i=0}^{N-1}\|\eta_\beta^{(i)}u\|^2 +\smallo(\beta)\|u\|^2.\]

        We are left with the following term in~\eqref{eq:ims_formula}:
        \[\left\langle H_\beta^\shift \eta_\beta^{(N)}u,\eta_\beta^{(N)}u\right\rangle - \left\|u\nabla \eta_\beta^{(N)}\right\|^2.\]
        Note that, since
       ~$$ \supp\,\eta_\beta^{(N)} \subset \R^d\setminus \bigcup_{i=0}^{N-1} B\left(z_i,\frac12\largeRadius(\beta)\right),$$
        the fact that~$V$ is a smooth Morse function in~$ \Omega_\beta^\shift$ ensures that there exist~$C,\beta_0>0$ such that, for all~$\beta>\beta_0$,
        \[|\nabla V|^2_{\big|\supp\,\eta_\beta^{(N)}} \geq \frac1C\largeRadius(\beta)^2,\qquad |\Delta V|_{\big|\supp\,\eta_\beta^{(N)}} \leq C,\]
        hence:
        \[\forall\,x\in\supp\,\eta_\beta^{(N)},\qquad U_\beta(x)=\frac{\beta^2}4|\nabla V(x)|^2 -\frac\beta 2 \Delta V(x) \geq  \frac{\beta^2 \largeRadius(\beta)^2}{4C}-\beta C.\]
        Since~$-\Delta\geq 0$, we have, for~$\beta>\beta_0$ and any $v\in H_0^1(\supp\,\eta_\beta^{(N)})$:
        \[\left\langle H_\beta^\shift v,v\right\rangle_{L^2(\supp\,\eta_\beta^{(N)})} \geq \left\langle U_\beta v,v\right\rangle_{L^2(\supp\,\eta_\beta^{(N)})} \geq \beta\left(\frac{\beta\largeRadius(\beta)^2}{4C}-C\right)\|v\|_{L^2(\supp\,\eta_\beta^{(N)})}^2.\]
        Since~$\sqrt\beta\largeRadius(\beta)\xrightarrow{\beta\to\infty}+\infty$, we conclude that, for~$\beta$ large enough, and since~$\eta_\beta^{(N)}u$ belongs to~$H_0^1(\supp\,\eta_\beta^{(N)})$, we have:
        \[\left\langle H_\beta^\shift \eta_\beta^{(N)}u,\eta_\beta^{(N)}u\right\rangle \geq \beta\lambda_{k,\alpha^{\shift,+}}^{\mathrm{H}}\left\|\eta_\beta^{(N)}u\right\|^2.\]
        Finally, we note that, owing to the condition~\eqref{eq:partition_of_unity_condition}, and the fact that the~$\eta_\beta^{(i)}$ have disjoint supports for~$i=0,\dots,N-1$, it also holds:
       ~$$\|u\nabla \eta_\beta^{(N)}\|^2 = \smallo(\beta)\|u\|^2.$$

        In conclusion, we have shown that
        \begin{equation}
            \left\langle H_\beta^h u,u \right\rangle \geq \beta\lambda_{k,\alpha^{\shift,+}}^{\mathrm H} \sum_{i=0}^{N}\|\eta_\beta^{(i)}u\|_{L^2(\Omega_\beta^\shift)}^2 + \smallo(\beta)\|u\|_{L^2(\Omega_\beta^\shift)}^2.
        \end{equation}
        The lower bound~\eqref{eq:harm_final_lower_bound} follows easily since the~$\eta_\beta^{(i)}$ form a quadratic partition of unity on~$\Omega_\beta^\shift$, that is
       ~$$\forall\,v\in L^2(\Omega_\beta^\shift),\quad\|v\|_{L^2(\Omega_\beta^\shift)}^2 = \sum_{i=0}^{N} \|\eta_\beta^{(i)}v\|_{L^2(\Omega_\beta^\shift)}^2.$$
    \end{proof}

    \section{Proof of Theorem~\ref{thm:eyring_kramers}}
    \label{sec:proof_ek}
    In this section, we derive finer asymptotics for the first Dirichlet eigenvalue~$\lambda_{1,\beta}$ of~$-\cL_\beta$, in the regime~$\beta\to\infty$, as described by Theorem~\ref{thm:eyring_kramers}.
    The proof of~\eqref{eq:eyring_kramers}, inspired by a construction performed in~\cite{LPN21} in the case of a static domain, relies on the definition of accurate approximations of the first Dirichlet eigenvector for~$-\cL_\beta$, so-called quasimodes.
    In Section~\ref{subsec:local_energy_estimates}, we derive the necessary local estimates to ensure the well-posedness of the construction, in the presence of a temperature-dependent boundary. In Section~\ref{subsec:ek_quasimodes}, we perform the construction of precise quasimodes, on slightly perturbed domains chosen so that the Dirichlet boundary condition is satisfied. In Section~\ref{subsec:laplace}, we present a technical result to deal with Laplace asymptotics in the presence of moving boundaries.
    In Section~\ref{subsec:semiclassical_estimates}, we prove the key semiclassical estimates needed to conclude the proof of Theorem~\ref{thm:eyring_kramers}, which is done in Section~\ref{subsec:eyring_kramers_final_proof}.
    
    \subsection{Local energy estimates}
    \label{subsec:local_energy_estimates}
            A technical detail one has to address in this construction is that in order for the quasimodes~\eqref{eq:global_quasimode} introduced below to be smooth and supported in~$\Omega_\beta^\pm$, we may need to reduce the value of~$\largeRadius(\beta)$ in Assumption~\eqref{hyp:locally_flat} in order for various energy estimates to hold in the neighborhoods~$B(z_i,\largeRadius(\beta))$ of low-energy saddle points.
            In fact, we show in Proposition~\ref{prop:qm_energy_estimates} that there exists some constant~$\varepsilon_0(V,z_0)>0$ independent of~$\beta$ such that requiring~$\delta(\beta)<\varepsilon_0(V,z_0)$ (see \eqref{hyp:bound_delta}) allows for the construction of valid quasimodes.
            
            We begin by recalling some facts concerning the geometry of the basin of attraction~$\basin{z_0}$ and the behavior of~$V$ near low-energy critical points.
            Recall the definition~\eqref{eq:vstar}.
            \begin{lemma}
                \label{lemma:local_minimum}
                Let~$z\in \partial \basin{z_0}$ be a local minimum for~$V|_{\mathcal S(z_0)}$. Then,~$\nabla V(z) = 0$ and~$\mathrm{Ind}(z)=1$.
            \end{lemma}
            \begin{proof}
                From~\cite[Theorem B.13]{MS14}, we may write~$\mathcal S(z_0)$ as an at most countable disjoint union of submanifolds
                \begin{equation}
                \label{eq:boundary_decomposition}
                \mathcal S({z_0}) = \bigcup_{\mathbf{m}\in\mathcal Z(z_0)} \mathcal W^+(\mathbf m),
                \end{equation}
                where~$\mathcal Z(z_0)$ is the set of critical points of $V$ on~$\mathcal S(z_0)=\partial\basin{z_0} \cap \bigcup_{z\in \mathcal M(V)\setminus \{z_0\}}\partial\basin{z}$, which all have index at least $1$, and we recall the definition~\eqref{eq:stable_manifold} of the stable manifold~$\mathcal W^+$.
                Moreover,~$\mathcal W^+(\mathbf m)$ is a~$(d-\mathrm{Ind}(\mathbf m))$-dimensional manifold.

                From~\eqref{eq:boundary_decomposition}, there exists a unique critical point~$\mathbf m$ of $V$ such that~$z\in\mathcal W^+(\mathbf m)$, so that by definition the gradient flow~$\phi_t(z)$ converges to~$\mathbf m$ in the limit~$t\to\infty$, and is included in~$\mathcal S({z_0})$, which is positively invariant for the gradient flow, as a union of stable manifolds. Since~$V$ decreases along trajectories of~$\phi$ and~$z$ is a local minimum of~$V$ on~$\mathcal S(z_0)$, it must hold that~$\nabla V(z)=0$. In particular, $\mathbf m=z$.

                It only remains to check that~$\Ind(z) = 1$. We first show that there exists~$\mathbf m_1\in\mathcal Z(z_0)$ with~$\Ind(\mathbf m_1)=1$ such that~$z\in \overline{\mathcal W^+(\mathbf m_1)}$. Assume for the sake of contradiction that this is not the case, hence there exists~$r>0$ such that
                \begin{equation}
                    \mathcal S({z_0})\cap B(z,r) \subset \bigcup_{\stackrel{\mathbf{m}\in\mathcal Z(z_0)}{\mathrm{Ind}(\mathbf{m})\geq 2}}\mathcal{W}^+(\mathbf{m}).
                \end{equation}
                Since the set on the right-hand side has dimension at most~$d-2$, the set~$B(z,r)\setminus\mathcal S(z_0)$ is connected (see~\cite[Theorem IV 4, Corollary 1]{HW15}). On the other hand, the sets
                \begin{equation}
                    \mathcal A'(y) := \overline{\basin{y}}\setminus \mathcal S(y),\qquad y\in\mathcal M(V)
                \end{equation}
                are open and pairwise disjoint, and the disjoint open cover
                \begin{equation}
                    B(z,r)\setminus \mathcal S(z_0) \subset \bigcup_{y\in\mathcal M(V)}\mathcal A'(y)
                \end{equation}
                has at least two non-empty components by definition of~$\mathcal S(z_0)$, and therefore cannot be connected.
 
                We finally show that~$\Ind(z)=1$. We already know that~$\Ind(z)\geq 1$, and again assume for the sake of contradiction that~$\Ind(z)>1$.
                Since~$\dim(\mathcal W^-(z)) = \Ind(z)$ and~$\dim(\mathcal W^+(\mathbf m_1))=d-1$, there exists~$v \in T_z \mathcal W^+(\mathbf{m}_1)\cap T_z \mathcal W^-(z)$ such that~$v\neq 0$. Here,~$T_z \mathcal W^+(\mathbf{m}_1)$ denotes the tangent half-space at the boundary, consisting of initial velocities of paths entering~$\mathcal W^+(\mathbf{m}_1)$ starting from~$z$.
                Let then $f$ be a~$\mathcal C^\infty$ map~$[0,1)\to \{z\}\cup \mathcal W^+(\mathbf m_1) \subset \mathcal S(z_0)$ such that~$f(0)=z$ and~$f'(0^+)=v$.
                Expanding, we get
                \[V(f(t))=V(z)+\frac{t^2}2 v^\intercal \nabla^2 V(z) v + \mathcal{O}(t^3),\]
                where we used~$\nabla V(z)=0$ twice. Since~$V(f(t))\geq V(z)$ for~$t$ sufficiently small, this implies~$v^\intercal \nabla^2 V(z) v \geq 0$. We have reached a contradiction, since~$T_z\mathcal W^{-}(z)$ is spanned by the negative eigendirections of~$\nabla^2 V(z)$.
            \end{proof}

            This leads to the following estimate on the position of~$\partial \basin{z_0}$ near a low-energy separating saddle point. Recall the definition of the local coordinates~\eqref{eq:local_coordinates}.
            \begin{lemma}
                \label{lemma:stable_manifold}
                For all~$T>0$, there exist~$\varepsilon_0(T),K(T)>0$ such that for all~$\varepsilon<\varepsilon_0(T)$ and all~$i\in I_{\min}$,
                \begin{equation}
                    \label{eq:basin_local_geometry}
                    \begin{aligned}
                        &\{y_1^{(i)} < -T\varepsilon\} \cap B(z_i,K(T)\sqrt{\varepsilon}) \subset \overline{\basin{z_0}},\\
                        &\{y_1^{(i)} > T\varepsilon\} \cap B(z_i,K(T)\sqrt{\varepsilon}) \subset \R^d \setminus \basin{z_0}.
                    \end{aligned}
                \end{equation}
            \end{lemma}
            \begin{proof}
                Let~$T>0$ and~$i\in I_{\min}$. In a neighborhood of~$z_i$, the decomposition~\eqref{eq:boundary_decomposition} shows that~$\partial \overline{\basin{z_0}} = \mathcal S(z_0)$ coincides with~$\mathcal W^{+}(z_i)$. The stable manifold theorem (see~\cite[Section 9.2]{T24}) implies that~$\mathcal S(z_0)$ is smooth in a neighborhood of~$z_i$, with furthermore~$\n_{\overline{\basin{z_0}}}(z_i) = \hessEigvec{i}{1}$.
                Recall that~$\sigma_{\overline{\basin{z_0}}}$ denotes the signed distance function to~$\partial\overline{\basin{z_0}}$ with the convention~\eqref{eq:sdf}. Its regularity in a neighborhood of~$z_i$ is guaranteed by the fact that, according to the definition~\eqref{eq:vstar_bis},~$z_i$ is a separating saddle point.
                Note that in the case where~$z\in\partial \basin{z_0}$ is a non-separating saddle point, the outward normal~$\n_{\basin{z_0}}(z) = -\nabla \sigma_{\basin{z_0}}(z)$ is not well-defined.
                A Taylor expansion gives
                \[\sigma_{\overline{\basin{z_0}}}(z_i+h) = \sigma_{\overline{\basin{z_0}}}(z_i) -h^\intercal\n_{\overline{\basin{z_0}}}(z_i) + R(h) = -h^\intercal\hessEigvec{i}{1} + R(h) = -y_1^{(i)}(z_i+h)+R(h),\]
                where~$|R(h)| < M_i|h|^2$ for some constant~$M_i>0$ and for all~$|h| < M_i^{-1}$. 
                Choosing~$\varepsilon_{0,i}=\frac{T}{M_i}$, we get, for all~$0<\varepsilon<\varepsilon_{0,i}$ and all~$x\in\{y_1^{(i)} < -T\varepsilon\}\cap B(z_i,K_i\sqrt\varepsilon)$,
                \begin{equation}
                    \sigma_{\overline{\basin{z_0}}}(x) \geq T\varepsilon - K_i^2M_i\varepsilon > 0
                \end{equation}
                for~$K_i<\sqrt{T/M_i}$, thus~$x\in\overline{\basin{z_0}}$ for~$K_i$ sufficiently small.
                
                Setting~$\varepsilon_0(T)=\min_i\varepsilon_{0,i}$ and~$K(T) =\min_i K_i$, we obtain the first inclusion in~\eqref{eq:basin_local_geometry}.
                 The second inclusion follows along the same lines.
            \end{proof}
            The following proposition allows us to multiply~$\largeRadius(\beta)$ by an arbitrarily small positive constant factor in Assumption~\eqref{hyp:energy_well}, which will be convenient for the construction of accurate quasimodes.
            \begin{proposition}
                \label{prop:suff_condition_delta}
                There exist~$\varepsilon_0(V,z_0),C(V,z_0),\beta_0>0$ such that, provided~\eqref{hyp:energy_well} holds and~$\largeRadius(\beta)<\varepsilon_0(V,z_0)$, for all~$0<\rho_0<1$, there exists~$\rho<\rho_0$ such that for all $\beta>\beta_0$, Assumption~\eqref{hyp:energy_well} again holds upon replacing~$\largeRadius(\beta)$ by~$\rho\largeRadius(\beta)$ and~$C_V$ by~$C(V,z_0)\rho^2$, i.e.
                \begin{equation}
                    \tag{\bf EK3'}
                    \label{hyp:energy_well_bis}
                    \left[\basin{z_0}\cap\{V<\Vstar+C(V,z_0)\rho^2\largeRadius(\beta)^2\}\right]\setminus \bigcup_{i\in I_{\min}} B(z_i,\rho\largeRadius(\beta)) \subset \Omega_\beta.
                \end{equation}
                Furthermore,~$\varepsilon_0(V,z_0)$ can be chosen such that, for each~$0\leq i < N$,~$z_i$ is the only critical point of~$V$ in~$B\left(z_i,2\varepsilon_0(V,z_0)\right)$.
            \end{proposition}
            \begin{proof}
                We assume~\eqref{hyp:energy_well}, and show~\eqref{hyp:energy_well_bis}.
                For any~$C<C_V$, the inclusion~$\{V<\Vstar + C\rho^2\largeRadius(\beta)^2\}\subset \{V<\Vstar + C_V\largeRadius(\beta)^2\}$ implies that it is sufficient to show that, for each~$i\in I_{\min}$, it holds, for some~$C>0$ and~$0<\rho<1$, that
                \begin{equation}
                \label{eq:delta_reduction_inclusion}
                \mathscr S^{(i)}(\rho,C):=\basin{z_0}\cap\left\{V<\Vstar+C \rho^2\largeRadius(\beta)^2\right\}\cap\left[B(z_i,\largeRadius(\beta))\setminus B(z_i,\rho\largeRadius(\beta))\right]\subset \Omega_\beta.
                \end{equation}
                The Taylor expansion of~$V$ in the coordinates~\eqref{eq:local_coordinates} around~$z_i$ gives (recalling~$V(z_i)=\Vstar$)
                \begin{equation}
                    \label{eq:V_local}
                        V - \Vstar = \frac12\sum_{j=1}^d \hessEigval{i}{j}y^{(i)2}_j + \frac12 \int_0^1 (1-t)^2 D^3 V\left(y^{(i)-1}\left(ty^{(i)}\right)\right):y^{(i)\otimes 3}\,\d t = \frac12\sum_{j=1}^d \hessEigval{i}{j}y^{(i)2}_j + \O(|y^{(i)}|^3).
                    \end{equation}
                    It follows that the following estimate holds for all~$i\in I_{\min}$ and for some~$M>0$ on~$B(z_i,\frac1M)$:
                    \begin{equation}
                        \label{eq:quadratic_estimate}
                        Q_{-}(y^{(i)}) \leq V-\Vstar,
                    \end{equation}
                    where we denote
                    \begin{equation}
                        \label{eq:quadratic_forms}
                        Q_{-}(y) = M^{- 1}|y'|^2 - M y_1^2 = M^{-1}|y|^2 - (M+M^{-1})y_1^2.
                    \end{equation}
                    We note that the set~$\{Q_{-}>0\}$ is simply given by the cone~$\{|y_1| <  M^{-1}|y'|\}$.

                    We assume that~$\largeRadius(\beta)<M^{-1}$ for sufficiently large~$\beta$, and let~$x\in \mathscr S^{(i)}(\rho,C)$.
                    It follows from~\eqref{eq:quadratic_estimate} and~$V(x)<\Vstar + C \rho^2\largeRadius(\beta)^2$ that
                    $Q_{-}(y^{(i)}(x)) < C \rho^2\largeRadius(\beta)^2$, from which we gather
                    \[\frac{M^{-1}|y^{(i)}(x)|^2-C\rho^2\largeRadius(\beta)^2}{M+M^{-1}}<y_1^{(i)2}(x).\]
                    Therefore, for~$0<C<M^{-1}$, and since~$|y^{(i)}(x)|>\rho\largeRadius(\beta)$, it holds that~$|y_1^{(i)}(x)|>\sqrt{\frac{M^{-1}-C}{M+M^{-1}}}\rho\largeRadius(\beta)$.
                    We take~$C(V,z_0)=C=\min(C_V,M^{-1}/2)$, and set~$\zeta := \sqrt{\frac{M^{-1}-C}{M+M^{-1}}}>0 $.
                    We distinguish two cases.
                    \begin{itemize}
                     \item{If~$y_1^{(i)}(x) < -\zeta \rho\largeRadius(\beta)<\epsLimit{i}/\sqrt\beta - \smallRadius(\beta)$, which holds for~$\beta$ sufficiently large, it holds asymptotically that~$x\in\Omega_\beta$ by~\eqref{hyp:locally_flat}. This fixes~$\beta_0$.}
                    \item{On the other hand, if~$y_1^{(i)}(x) >\zeta \rho\largeRadius(\beta)$ and~$\largeRadius(\beta)<K(\zeta)\sqrt{\varepsilon_0(\zeta)}$, applying Lemma~\ref{lemma:stable_manifold} with~$T=\zeta$ and~$\varepsilon = \largeRadius(\beta)^2/K(T)^2$, it holds that~$x\not\in\basin{z_0}$ provided~$\rho>\largeRadius(\beta)/K(\zeta)^2$. This in turn implies~$x\not\in\mathscr S^{(i)}(\rho,C)$, which contradicts the assumption on~$x$ and precludes this case.}
                    \end{itemize}   
                    It follows that, provided~$\largeRadius(\beta)< \varepsilon_0(V,z_0):=\min\{M^{-1},K(\zeta)^2/2,K(\zeta)\sqrt{\varepsilon_0(\zeta)}\}$, there exist~$0<\rho\leq1/2$ and~$C>0$, such that for sufficiently large~$\beta$, the required inclusion~\eqref{eq:delta_reduction_inclusion} holds.
                    The conclusion of Proposition~\ref{prop:suff_condition_delta} follows by iterating this argument a finite number of times.

                    The final condition on~$\varepsilon_0(V,z_0)$ can clearly be satisfied, since $V$ has finitely many isolated critical points in~$\mathcal K$.
            \end{proof}
            \begin{remark}
                The constant~$\varepsilon_0(V,z_0)$ is actually a function of~$\min_{i\in I_{\min}}\min_{1\leq j\leq d} |\hessEigval{i}{j}|$, where we recall the definition~\eqref{eq:eigvals_hessian}.
            \end{remark}
            We finally state the following simple estimates.
            \begin{lemma}
                \label{lemma:taylor_saddle}
                There exist~$\varepsilon_0>0$,~$0<C_\xi<1$ and $0<M(V,z_0)$ such that, for all~$i\in I_{\min}$ and all~$0<\varepsilon<\varepsilon_0$, it holds that
                \begin{equation}
                    \label{eq:corona_high_energy_strip}
                    \left[B(z_i,\varepsilon)\setminus B\left(z_i,\frac12\varepsilon\right)\right]\cap\left\{|y_1^{(i)}|<C_\xi\varepsilon\right\}\subset\left\{V>\Vstar+M(V,z_0)\varepsilon^2\right\},
                \end{equation}
                \begin{equation}
                    \label{eq:boundary_energy_estimate}
                    \mathcal S(z_0)\setminus \bigcup_{i\in I_{\min}}B(z_i,\varepsilon) \subset \left\{V>\Vstar+M(V,z_0)\varepsilon^2\right\}.
                \end{equation}
            \end{lemma}
            \begin{proof}
                We recall the bound~\eqref{eq:quadratic_estimate}.
                We first find~$C_\xi >0$ so that for all~$0<\varepsilon<\varepsilon'_0:=M^{-1}$,~\eqref{eq:corona_high_energy_strip} holds.
                For any~$0<C_\xi<1$ and all~$x\in \{|y_1^{(i)}|< C_\xi\varepsilon\} \cap \left[B(z_i,\varepsilon)\setminus B\left(z_i,\frac12\varepsilon\right)\right]$, it holds that
                \[Q_{-}(y^{(i)}(x)) \geq \left[\frac1{4M} - C_\xi^2\left(M+\frac1M\right)\right]\varepsilon^2.\]
                In particular, choosing~$C_\xi < \frac1{2\sqrt{1+M^2}}$ ensures that~$V(x)-\Vstar>Q_-(y^{(i)}(x))>0$, and taking $M(V,z_0) <\frac1{4M} - C_\xi^2\left(M+\frac1M\right)$ ensures~\eqref{eq:corona_high_energy_strip}.

                Let us next show~\eqref{eq:boundary_energy_estimate}. Since~$V$ is decreasing along trajectories of the gradient flow~$\phi_t$, it follows that~$V|_{\overline{\basin{z_0}}}$ is bounded from below by~$V(z_0)$. From the boundedness of~$\basin{z_0}\cap\{V<\Vstar\}$ (see Assumption~\eqref{hyp:energy_well_bounded}) and the Morse property of~$V$, the set~$I_{\min}$ is finite, and by Lemma~\ref{lemma:local_minimum} above, the points~$(z_i)_{i\in I_{\min}}$ are the unique minimizers of~$V|_{\mathcal S(z_0)}$. Since~$I_{\min}$ is finite, the stable manifold theorem implies that the restricted Hessian~$\nabla^2 V|_{\mathcal S(z_0)}$ is positive definite in a neighborhood of~$\{z_i\}_{i\in I_{\min}}$.
                Therefore, there exist constants~$C',\varepsilon_0>0$ with~$\varepsilon_0\leq\varepsilon_0'$ such that for all~$0<\varepsilon<\varepsilon_0$, it holds that
                \[\mathcal S(z_0)\setminus \bigcup_{i\in I_{\min}}B(z_i,\varepsilon) \subset \left\{V>\Vstar+C'\varepsilon^2\right\},\]
                and choosing~$M(V,z_0)<C'$ suffices to ensure~\eqref{eq:boundary_energy_estimate}.
            \end{proof}
        \subsection{Construction of the quasimodes on perturbed domains}
        \label{subsec:ek_quasimodes}
        From now on, we assume that assumptions~\eqref{hyp:one_minimum},~\eqref{hyp:energy_well_bounded},~\eqref{hyp:energy_well} and~\eqref{hyp:bound_delta} hold.

        In the literature dealing with semiclassical asymptotics in the presence of a Dirichlet boundary~(see e.g.~\cite{HN06,LPN21}), a common way to define quasimodes is to express them in a dedicated set of local coordinates around each (generalized) critical point of interest, which is adapted to the local quadratic behavior of~$V$ and in which the geometry of the boundary becomes locally linear. This allows one to perform the analysis in a simplified geometric setting.
        However, the flexibility regarding the specific geometry of the boundary afforded by Assumption~\eqref{hyp:locally_flat} makes the definition of such a set of local coordinates rather difficult. Instead, we base our argument on the following comparison principle for Dirichlet eigenvalues, or so-called domain monotonicity, which is well-known in the case of the Laplacian~$V=0$. 
        \begin{proposition}
        \label{prop:comparison_principle}
        Let~$A \subset B$ be bounded open subsets of~$\R^d$,~$\beta>0$ and~$k\in\N^*$.
        Then
        \[\lambda_{k,\beta}(B)\leq\lambda_{k,\beta}(A),\]
        where we recall that~$\lambda_{k,\beta}(A)$ is the~$k$-th eigenvalue of~$-\cL_\beta$ with Dirichlet boundary conditions on~$\partial A$.
    \end{proposition}
    \begin{proof}
        The fact that~$X\in\{A,B\}$ is bounded and open ensures that the Dirichlet realization of~$-\cL_\beta$ on~$X$ has pure point spectrum tending to~$+\infty$, due to the compact embeddings~$H_0^1(X)\hookrightarrow L^2(X)$, and~$\lambda_{k,\beta}(X)$ is therefore well-defined.
        The inequality follows immediately from the Min-Max Courant--Fischer principle and the inclusion of form domains:
        \[H_{0,\beta}^1(A)\subset H_{0,\beta}^1(B).\]        
    \end{proof}
    
    Our approach relies on the construction, for sufficiently large~$\beta$, of two modified domains~$\Omega_\beta^{\pm}$ such that
    \begin{equation}
        \label{eq:domain_sandwich}
        \Omega_\beta^- \subseteq \Omega_\beta \subseteq \Omega_\beta^+,
    \end{equation}
    and whose boundaries are flat in the neighborhood of low-lying separating saddle points close to the boundary. These domains are defined using Proposition~\ref{prop:domain_extension} with~$\rho(\beta)=2\smallRadius(\beta)$, and their boundaries are shaped like hyperplanes in the neighborhood of each~$z_i$ such that~$\epsLimit{i}<+\infty$, as made precise by Equation~\eqref{eq:inclusion_perturbed_domain}.

    \begin{remark}
    \label{rem:energy_well_still_holds}
    Since~\eqref{hyp:energy_well} holds for~$\Omega_\beta$, it also holds for~$\Omega_\beta^{\pm}$. For~$\Omega_\beta^{+}$, this is immediate since~$\Omega_\beta\subset\Omega_\beta^{+}$. For~$\Omega_\beta^{-}$, the proof of Proposition~\ref{prop:domain_extension} shows that one can take~$\Omega_\beta^-$ to be a~$\smallo(\largeRadius(\beta)^2)$-perturbation of~$\Omega_\beta$ outside of~$\bigcup_{i\in I_{\min}} B\left(z_i,\largeRadius(\beta)\right)$. Since~$V$ is uniformly Lipschitz on~$\mathcal K$, this implies that~\eqref{hyp:energy_well} still holds, possibly with a smaller constant~$C_V$.

    For similar reasons, choosing~$\Omega_\beta^-$ to be an~$\smallo(\delta(\beta))$-perturbation of~$\Omega_\beta$ near critical points~$z_i$ such that~$\epsLimit{i}=+\infty$, we may assume that~$B(z_i,\largeRadius(\beta))\subset\Omega_\beta^-$ for~$\beta$ sufficiently large, after possibly reducing~$\delta(\beta)$ by a constant factor.
    We will henceforth assume that both these properties hold for~$\Omega_\beta^-$.
    \end{remark}

    Upon replacing~$\largeRadius(\beta)$ by~$c\largeRadius(\beta)$ for some~$c\leq\frac12$ in Assumption~\eqref{hyp:locally_flat} (which is allowed according to~\eqref{hyp:bound_delta} and Proposition~\ref{prop:suff_condition_delta}), we henceforth assume that, for~$\beta$ sufficiently large,
    \begin{equation}
        \label{eq:perturbed_domain_shape}
        \forall\,i\in I_{\min},\qquad\Omega_\beta^\pm\cap B(z_i,\largeRadius(\beta)) = \left[z_i + \halfSpace{i}\left(\frac{\epsLimit{i}}{\sqrt\beta}\pm 2\gamma(\beta)\right)\right]\cap B(z_i,\largeRadius(\beta)),
    \end{equation}
    which will somewhat simplify the presentation.

    The proof of Theorem~\ref{thm:eyring_kramers} relies on the construction, inspired by~\cite{BGK05,LPM20,LPN21}, of approximate eigenmodes~$\psi_\beta^{\pm}$ for the Dirichlet realization of~$-\cL_\beta$ on each of the domains~$\Omega_\beta^\pm$, which will be sufficiently precise to provide an asymptotic equivalent for~$\lambda_{1,\beta}\left(\Omega_\beta^{\pm}\right)$. The comparison principle of Proposition~\ref{prop:comparison_principle} will finally yield the conclusion of Theorem~\ref{thm:eyring_kramers}. 
    More precisely, solutions to an elliptic PDE, obtained by linearizing~$\cL_\beta$ in the neighborhood of each of the~$\{z_i,\,i\in I_{\min}\}$, are used to define these quasimodes locally, which are then rather crudely extended to define elements of the operator domains~$H_{0,\beta}^1(\Omega_\beta^{\pm})\cap H_\beta^2(\Omega_\beta^\pm)$.
    Using Theorem~\ref{thm:harm_approx} to ensure that a spectral gap exists, we can then use a resolvent estimate (Lemma~\ref{lemma:resolvent_estimate} below) to control the error incurred by projecting the quasimodes onto the subspace spanned by the first eigenmode of the Dirichlet realization of~$-\cL_\beta$ on~$\Omega_\beta^{\pm}$. This allows us to derive the Eyring--Kramers formula in a straightforward way.
    
            Our quasimodes are defined by the following convex combinations
            \begin{equation}
                \label{eq:global_quasimode}
                \psi^{\pm}_\beta = \frac1{Z^{\pm}_\beta}\left[\eta_\beta+\sum_{i\in I_{\min}}\chi_\beta^{(i)}\left(\varphi_\beta^{(i),\pm}- \eta_\beta\right)\right],
            \end{equation}
            where the~$\chi_\beta^{(i)}$ are defined in~\eqref{eq:cutoff}, and where~$Z^{\pm}_\beta$ are normalizing constants imposing
            $$\int_{\Omega^\pm_\beta}{\psi_\beta^\pm}^2 \e^{-\beta V}=1.$$
            This definition uses auxiliary functions~$\eta_\beta$ and~$\left(\varphi_\beta^{(i),\pm}\right)_{i\in I_{\min}}$.
            
            The function~$\eta_\beta$ is a crude cutoff function used to localize the analysis in some small neighborhood of the energy basin separating the minimum~$z_0$ from the low-energy saddle points~$\{z_i,\,i\in I_{\min}\}$.
            More precisely, we define the energy cutoff~$\eta_\beta$ by
            \begin{equation}
                \label{eq:energy_cutoff}
                \eta_\beta(x) = \eta\left(\frac{V(x)-\Vstar}{\energyCutoffConst \largeRadius(\beta)^2}\right)\1_{\overline{\basin{z_0}}}(x),
            \end{equation}
            where~$\eta\in\mathcal C^\infty(\R)$ is a model cutoff function chosen such that
            \begin{equation}
            \label{eq:eta_cutoff_def}
            \1_{(-\infty,\frac12)}\leq \eta\leq \1_{(-\infty,1)},
            \end{equation}
            and~$\energyCutoffConst>0$ is a constant we make precise in Proposition~\ref{prop:qm_energy_estimates} below.
            This rough construction is, in the neighborhood of low-energy saddle points, replaced by a finer local approximation, which one can (roughly) view as the solution to a linearized Dirichlet problem.

            The functions~$\varphi_\beta^{(i),\pm}$ are defined, for~$i\in I_{\min}$, by
            \begin{equation}
                \label{eq:local_quasimode}
                \varphi_\beta^{(i),\pm}(x) = \frac{\displaystyle\int_{y_1^{(i)}(x)}^{\frac{\epsLimit{i}}{\sqrt\beta}\pm 2\smallRadius(\beta)}\e^{-\beta \frac{|\hessEigval{i}{1}|}{2}t^2} \xi_\beta(t)\,\d t}{\displaystyle\int_{-\infty}^{\frac{\epsLimit{i}}{\sqrt\beta}\pm2\smallRadius(\beta)} \e^{-\beta\frac{|\hessEigval{i}{1}|}{2}t^2} \xi_\beta(t)\,\d t}
            \end{equation}
            where~$\xi_\beta$ is again a~$\testfuncs(\R)$ cutoff function used to localize the support of~$\nabla\varphi_\beta^{(i)}$.
            It is convenient for the analysis to work with a~$\varphi_\beta^{(i),\pm}$ whose gradient is localized around~$z_i$. We thus take~$\xi_\beta\in\testfuncs(\R)$ to be an even, smooth cutoff function satisfying
            \begin{equation}
                \label{eq:local_quasimode_cutoff_support}
                \1_{\left(-\frac{\gaussianCutoffConst}2\largeRadius(\beta),\frac{\gaussianCutoffConst}2\largeRadius(\beta)\right)}\leq \xi_\beta\leq \1_{\left(-\gaussianCutoffConst\largeRadius(\beta),\gaussianCutoffConst\largeRadius(\beta)\right)},
            \end{equation}
            where~$\gaussianCutoffConst>0$ is a constant whose value we make precise in Proposition~\ref{prop:qm_energy_estimates} below. 

            \begin{remark}
            \label{rem:committor}
            The functions~$\varphi_\beta^{(i),\pm}$ have a simple probabilistic interpretation. Observe that, upon formally taking~$\xi_\beta = \1_{[-\delta,+\infty)}$ for some small constant~$\delta(\beta)=\delta>0$, and~$\smallRadius(\beta)=0$ in the definition~\eqref{eq:local_quasimode} (in which case we denote~$\varphi_\beta^{(i)}=\varphi_\beta^{(i),\pm}$), an easy computation gives
            \[\left\{\begin{aligned}\left((x-z_i)^\intercal\nabla^2 V(z_i)\nabla - \frac1\beta \Delta\right)\varphi_\beta^{(i)} &= 0\text{ on }\halfSpace{i}\left(\frac{\epsLimit{i}}{\sqrt\beta}\right),\\
            \varphi_\beta^{(i)} &= 0\text{ on }\partial\halfSpace{i}\left(\frac{\epsLimit{i}}{\sqrt\beta}\right),\\
        \varphi_\beta^{(i)} &= 1\text{ on } \partial\halfSpace{i}\left(-\delta\right).\end{aligned}\right.\]
            The solution~$\varphi_\beta^{(i)}$ can be expressed as the so-called committor function (or equilibrium potential, see~\cite{BEGK04})
            $$
            \varphi_\beta^{(i)}(x)=\P_x\left(\tau^{(i)}_{\partial \halfSpace{i}\left(-\delta\right)}<\tau^{(i)}_{\partial\halfSpace{i}\left(\frac{\epsLimit{i}}{\sqrt\beta}\right)}\right),
            $$
            where~$\tau^{(i)}_{A}$ denotes the hitting time of the set~$A\subset \R^d$ for the linearized dynamics around~$z_i$:
            $$
                \d X_t^{\beta,(i)} = -\nabla^2 V(z_i)\left(X_t^{\beta,(i)}-z_i\right)\,\d t + \sqrt{\frac{2}{\beta}}\,\d W_t.
            $$
            The functions~$\varphi_\beta^{(i),\pm}$ should therefore be understood as approximating (locally around~$z_i$) the probability of reaching the interior of the energy well before exiting the domain.
            \end{remark}

            \begin{proposition} 
                \label{prop:qm_energy_estimates}
                There exist a bounded open set~$\mathcal U_0\subset\R^d$, and positive constants~$\beta_0,C_\xi,C_\eta>0$, such that for all~$\beta>\beta_0$,
                the functions~$\psi_\beta^\pm$ defined in~\eqref{eq:global_quasimode} and~$\eta_\beta$ defined in~\eqref{eq:energy_cutoff} satisfy the following conditions:
                \begin{enumerate}[]
                    \item{\begin{equation}\label{eq:quasimode_support}
                        \supp\,\psi_\beta^\pm \subset \overline{\mathcal U_0} \text{ and $z_0$ is the unique minimum of~$V$ in~$\overline{\mathcal U_0}$},
                    \end{equation}}
                    \item{
                        \begin{equation}
                        \label{eq:quasimode_smooth}
                        \psi_\beta^\pm \in\mathcal{C}_{\mathrm{c}}^\infty\left(\R^d\right),
                        \end{equation}
                    }
                    \item{
                        \begin{equation}
                            \label{eq:quasimode_boundary_condition}
                            \psi_\beta^{\pm} \in H_{0,\beta}^1(\Omega_\beta^\pm)\cap H_\beta^2(\Omega_\beta^\pm),
                        \end{equation}
                    }
                    \item{\begin{equation}\label{eq:energy_cutoff_support}
                        \forall\,i\in I_{\min},\quad\supp\,\eta_\beta \cap \left[B(z_i,\largeRadius(\beta))\setminus B\left(z_i,\frac12\largeRadius(\beta)\right)\right] \subset \{y_1^{(i)}<-\gaussianCutoffConst \largeRadius(\beta)\},
                    \end{equation}
                    }
                    \item{\begin{equation}\label{eq:quasimode_high_energy}
                        \supp\,\nabla \psi_\beta^\pm \setminus \bigcup_{i\in I_{\min}} B\left(z_i,\largeRadius(\beta)\right) \subset \left\{ V \geq \Vstar + \frac{\energyCutoffConst}2\largeRadius(\beta)^2\right\}\cap \overline{\basin{z_0}},
                    \end{equation}}
                    \item{\begin{equation}\label{eq:quasimode_locally_fine}
                        \psi_\beta^\pm \equiv \frac1{Z_\beta^\pm}\varphi_\beta^{(i),\pm}\text{ on }B\left(z_i,\frac12\largeRadius(\beta)\right),
                    \end{equation}}

                    \item{\begin{equation}\label{eq:quasimode_flat}
                        \psi_\beta^\pm \equiv \frac1{Z_\beta^\pm}\text{ on }\left[\Omega_\beta^\pm \setminus \bigcup_{i\in I_{\min}} B\left(z_i,\largeRadius(\beta)\right)\right] \cap \left\{V < \Vstar + \frac{\energyCutoffConst}2\largeRadius(\beta)^2\right\} \cap \overline{\basin{z_0}}.
                    \end{equation}}
                \end{enumerate}
            \end{proposition}
            \begin{proof}

                We take~$C_\xi$ as in Lemma~\ref{lemma:taylor_saddle}, and~$C_\eta < \min\{C_V,C(V,z_0),M(V,z_0)\}$, where~$C_V$ is given in~\eqref{hyp:energy_well},~$C(V,z_0)$ is given in~Proposition~\ref{prop:suff_condition_delta} and~$M(V,z_0)$ is given in Lemma~\ref{lemma:taylor_saddle}.
                In the following proof, we reduce the value of~$\delta(\beta)$ several times by invoking~\eqref{hyp:bound_delta} and Proposition~\ref{prop:suff_condition_delta}.

                The properties~\eqref{eq:quasimode_high_energy},~\eqref{eq:quasimode_locally_fine} and~\eqref{eq:quasimode_flat} are immediate consequences of the definitions~\eqref{eq:global_quasimode},~\eqref{eq:energy_cutoff}, and are verified by construction regardless of the value of~$\largeRadius(\beta)$.
                
                The property~\eqref{eq:energy_cutoff_support} follows from~\eqref{eq:energy_cutoff},~\eqref{eq:basin_local_geometry} and~\eqref{eq:corona_high_energy_strip}, once one imposes~$\largeRadius(\beta)$ to be smaller than~$\varepsilon_0(C_\xi)$ obtained in Lemma~\ref{lemma:stable_manifold} and than~$\varepsilon_0$ from Lemma~\ref{lemma:taylor_saddle}, at least for~$\beta$ sufficiently large.

                Let us now define~$\mathcal U_0$ and prove~\eqref{eq:quasimode_support}. For any local minimum~$m\in\R^d$ of~$V$, standard arguments (see~\cite[Chapter 8]{T24}) show that the basin~$\basin{m}$ is open and its boundary contains no local minimum of~$V$.
                Furthermore, for~$\varepsilon<2\varepsilon_0(V,z_0)$, $\overline{B(z_i,\varepsilon)}$ contains no local minimum of~$V$ for any~$i\in I_{\min}$. Let us then define
                \[\mathcal U_0 = \left[\left(\overline{\basin{z_0}}\setminus \mathcal S(z_0)\right)\cap \left\{V<\Vstar + \varepsilon \right\}\right] \cup\bigcup_{i\in I_{\min}} B(z_i,\varepsilon).\]
                Note that, according to~\eqref{hyp:uniformly_bounded} and~\eqref{hyp:energy_well}, there exist~$\delta_0,\varepsilon_0>0$ such that
                $$
                \basin{z_0}\cap\{V<\Vstar+\delta_0\}\subset \mathcal K \cup \bigcup_{i\in I_{\min}} B(z_i,\varepsilon_0),
                $$
                where the set on the right-hand side of the inclusion is bounded by~\eqref{hyp:energy_well_bounded} which implies that~$I_{\min}$ is finite. For~$0<\varepsilon\leq \delta_0$,
                $$
                \basin{z_0}\cap\{V<\Vstar+\varepsilon\}\subset  \basin{z_0}\cap\{V<\Vstar+\delta_0\}
                $$
                is bounded, so that~$\mathcal U_0$ is also bounded for a choice of~$\varepsilon$ sufficiently small. We may also choose~$\varepsilon$ sufficiently small so that $z_0$ is the unique minimum of~$V$ in~$\overline{\mathcal U}_0$.
                
                Let us check that~$\supp\,\psi_\beta^\pm \subset \overline{\mathcal U_0}$.
                For~$i\in I_{\min}$,~$\psi_\beta^\pm$ can be written as
                \[\psi_\beta^\pm = \left\{\begin{aligned}&\localCutoff{i}\varphi_\beta^{(i),\pm}\equiv\varphi_\beta^{(i)}\text{ on }B\left(z_i,\frac12\largeRadius(\beta)\right),\\&(1-\localCutoff{i})\eta_\beta + \localCutoff{i}\varphi_\beta^{(i),\pm}\text{ on }B\left(z_i,\largeRadius(\beta)\right)\setminus B\left(z_i,\frac12\largeRadius(\beta)\right),\\ &\eta_\beta\text{ on }\R^d \setminus \bigcup_{i\in I_{\min}} B\left(z_i,\largeRadius(\beta)\right),\end{aligned}\right.\]
                using \eqref{eq:reference_cutoff},~\eqref{eq:cutoff_supports} and the definition~\eqref{eq:global_quasimode}.
                
                Therefore, it is clear that, for~$C_\eta\largeRadius(\beta)^2,\largeRadius(\beta) < \varepsilon$, it is the case that~$\supp\,\psi_\beta^\pm \subset \overline {\mathcal U_0}$:
                the first condition ensures~$\supp\,\eta_\beta\subset \overline{\mathcal U_0}$, while the second ensures~$\supp\,\sum_{i\in I_{\min}}\chi_\beta^{(i)}\subset \overline{\mathcal U_0}$.
                Reducing~$\largeRadius(\beta)$ once again to satisfy these constraints, using~\eqref{hyp:bound_delta} and Proposition~\ref{prop:suff_condition_delta},~\eqref{eq:quasimode_support} follows.

                Let us show~\eqref{eq:quasimode_smooth}.
                From~\eqref{eq:quasimode_support},~$\psi_\beta^\pm$ is compactly supported, and from the definition~\eqref{eq:global_quasimode} and the smoothness, for each~$i\in I_{\min}$, of~$\varphi_\beta^{(i),\pm}\chi_\beta^{(i)}$ on $B(z_i,\largeRadius(\beta))$, and since~$\psi_\beta^\pm$ coincides with~$\varphi_\beta^{(i),\pm}$ on~$B\left(z_i,\frac12\largeRadius(\beta)\right)$, it is sufficient to check that~$\eta_\beta$ is smooth on~$\R^d\setminus \bigcup_{i\in I_{\min}} B\left(z_i,\frac12\largeRadius(\beta)\right)$.
                In turn, it is sufficient to show, from the definition~\eqref{eq:energy_cutoff}, that~$\mathcal S(z_0)\cap\{V<\Vstar+\energyCutoffConst\largeRadius(\beta)^2\}$ is contained in~$\bigcup_{i\in I_{\min}} B\left(z_i,\frac12\largeRadius(\beta)\right)$ for~$\beta$ sufficiently large, where we use~\eqref{eq:s_alt_rep}. This in turn follows immediately from the estimate~\eqref{eq:boundary_energy_estimate}, and from the choice~$C_\eta<M(V,z_0)$, provided~$\largeRadius(\beta)/2<\varepsilon_0$ from Lemma~\ref{lemma:taylor_saddle}.
                
                We finally show~\eqref{eq:quasimode_boundary_condition}. From~\eqref{eq:quasimode_smooth}, it is clear that~$\psi_\beta^\pm \in H^2_\beta(\Omega_\beta^\pm)$. Thus it only remains to show~$\psi_\beta^\pm \in H_{0,\beta}^1(\Omega_\beta^{\pm})$, for which we in fact show that~$\psi_\beta^{\pm}|_{\partial \Omega_\beta^\pm}\equiv 0$ holds. Again, we decompose the argument.
                On~$\partial\Omega_\beta^{\pm}\setminus \left[\overline{\basin{z_0}}\cup\bigcup_{i\in I_{\min}} B(z_i,\largeRadius(\beta))\right]$, it holds that~$\psi_\beta^\pm\equiv \eta_\beta \equiv 0$, and on~$\left[\partial\Omega_\beta^\pm \setminus \bigcup_{i\in I_{\min}}B(z_i,\largeRadius(\beta)) \right]\cap \overline{\basin{z_0}}$, it holds that~$\psi_\beta^\pm\equiv \eta_\beta\equiv 0$ from the inclusion~\eqref{hyp:energy_well} and the choice~$C_\eta<C_V$.
                Now let~$i\in I_{\min}$. If~$\epsLimit{i}=+\infty$, then~$B(z_i,\largeRadius(\beta)) \subset \Omega_\beta \subset\Omega_\beta^+$ by~\eqref{hyp:locally_flat}. By slightly reducing~$\largeRadius(\beta)$ if necessary, it is possible to ensure that~$B(z_i,\largeRadius(\beta))\subset\Omega_\beta^-$ (see Remark~\ref{rem:energy_well_still_holds}). Thus, in this case,~$\partial\Omega_\beta^\pm \cap B(z_i,\largeRadius(\beta)) = \varnothing$.
                We now assume~$\epsLimit{i}<+\infty$. From~\eqref{eq:perturbed_domain_shape} and~\eqref{eq:local_quasimode}, the boundary condition~$\psi_\beta^{\pm}|_{\partial \Omega_\beta^\pm}\equiv 0$ is satisfied on~$\partial\Omega_\beta^{\pm}\cap B(z_i,\frac12\largeRadius(\beta))$, since on this set,~$\psi_\beta^\pm \equiv \varphi_\beta^{(i),\pm}$.
                On~$\partial\Omega_\beta^{\pm}\cap \left[B(z_i,\largeRadius(\beta))\setminus B\left(z_i,\frac12\largeRadius(\beta)\right)\right]$, this finally follows, for~$\beta$ sufficiently large, from the energy estimate~\eqref{eq:corona_high_energy_strip}, the choices~$\largeRadius(\beta)<\varepsilon$,~$C_\eta<M(V,z_0)$ and the fact that, on this set~$\left|y_1^{(i)}\right| = \left|\epsLimit{i}/\sqrt\beta \pm 2\smallRadius(\beta)\right| \ll C_\xi\largeRadius(\beta)$.
                Since~$\Omega_\beta^\pm$ are regular domains from Proposition~\ref{prop:domain_extension},~$\psi_\beta^\pm \in H_{0}^1(\Omega_\beta^\pm) = H_{0,\beta}^1(\Omega_\beta^\pm)$ by the trace theorem. The latter equality follows from the smoothness of~$V$ and the boundedness of~$\Omega_\beta^\pm$. This concludes the proof of~\eqref{eq:quasimode_boundary_condition}.
                \end{proof}
            \begin{figure}
                \centering
                \begin{tikzpicture}
                    \definecolor{lblue}{rgb}{0.8,0.8,1.0}
                    \definecolor{lred}{rgb}{1.0,0.6,0.6}
                    \definecolor{dgreen}{rgb}{0.0,0.5,0.0}
                    \node (background) at (0,-0.16) {\includegraphics[width=\linewidth]{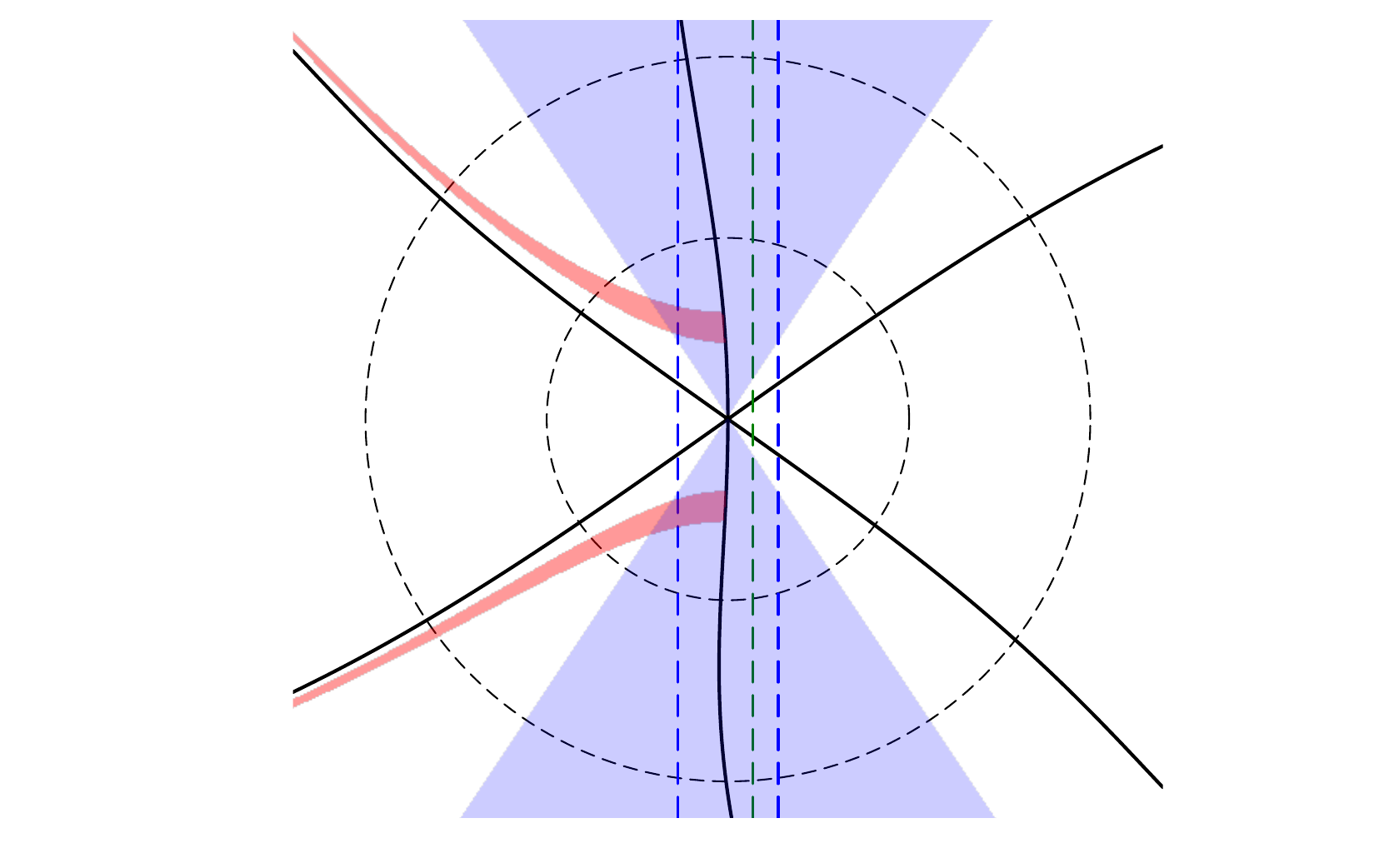}};
                    \draw[thick,->] (-6,0)--(5,0);
                    \draw (5,0) node[above] {$y_1^{(i)}$};

                    \draw[lred,fill=lred](-9,3.66) rectangle (-8.5,3.33);
                    \draw[lblue,fill=lblue] (-9,3.16) rectangle (-8.5,2.83);
                    
                    \draw (-8.5,3.5) node[black,right] {$\{\eta_\beta \neq 0,1\}$};
                    \draw (-8.5,3) node[black,right] {$\{Q_-(y^{(i)})>0\}$};
                    \draw[dgreen,dashed,very thick] (-9,2.5)--(-8.5,2.5) node[black,right] {$\{y_1^{(i)} = \epsLimit{i}/\sqrt\beta + 2\smallRadius(\beta)\}$};
                    \draw[blue,dashed,very thick] (-9,2)--(-8.5,2) node[black,right] {$\{y_1^{(i)}=\pm C_\xi\largeRadius(\beta)\}$};
                    \draw[black,dashed,very thick] (-9,1.5)--(-8.5,1.5) node[right] {$\partial\,\supp\,\nabla \localCutoff{i}$};
                    \draw (-8.75,1) node {\bf 1} (-8.5,1) node[right] {$\frac12\largeRadius(\beta)$};
                    \draw (-8.75,0.5) node {\bf 2} (-8.5,0.5) node[right] {$\largeRadius(\beta)$};
                    \draw (-8.75,0) node {\bf 3} (-8.5,0) node[right] {$\{V=\Vstar\}$};
                    \draw (-8.75,-0.5) node {\bf 4} (-8.5,-0.5) node[right] {$\mathcal W^+(z_i)$};

                    \draw[black,very thick,shift = {(2.31,0)}](0,-0.07)--(0,0.07) (0,0)+(0.23,0.0) node[below,scale=0.8] {\bf 1};
                    \draw[black,very thick,shift = {(4.25,0)}](0,-0.07)--(0,0.07) (0,0)+(0.23,0.0) node[below,scale=0.8] {\bf 2};
                    \draw (4.5,3.0) node[scale=0.8] {\bf 3};
                    \draw (0.2,3.5) node[scale=0.8] {\bf 4};
                \end{tikzpicture}
                \caption{Construction of the quasimode~\eqref{eq:global_quasimode} in the neighborhood of the low-energy saddle point~$z_i$, depicted in the adapted~$y^{(i)}$ coordinates. Here, we depict the elements entering into the construction of~$\psi_\beta^{+}$, in the case~$\epsLimit{i}<+\infty$. The shaded blue cone corresponds to the positive superlevel set of the quadratic form~$Q_{-}$ from the proof of Proposition~\ref{prop:suff_condition_delta}.}
            \end{figure}
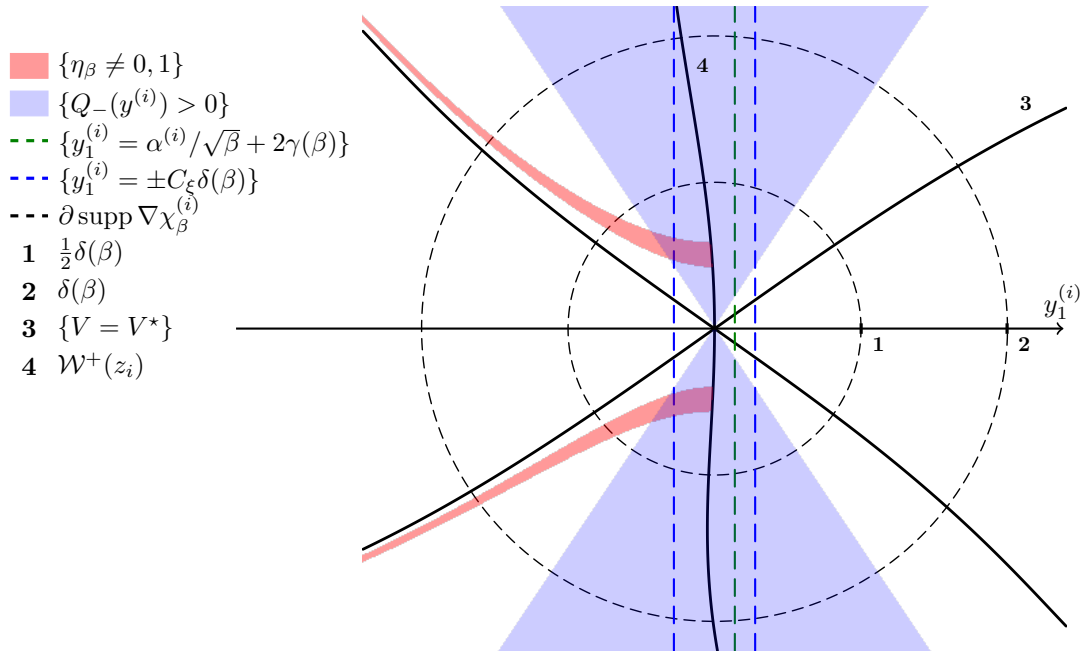

        We assume, for the remainder of this work, and without loss of generality, that~$\largeRadius(\beta)$ is asymptotically sufficiently small for the conclusions of Proposition~\ref{prop:qm_energy_estimates} to hold.
        We now derive the first preliminary result for the proof of the formula~\eqref{eq:eyring_kramers}, which is a variant of the Laplace method for moving domains.
        
        \subsection{Laplace's method on moving domains}
        \label{subsec:laplace}
        We present in this section a key technical tool, a variant of the Laplace method for exponential integrals in the case where the domain of integration varies with the asymptotic parameter. Moreover, we allow for a minimum of the argument of the exponential lying outside of the domain, but close to the boundary in some scaling made precise in Proposition~\ref{prop:laplace}.
        We recall the symmetric difference of sets, which we denote by
        \[A\triangle B := (A\cup B) \setminus (A\cap B) = (A\setminus B) \cup (B\setminus A).\]
        We show the following result.
        \begin{proposition}
            \label{prop:laplace}
                Consider a family~$(A_\lambda)_{\lambda>0}$ of Borel sets. Assume that there exist~$\mathcal K,A_\infty\in \mathcal B(\R^d)$ with non-empty interiors,~$x_0\in \overset{\circ}{\mathcal K}$, and~$\epsilon >0$ such that the following properties hold:
                \begin{itemize}
                    \item{The set~$\mathcal K$ is compact, and the following inclusion holds: \begin{equation}\label{eq:laplace_l1}\forall\,\lambda >0,\qquad A_\lambda \subseteq \mathcal K. \tag{\bf L1}\end{equation}}
                    \item{The functions \begin{equation}\label{eq:laplace_l2}f \in \mathcal C^{4}(\mathcal K),\qquad g \in \mathcal C^{2}(\mathcal K) \tag{\bf L2}\end{equation} are such that \begin{equation}\label{eq:laplace_l3}\underset{x\in\mathcal K}{\mathrm{Argmin}} f(x) = \{x_0\},\qquad \nabla^2 f(x_0) \geq \epsilon\Id,\qquad g(x_0) \neq 0.\tag{\bf L3}\end{equation}}
                    \item{The domains admit a limit in the semiclassical scaling around~$x_0$: \begin{equation}\label{eq:laplace_l4} \1_{\sqrt\lambda(A_\lambda -x_0)} \xrightarrow[\lambda\to+\infty]{\text{a.e.}} \1_{A_\infty}.\tag{\bf L4}\end{equation}}
                \end{itemize}
            Then, 
            \begin{equation}
                \label{eq:laplace_asymptotics}
                \int_{A_\lambda} \e^{-\lambda f(x)}g(x)\,\d x = \left(\frac{2\pi}{\lambda}\right)^{\frac d2}\e^{-\lambda f(x_0)}\left(\det \nabla^2 f(x_0)\right)^{-\frac12}g(x_0)\P(\mathcal G \in A_\infty)\left(1+\O(\varepsilon(\lambda)) + \O(\lambda^{-\frac r2})\right)
            \end{equation}
            as~$\lambda\to +\infty$, where~$\mathcal G$ is a Gaussian random variable with distribution~$\mathcal N\left(0,\nabla^2 f(x_0)^{-1}\right)$, and the dominant error terms are determined by:
            $$\varepsilon(\lambda) = \P\left(\mathcal G \in \left[\sqrt\lambda\left(A_\lambda-x_0\right)\right] \triangle\, A_\infty\right),$$
            and
            \[\begin{cases}
                r = 2,& \text{if } A_\infty = -A_\infty,\\
                r =1 & \text{otherwise}.
            \end{cases}
            \]
        \end{proposition}

        \begin{remark}
            The conclusion still holds true assuming only that~$\mathcal K$ is closed but not necessarily bounded, requiring instead that~$g\in L^1(\mathcal K)$, and that there exists~$\delta_0>0$ such that, for any~$0<\delta<\delta_0$,
            \begin{equation}
                \gamma(\delta) := \underset{x\in \mathcal K \setminus B(x_0,\delta)}{\inf}\, \{f(x)-f(x_0)\}>0.
            \end{equation}
            The proof of this variant is verbatim the same as the one given below, upon replacing~$\gamma$ by $\gamma(\delta)$ with~$\delta<\delta_0$ in~\eqref{eq:laplace_high_energyef_gamma}.
        \end{remark}
    
        \begin{proof}[Proof of Proposition~\ref{prop:laplace}]
            Up to a translation by~$x_0$ and considering instead~$\tilde f(x)=f(x)-f(x_0)$ and~$\tilde g(x)=g(x)/g(x_0)$, we may assume without loss of generality that~$x_0=0$,~$f(0) = 0$ and~$g(0)=1$.

            Let us denote by~$\mathcal Q := \nabla^2 f(0)$ the Hessian of~$f$ at the minimum, which, according to~\eqref{eq:laplace_l3}, is bounded from below by~$\epsilon>0$.
            We make use of the following Taylor expansions, which are valid in view of~\eqref{eq:laplace_l2}:
            \begin{equation}
                \label{eq:laplace_taylor_f}
                f(x) = \frac12 x^\intercal \mathcal Q x + R_f(x),\qquad R_f(x) = \frac12\int_0^1 (1-t)^2 D^3 f(tx):x^{\otimes 3} \,\d t,
            \end{equation}
            \begin{equation}
                \label{eq:laplace_taylor_g}
                g(x) = 1 + x^\intercal \nabla g(0) + R_g(x),\qquad R_g(x)=\int_0^1 (1-t) x^\intercal \nabla^2 g(tx)x\,\d t.
            \end{equation}
            Let~$\delta>0$ be such that~$B(0,\delta) \subset \mathcal K$, which exists since~$x_0\in\overset{\circ}{\mathcal K}$. We may furthermore assume (upon possibly reducing~$\delta$), according to~\eqref{eq:laplace_taylor_f}, that for some~$C>0$,
            \begin{equation}
                \label{eq:laplace_f_minorization}
                \forall\,x\in B(0,\delta),\qquad f(x) \geq \frac1C |x|^2,\qquad |R_f(x)| \leq C|x|^3,\qquad |R_g(x)| \leq C|x|^2.
            \end{equation}
            In addition, by~\eqref{eq:laplace_l3} and the compactness of~$\mathcal K$:
            \begin{equation}
                \label{eq:laplace_high_energyef_gamma}
                \gamma := \underset{x\in \mathcal K\setminus B(0,\delta)}{\min}\, f(x)> 0.
            \end{equation}
            Then, using the inclusion~\eqref{eq:laplace_l1}:
            \begin{equation}
                \left|\int_{A_\lambda \setminus B(0,\delta)} \e^{-\lambda f(x)}g(x)\,\d x\right| \leq \e^{-\lambda \gamma} \|g\|_{L^1(\mathcal K)} = \O(\e^{-\lambda\gamma}),
            \end{equation}
            since~$g$ is integrable on~$\mathcal K$ by~\eqref{eq:laplace_l2}.

            It remains to estimate
            \begin{equation}
                I(\lambda,1) := \int_{A_\lambda \cap B(0,\delta)} \e^{-\lambda f(x)}g(x)\,\d x,
            \end{equation}
            for which we introduce the following parametric integral, for~$0\leq t\leq 1$:
            \begin{equation}
                I(\lambda,t) := \int_{A_\lambda \cap B(0,\delta)} \e^{-\lambda\left(\frac12 x^\intercal \mathcal Q x + t R_f(x)\right)}g(x)\,\d x.
            \end{equation}
            The role of~$t$ is to interpolate between the quadratic approximation of~$f$ around the minimum and~$f$ itself.

            From~\eqref{eq:laplace_f_minorization} and~\eqref{eq:laplace_l3}, we deduce that~$c = \min\{\frac1C,\epsilon/2\}>0$ is such that
            \begin{equation}
                \label{eq:laplace_interpolation_positivity}
                \forall\, x\in B(0,\delta),\,\forall\,0\leq t\leq 1,\qquad\frac 12x^\intercal \mathcal Q x + t R_f(x) = \frac{1-t}2x^\intercal \mathcal Q x + tf(x) \geq c|x|^2.
            \end{equation}
            We then write, by Taylor's theorem:
            \begin{equation}
                I(\lambda,1) = I(\lambda,0) + \frac{\partial I}{\partial t}(\lambda,0)  + \int_0^1 \frac{\partial^2 I}{\partial t^2}(\lambda,t)(1-t)\,\d t,
            \end{equation}
            with
            \begin{equation}
                \frac{\partial^k I}{\partial t^k}(\lambda, t) = (-\lambda)^k\int_{A_\lambda \cap B(0,\delta)} \e^{-\lambda\left(\frac12x^\intercal \mathcal Q x +tR_f(x)\right)}R_f(x)^k g(x)\,\d x.
            \end{equation}
            We can then estimate, for~$k=2$, in view of~\eqref{eq:laplace_interpolation_positivity}:
            \begin{equation}
                \label{eq:laplace_error_a}
                \begin{aligned}
                    \left|\frac{\partial^2 I}{\partial t^2}(\lambda, t)\right| &\leq \|g\|_{L^\infty(B(0,\delta))} \lambda^2\int_{B(0,\delta)}\e^{-\lambda c|x|^2}R_f(x)^2\,\d x\\
                    &\leq K\lambda^2\int_{B(0,\delta)}\e^{- \lambda c|x|^2}|x|^6\,\d x = \lambda^{-\frac d2}\O(\lambda^{-1}),
                \end{aligned}
            \end{equation}
            uniformly in~$t\in(0,1)$, where we used the change of variables~$y=\sqrt\lambda x$ to obtain the last equality.
            It follows that
            \begin{equation}
                \label{eq:laplace_expansion}
                I(\lambda,1) = \left[I + \frac{\partial I}{\partial t}\right](\lambda,0) + \lambda^{-\frac d2}\O(\lambda^{-1}).
            \end{equation}
            The bracketed term may be rewritten as a Gaussian expectation:
            \begin{equation}
                \begin{aligned}
                    \left[I + \frac{\partial I}{\partial t}\right](\lambda,0) &= \int_{A_\lambda \cap B(0,\delta)} \e^{-\lambda \frac12 x^\intercal \mathcal Q x}\left(1 - \lambda R_f(x)\right)g(x)\, \d x\\
                    &= \left(\frac{2\pi}{\lambda}\right)^{\frac d2}|\det \mathcal Q|^{-\frac12}\E\left[\left(1-\lambda R_f({\mathcal G}/\sqrt\lambda)\right)g({\mathcal G}/\sqrt\lambda)\1_{{\mathcal G} \in \sqrt\lambda A_\lambda \cap B(0,\sqrt\lambda \delta)}\right],
                \end{aligned}
            \end{equation}
            where~${\mathcal G}\sim\mathcal N(0,\mathcal Q^{-1})$.
            Before estimating the expectation, we write the following expansion allowed by~\eqref{eq:laplace_l2}:
            \begin{equation}
                \label{eq:laplace_taylor_f2}
                R_f(x) = \frac 16 D^3 f(0):x^{\otimes 3} + \widetilde R_f(x),\qquad \left|\widetilde R_f(x)\right|\leq \widetilde C |x|^4\, \text{on }B(0,\delta),
            \end{equation}
            which, together with~\eqref{eq:laplace_taylor_g}, gives almost surely:
            \begin{equation}
                \begin{aligned}
                    \left(1-\lambda R_f({\mathcal G}/\sqrt\lambda)\right)g({\mathcal G}/\sqrt\lambda) &= \left(1- \frac {\lambda^{-\frac12}}6 D^3 f(0):{\mathcal G}^{\otimes 3} - \lambda\widetilde R_f({\mathcal G}/\sqrt\lambda)\right)\left(1 + \lambda^{-\frac12}{\mathcal G}^\intercal \nabla g(0) + R_g({\mathcal G}/\sqrt\lambda)\right)\\
                    &= 1 + \lambda^{-\frac12}\left[{\mathcal G}^\intercal \nabla g(0)-\frac16 D^3 f(0):{\mathcal G}^{\otimes 3}\right] + \lambda^{-1} S({\mathcal G},\lambda).
                \end{aligned}
            \end{equation}
            A straightforward computation and estimate using the bounds~\eqref{eq:laplace_f_minorization} and~\eqref{eq:laplace_taylor_f2} shows that there exist~$K,\lambda_0 >0$ such that, for all~$\lambda>\lambda_0$, it holds, almost surely,
            \begin{equation}
                \left|S({\mathcal G},\lambda)\right| \leq K\left(1+|{\mathcal G}|^6\right).
            \end{equation}
            In particular~$\E[|S({\mathcal G},\lambda)|] = \O(1)$ in the limit~$\lambda\to\infty$.
            It follows that
            \begin{equation}
            \begin{aligned}
                &\E\left[\left(1-\lambda R_f({\mathcal G}/\sqrt\lambda)\right)g({\mathcal G}/\sqrt\lambda)\1_{{\mathcal G} \in \sqrt\lambda A_\lambda \cap B(0,\sqrt\lambda \delta)}\right]\\
                &= \E\left[\left(1+\lambda^{-\frac12}P({\mathcal G})\right)\1_{{\mathcal G}\in A_\infty}\right]+ \E\left[\left(1+\lambda^{-\frac12}P({\mathcal G})\right)\left(\1_{{\mathcal G}\in\sqrt\lambda A_\lambda \cap B(0,\sqrt\lambda\delta)}-\1_{{\mathcal G}\in A_\infty}\right)\right] + \O(\lambda^{-1}),
            \end{aligned}
            \end{equation}
            where~$P({\mathcal G}) = {\mathcal G}^\intercal \nabla g(0)-\frac16 D^3 f(0):{\mathcal G}^{\otimes 3}$ is an odd polynomial in~${\mathcal G}$. By symmetry, the first term in the sum is then given by:
            \begin{equation}
                \label{eq:laplace_error_b}
                \E\left[(1+\lambda^{-\frac12}P({\mathcal G}))\1_{{\mathcal G}\in A_\infty}\right] = \begin{cases}
                    \P({\mathcal G} \in A_\infty)&\text{if } A_\infty = - A_\infty,\\
                    \P({\mathcal G}\in A_\infty)\left(1+\O(\lambda^{-\frac12})\right)&\text{otherwise.}
                \end{cases}
            \end{equation}
            We are left with the task of estimating
            \begin{equation}
                \label{eq:laplace_error_bb}
                \begin{aligned}
                    \left|\E\left[(1+\lambda^{-\frac12}P({\mathcal G}))\left(\1_{{\mathcal G}\in\sqrt\lambda A_\lambda \cap B(0,\sqrt\lambda\delta)}-\1_{{\mathcal G}\in A_\infty}\right)\right]\right| &\leq \E\left[\left|1 + \lambda^{-\frac12}P({\mathcal G})\right|\1_{{\mathcal G}\in\sqrt\lambda A_\lambda \setminus B(0,\sqrt\lambda\delta)}\right]\\
                    &+ \E\left[\left|1 + \lambda^{-\frac12}P({\mathcal G})\right|\left|\1_{{\mathcal G} \in \sqrt\lambda A_\lambda}-\1_{{\mathcal G} \in A_\infty}\right|\right].
                \end{aligned}
            \end{equation}

            By a standard Gaussian decay estimate and the second condition in~\eqref{eq:laplace_l3}, the term
            \[ \E\left[\left|1 + \lambda^{-\frac12}P({\mathcal G})\right|\1_{\mathcal G\in\sqrt\lambda A_\lambda \setminus B(0,\sqrt\lambda\delta)}\right] = \O(\e^{-\frac{\lambda}3\epsilon \delta^2})\]
            is negligible with respect to~$\lambda^{-1}$. Noting that~$\1_{\mathcal G\in\sqrt\lambda A_\lambda \triangle A_\infty} = \abs{\1_{\mathcal G\in\sqrt\lambda A_\lambda}-\1_{\mathcal G\in A_\infty}}$, we further obtain by a triangle inequality
            \[\E\left[\left|1 + \lambda^{-\frac12}P({\mathcal G})\right|\left|\1_{{\mathcal G}\in\sqrt\lambda A_\lambda}-\1_{{\mathcal G} \in A_\infty}\right|\right] \leq \P\left({\mathcal G} \in \sqrt\lambda A_\lambda \triangle A_\infty\right) + \E\left[\lambda^{-\frac12}|P({\mathcal G})|\1_{\mathcal G\in\sqrt\lambda A_\lambda \triangle A_\infty}\right].\]
            The term~$\varepsilon(\lambda)=\P\left({\mathcal G} \in \sqrt\lambda A_\lambda \triangle A_\infty\right)=\smallo(1)$ by Assumption~\eqref{eq:laplace_l4}. Let us show that~$\E\left[\lambda^{-\frac12}|P({\mathcal G})|\1_{\mathcal G\in\sqrt\lambda A_\lambda \triangle A_\infty}\right]=\O(\varepsilon(\lambda)+\lambda^{-1})$.
            Decomposing this term, we obtain
            \begin{equation}\label{eq:laplace_error_c}\begin{aligned}\E\left[\lambda^{-\frac12}|P({\mathcal G})|\1_{\mathcal G\in\sqrt\lambda A_\lambda \triangle A_\infty}\1_{|P({\mathcal G})|\leq \lambda^{\frac12}}\right] &+ \E\left[\lambda^{-\frac12}|P({\mathcal G})|\1_{\mathcal G\in\sqrt\lambda A_\lambda \triangle A_\infty}\1_{|P({\mathcal G})|>\lambda^{\frac12}}\right]\\
                &\leq \P\left({\mathcal G} \in \sqrt\lambda A_\lambda \triangle A_\infty\right) + \E\left[\lambda^{-\frac12}|P({\mathcal G})|\1_{|P({\mathcal G})|>\lambda^{\frac12}}\right]\end{aligned}.\end{equation}

            Furthermore, since~$P(x)$ is bounded by~$C_{d,f}|x|^{3}~$ for some constant~$C_{d,f}>0$ outside of a compact set, it holds, for sufficiently large~$\lambda$, that
            \begin{equation}
                \label{eq:laplace_error_d}
                \E\left[\lambda^{-\frac12}|P({\mathcal G})|\1_{|P({\mathcal G})|>\lambda^{\frac12}}\right] \leq C_{d,f}\E\left[\lambda^{-\frac12}|{\mathcal G}|^3 \1_{|{\mathcal G}|^3>C_{d,f}^{-1}\lambda^{\frac12}}\right] = \O(\lambda^{-\frac12}\e^{-\epsilon\lambda^{1/3}/3})=\O(\lambda^{-1}),
            \end{equation}
            by a Gaussian decay estimate.

            In view of~\eqref{eq:laplace_expansion}, and collecting the estimates~\eqref{eq:laplace_error_a},~\eqref{eq:laplace_error_b},~\eqref{eq:laplace_error_bb},~\eqref{eq:laplace_error_c} and~\eqref{eq:laplace_error_d}, we conclude that
            \begin{equation}
                I(\lambda,1) = \left(\frac{2\pi}{\lambda}\right)^{\frac d2}\left(\det \mathcal Q\right)^{-\frac12}\P({\mathcal G} \in A_\infty)\left(1+\O(\lambda^{-1})+\O(\varepsilon(\lambda))+ \O(\lambda^{-\frac 12})\1_{A_\infty \neq - A_\infty}\right),
            \end{equation}
            which gives the claimed asymptotic behavior~\eqref{eq:laplace_asymptotics}.
        \end{proof}
        Note that the same proof can be adapted to compute higher-order terms in the asymptotic expansion.

        \subsection{Low-temperature estimates}
        \label{subsec:semiclassical_estimates}
        We obtain in this section the key estimates on the quasimode~\eqref{eq:global_quasimode} needed to compute the asymptotic behavior of~$\left\|\nabla \psi_\beta\right\|_{\Lmu(\Omega_\beta)}$ and~$\left\|\cL_\beta\psi_\beta\right\|_{\Lmu(\Omega_\beta)}$ in the limit~$\beta\to\infty$.
        Good estimates for these quantities, summarized in Proposition~\ref{prop:semiclassical_estimates}, together with a resolvent estimate given by Lemma~\ref{lemma:resolvent_estimate} below, will yield the modified Eyring--Kramers formula~\eqref{eq:eyring_kramers} in the proof of Theorem~\ref{thm:eyring_kramers}, concluded in Section~\ref{subsec:eyring_kramers_final_proof}.

        For convenience, we decompose the analysis according to the following partition of the domain~$\Omega_\beta^{\pm}$:
        \begin{equation}
            \label{eq:domain_partition}
            \Omega_\beta^\pm = \bigcup_{i\in I_{\min}} \left[\Aseti\cup\Bseti{\pm}\cup\Cseti\cup\Dseti{\pm}\right] \cup \Eset\cup \Fset\cup\Gset{\pm},
        \end{equation}
        where, for~$i\in I_{\min}$:
        \begin{enumerate}[]
            \item{\begin{equation}\label{eq:Aseti}\Aseti = \left[B\left(z_i,\largeRadius(\beta)\right)\setminus B\left(z_i,\frac12\largeRadius(\beta)\right)\right]\cap \left\{y_1^{(i)}< -\gaussianCutoffConst\largeRadius(\beta)\right\}\cap \left\{ V-\Vstar \geq \energyCutoffConst\largeRadius(\beta)^2\right\},\end{equation}}
            \item{\begin{equation}\label{eq:Bseti}\Bseti{\pm} =  \Omega_\beta^{\pm}\cap \left[B\left(z_i,\largeRadius(\beta)\right)\setminus B\left(z_i,\frac12\largeRadius(\beta)\right)\right]\cap \left\{|y_1^{(i)}| \leq \gaussianCutoffConst\largeRadius(\beta)\right\},\end{equation}}
            \item{\begin{equation}\label{eq:Cseti}\Cseti =  \left[B\left(z_i,\largeRadius(\beta)\right)\setminus B\left(z_i,\frac12\largeRadius(\beta)\right)\right]\cap \left\{ \frac12\energyCutoffConst\largeRadius(\beta)^2< V-\Vstar < \energyCutoffConst\largeRadius(\beta)^2\right\}\cap \left\{y_1^{(i)}<-\gaussianCutoffConst\largeRadius(\beta)\right\},\end{equation}}
            \item{\begin{equation}\label{eq:Dseti}\Dseti{\pm} = \Omega_\beta^{\pm} \cap B\left(z_i,\frac12\largeRadius(\beta)\right),\end{equation}}
            \item{\begin{equation}\label{eq:Eset}\Eset = \left[\overline{\basin{z_0}} \cap \left\{\frac{\energyCutoffConst}2\largeRadius(\beta)^2\leq V-\Vstar\leq \energyCutoffConst\largeRadius(\beta)^2\right\} \right] \setminus \bigcup_{i\in I_{\min}} B\left(z_i,\largeRadius(\beta)\right),\end{equation}}
            \item{\begin{equation}\label{eq:Fset}\Fset = \left[\overline{\basin{z_0}}\cap \left\{V-\Vstar<\frac{\energyCutoffConst}2\largeRadius(\beta)^2\right\}\right] \setminus \bigcup_{i\in I_{\min}} B\left(z_i,\frac12\largeRadius(\beta)\right),\end{equation}}
            \item{\begin{equation}\label{eq:Gset}\Gset{\pm} = \Omega_\beta^{\pm}\setminus \left(\bigcup_{i\in I_{\min}} \left[\Aseti\cup\Bseti{\pm}\cup\Cseti\cup\Dseti{\pm}\right] \cup \Eset\cup\Fset\right).\end{equation}}
        \end{enumerate}

        Note that, according to~\eqref{eq:perturbed_domain_shape}, the sets~$\Bseti{\pm}$ and~$\Dseti{\pm}$ have simple representations, namely
        \[\Dseti{\pm} = \left\{y_1^{(i)}<\epsLimit{i}/\sqrt\beta\pm2\smallRadius(\beta)\right\} \cap B\left(z_i,\frac12\largeRadius(\beta)\right),\]
        \[\Bseti{\pm} =  \left\{y_1^{(i)}<\epsLimit{i}/\sqrt\beta\pm2\smallRadius(\beta)\right\}\cap \left[B\left(z_i,\largeRadius(\beta)\right)\setminus B\left(z_i,\frac12\largeRadius(\beta)\right)\right]\cap \left\{|y_1^{(i)}| \leq \gaussianCutoffConst\largeRadius(\beta)\right\}.\]
        We refer the reader to Figure~\ref{fig:partition} for a pictorial representation of these sets.
        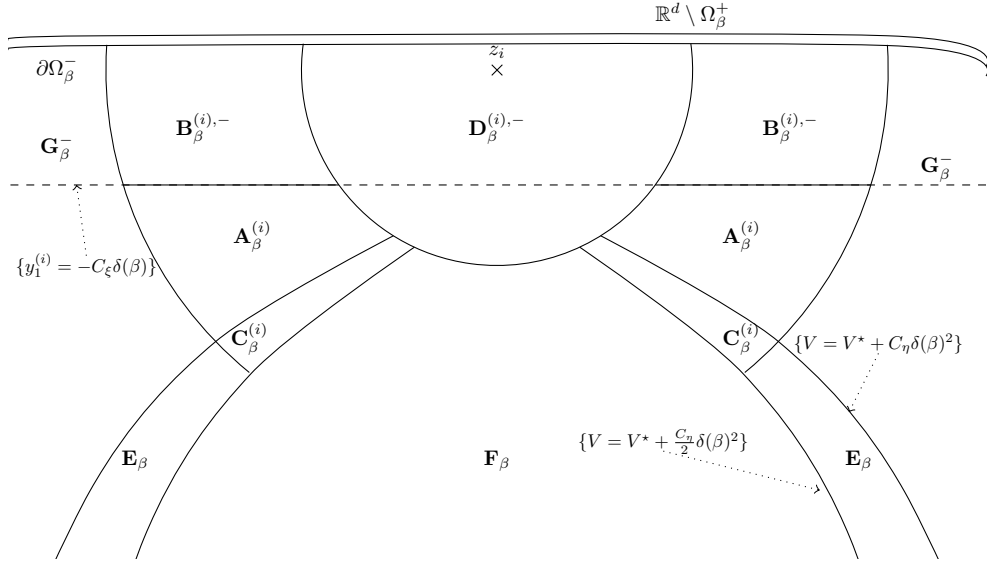
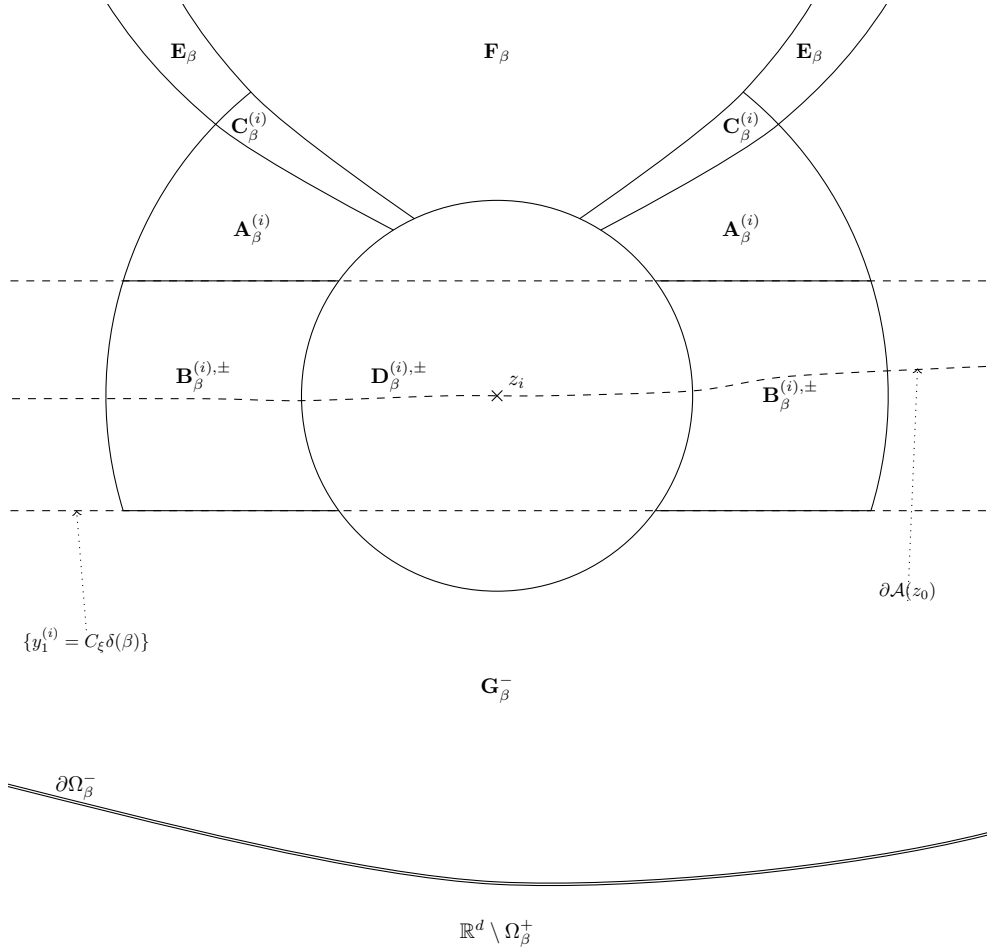
\begin{figure}
            \centering
            \subfigure[The case~$\epsLimit{i}<+\infty$. The boundaries~$\partial \Omega_\beta^{\pm}$ are hyperplanes inside~$B(z_i,\largeRadius(\beta))$, separated by a distance~$4\smallRadius(\beta)$, as described in~\eqref{eq:perturbed_domain_shape}.]{
\begin{tikzpicture}[xscale=1.3,yscale=1.3,rotate=90]
    \coordinate (sw) at (-5,-5);
    \coordinate (ne) at (0.7,5);
    \begin{scope}
        \clip (sw) rectangle (ne);
        \draw(0,0)+(82.243:2) arc (82.243:277.757:2);
        \draw(0,0)+(86.3:4) arc (86.3:140.7:4);
        \draw(0,0)+(-86.3:4) arc (-86.3:-140.7:4);
        \draw (-1.176,1.618) -- (-1.176,3.823);
        \draw[dashed] (-1.176,-5) -- (-1.176,5);
        \draw (-1.176,-1.618) -- (-1.176,-3.823);

    
        \draw plot [smooth] coordinates{+(148:2) +(134:4) +(136:6) +(144:10)};
        \draw plot [smooth] coordinates{+(155:2) +(141:4) +(143:6) +(151:10)};
        \draw plot [smooth] coordinates{+(-148:2) +(-134:4) +(-136:6) +(-144:10)};
        \draw plot [smooth] coordinates{+(-155:2) +(-141:4) +(-143:6) +(-151:10)};

        \draw plot [smooth] coordinates {(0.209,4.996) (0.251,3.992) (0.251,-3.992) (-0.061,-5.0)};
        \draw plot [smooth,shift={(0.1,0)}] coordinates {(0.209,4.996) (0.251,3.992) (0.251,-3.992) (-0.061,-5.0)};

        \draw (-0.6,0) node[scale=0.8] {$\Dseti{-}$};
    
        \draw  (-0.6,-3) node[scale=0.8] {$\Bseti{-}$};
        \draw  (-0.6,3) node[scale=0.8] {$\Bseti{-}$};
    
        \draw  (-1.7,-2.5) node[scale=0.8] {$\Aseti$};
        \draw  (-1.7,2.5) node[scale=0.8] {$\Aseti$};
    
        \draw  (-2.75,-2.5) node[scale=0.8] {$\Cseti$};
        \draw  (-2.75,2.53) node[scale=0.8] {$\Cseti$};
    
        \draw  (-4,-3.7) node[scale=0.8] {$\Eset$};
        \draw  (-4,3.7) node[scale=0.8] {$\Eset$};
    
        \draw  (-4,0) node[scale=0.8] {$\Fset$};
        \draw  (-1,-4.5) node[scale=0.8] {$\Gset{-}$};
        \draw  (-0.8,4.5) node[scale=0.8] {$\Gset{-}$};
        \draw  (-2,4.2) node[scale=0.7] {$\{y_1^{(i)} =-C_\xi\largeRadius(\beta)\}$};
        \draw[dotted,->] (-1.9,4.2) -- (-1.176,4.3) ;
        \draw  (-2.8,-3.9) node[scale=0.7] {$\{V=\Vstar+C_\eta\largeRadius(\beta)^2\}$};
        \draw[dotted,->] (-2.9,-3.9) -- (-3.5,-3.6) ;
        \draw  (-3.8,-1.7) node[scale=0.7] {$\{V=\Vstar+\frac{C_\eta}{2}\largeRadius(\beta)^2\}$};
        \draw[dotted,->] (-3.9,-1.7) -- (-4.3,-3.3) ;
        \draw  (0.55,-2) node[scale=0.8] {$\R^d\setminus \Omega_\beta^+$};
        \draw (0.0,4.5) node[scale=0.8] {$\partial\Omega_\beta^-$};

        \path (0,0) pic {cross=0.1};
        \draw (0.05,0) node[above,scale=0.8] {$z_i$};

    \end{scope}

    \end{tikzpicture}
            }
            \subfigure[The case~$\epsLimit{i}=+\infty$. The width of the shell~$\Omega_\beta^+\setminus\Omega_\beta^-$ is of order~$\varepsilon_1(\beta)$ around~$z_i$, as in the proof of Proposition~\ref{prop:domain_extension}.]{
                \begin{tikzpicture}[xscale=-1.3,yscale=1.3,rotate=-90]
                    \coordinate (sw) at (-4,-5);
                    \coordinate (ne) at (6,5);
                    \begin{scope}
                        \clip (sw) rectangle (ne);
                    \draw(0,0)+(0:2) arc (0:360:2);
                    \draw(0,0)+(72.909:4) arc (72.909:141:4);
                    \draw(0,0)+(-72.909:4) arc (-72.909:-141:4);
                    \draw (-1.176,1.618) -- (-1.176,3.823);
                    \draw (1.176,1.618) -- (1.176,3.823);
                    \draw (-1.176,-1.618) -- (-1.176,-3.823);
                    \draw (1.176,-1.618) -- (1.176,-3.823);
                    \draw[dashed] (-1.176,-5) -- (-1.176,5);
                    \draw[dashed] (1.176,-5) -- (1.176,5);
                
                    \draw plot [smooth] coordinates{+(148:2) +(134:4) +(136:6) +(144:10)};
                    \draw plot [smooth] coordinates{+(155:2) +(141:4) +(143:6) +(151:10)};
                    \draw plot [smooth] coordinates{+(-148:2) +(-134:4) +(-136:6) +(-144:10)};
                    \draw plot [smooth] coordinates{+(-155:2) +(-141:4) +(-143:6) +(-151:10)};

                    \draw[dashed]  plot [smooth] coordinates {(-0.3,-5) (-0.2,-3) (-0.05,-2) (0,-0.5) (0,0.5) (0.05,2) (0.03,3) (0.03,5)};
                    
                    \draw plot [smooth] coordinates{(2.5,-10) (4.5,-5) (5,0) (4,5) (2.5,10)};
                    \draw[shift={(-0.025,0.0)}] plot [smooth] coordinates{(2.5,-10) (4.5,-5) (5,0) (4,5) (2.5,10)};
                    \draw (-0.2,1) node[scale=0.8] {$\Dseti{\pm}$};
                
                    \draw  (0,-3) node[scale=0.8] {$\Bseti{\pm}$};
                    \draw  (-0.2,3) node[scale=0.8] {$\Bseti{\pm}$};
                
                    \draw  (-1.7,-2.5) node[scale=0.8] {$\Aseti$};
                    \draw  (-1.7,2.5) node[scale=0.8] {$\Aseti$};
                
                    \draw  (-2.75,-2.5) node[scale=0.8] {$\Cseti$};
                    \draw  (-2.75,2.53) node[scale=0.8] {$\Cseti$};
                
                    \draw  (-3.5,-3.2) node[scale=0.8] {$\Eset$};
                    \draw  (-3.5,3.2) node[scale=0.8] {$\Eset$};
                
                    \draw  (-3.5,0) node[scale=0.8] {$\Fset$};
                    \draw  (3,0) node[scale=0.8] {$\Gset{-}$}; 

                    \path (0,0) pic {cross=0.1};
                    \draw (0.0,-0.2) node[above,scale=0.8] {$z_i$};
                    
                    \draw  (5.5,0) node[scale=0.8] {$\R^d\setminus \Omega_\beta^+$};
                    \draw  (4.0,4.3) node[scale=0.8] {$\partial \Omega_\beta^-$};

                    \draw  (2.5,4.2) node[scale=0.7] {$\{y_1^{(i)} =C_\xi\largeRadius(\beta)\}$};
                    \draw[dotted,->] (2.4,4.2) -- (1.176,4.3) ;

                    \draw  (2,-4.2) node[scale=0.7] {$\partial\basin{z_0}$};
                    \draw[dotted,->] (2.1,-4.2) -- (-0.28,-4.3) ;
                    \end{scope}
                    \end{tikzpicture}
            }
        \caption{\label{fig:partition}Schematic representation of the partition~\eqref{eq:domain_partition}.}
        \end{figure}
        Before the statements and proof of the necessary estimates, let us give some informal indications of our strategy.
        On the high-energy sets~$\Aseti$,~$\Bseti{\pm}$,~$\Cseti$ and~$\Eset$ (contained in~$\{V\geq \Vstar+\energyCutoffConst\largeRadius(\beta)^2/2\}$), crude uniform bounds on the derivatives of~$\chi_\beta^{(i)}$,~$\eta_\beta$ and~$\varphi_\beta^{(i)}$ will suffice to identify the contribution of these sets to the~$\Lmu(\Omega_\beta)$-norms of both~$\nabla \psi_\beta^\pm$ and~$\cL_\beta\psi_\beta^\pm$ as superpolynomially decaying error terms.
        The sets~$\Fset$ and~$\Gset{\pm}$ are constructed so that~$\psi_\beta^\pm$ is constant on each of them, with respective values~$1/Z_\beta^{\pm}$ and~$0$, and will not contribute to the estimates.
        The only remaining contributions come from~$\Dseti{\pm}$, on which the~$\psi_\beta^\pm$ coincide with the finer approximations $\frac1{Z_\beta^{\pm}}\varphi_\beta^{(i),\pm}$, according to~\eqref{eq:quasimode_locally_fine}. The contribution of this set to the~$\Lmu(\Omega_\beta^\pm)$-norm of~$\nabla\psi_\beta^\pm$ is more finely analyzed using Proposition~\ref{prop:laplace}, giving sufficiently precise asymptotics.

        We now state the required estimates.
        \begin{proposition}
            \label{prop:semiclassical_estimates}
            The following estimates hold:
            \begin{equation}
                \label{eq:laplace_partition_function}
                Z_\beta^\pm = \e^{-\frac{\beta}2 V(z_0)} \left(\frac{2\pi}{\beta}\right)^{\frac d4}|\det \nabla^2 V(z_0)|^{-\frac14}\left(1+\O(\beta^{-1})\right),
            \end{equation}

            \begin{equation}
                \label{eq:laplace_high_energy}
                \left\|\nabla \psi_\beta^\pm\right\|^2_{\Lmu\left(\Omega_\beta^\pm \setminus \bigcup_{i\in{I_{\min}}} \Dseti{\pm}\right)} = \O\left(\e^{-\beta\left(\Vstar - V(z_0) +\frac{C_\eta}3 \largeRadius(\beta)^2\right)}\right),
            \end{equation}
            \begin{equation}
                \label{eq:laplace_generator_high_energy}
                \|\cL_\beta \psi_\beta^\pm\|^2_{\Lmu\left(\Omega_\beta^\pm \setminus \bigcup_{i\in{I_{\min}}} \Dseti{\pm}\right)} = \O\left(\e^{-\beta\left(\Vstar - V(z_0) +\frac{C_\eta}3 \largeRadius(\beta)^2\right)}\right),
            \end{equation}
            and for all~$i\in I_{\min}$, the following hold:
            \begin{equation}
                \label{eq:laplace_local_estimate}
                \left\|\nabla \psi_\beta^\pm\right\|^2_{\Lmu(\Dseti{\pm})} = \frac{\beta|\hessEigval{i}{1}|}{2\pi\Phi\left(|\hessEigval{i}{1}|^{\frac12}\epsLimit{i}\right)}\sqrt{\frac{\det \nabla^2 V(z_0)}{\left|\det \nabla ^2 V(z_i)\right|}} \e^{-\beta(\Vstar-V(z_0))} \left(1 +\O(r_i(\beta))\right),
            \end{equation}
            where~$\Phi$ is given by~\eqref{eq:gaussian_cdf}, and the dominant error term $r_i$ is given by:
            \begin{equation}
                \label{eq:laplace_error_terms}
                r_i(\beta)=\begin{cases}
                    \beta^{-1}, & \epsLimit{i} = +\infty,\\
                    \sqrt\beta\smallRadius(\beta)+\beta^{-\frac12}, & \epsLimit{i} < +\infty.
                \end{cases}
            \end{equation}
            
            Finally, we have
            \begin{equation}
                \label{eq:laplace_generator_local_estimate}
                \|\cL_\beta \psi_\beta^\pm\|^2_{\Lmu(\Dseti{\pm})} = \O\left(\beta^{-2}\largeRadius(\beta)^{-2}\right)\left\|\nabla\psi_\beta^\pm\right\|^2_{\Lmu(\Dseti{\pm})}.
            \end{equation}
        \end{proposition}
        \begin{proof}
            As a first step, we derive the asymptotic behavior of the normalizing constant as stated in~\eqref{eq:laplace_partition_function}.
            \paragraph{Asymptotic behavior of~$Z_\beta^\pm$.\newline}
            Using~\eqref{eq:quasimode_support}, and since~$B(z_0,\largeRadius(\beta))\subset \Omega_\beta$ for sufficiently large~$\beta$ by Assumptions~\eqref{hyp:one_minimum} and~\eqref{hyp:locally_flat}, a direct application of Proposition~\ref{prop:laplace} to the inequality
            \[ \int_{B(z_0,\largeRadius(\beta))} \left[\eta_\beta+\sum_{i\in I_{\min}}\localCutoff{i}\left(\varphi_\beta^{(i),\pm}- \eta_\beta\right)\right]^2 \e^{-\beta V} \leq Z_\beta^{\pm 2}  \leq \int_{{\mathcal U_0}}\left[\eta_\beta+\sum_{i\in I_{\min}}\localCutoff{i}\left(\varphi_\beta^{(i),\pm}- \eta_\beta\right)\right]^2\e^{-\beta V},\]
            since~$ \1_{B(z_0,\largeRadius(\beta))} \leq \left[\eta_\beta+\sum_{i\in I_{\min}}\localCutoff{i}\left(\varphi_\beta^{(i)}- \eta_\beta\right)\right] \leq \1_{\overline{\mathcal U_0}}$,
            gives~\eqref{eq:laplace_partition_function}. Indeed, by a standard Gaussian decay estimate, for~$\xi\sim\mathcal N(0,\nabla^2 V(z_0)^{-1})$,~$\P(\xi\not\in B(0,\sqrt\beta\largeRadius(\beta)))=\O\left(\e^{-c\beta\largeRadius(\beta)^2}\right)$, and $~\P(\xi\not\in B(0,\sqrt\beta {\mathcal U_0})) = \O\left(\e^{-c\beta}\right)$ for some~$c>0$.
            It follows from the scaling~\eqref{hyp:scaling_deltai} that the error term~$\epsilon(\beta)$ in~\eqref{eq:laplace_asymptotics} decays superpolynomially, while the~$\O(\beta^{-\frac12})$ term vanishes by symmetry of the limiting domains~$B(0,\sqrt\beta\largeRadius(\beta)),~\sqrt\beta\left({\mathcal U_0}-z_0\right)\to \R^d$, leaving a dominant error term in~$\O(\beta^{-1})$. This, of course, corresponds to the usual Laplace method.

            As announced, the remainder of the analysis is split according to the partition~\eqref{eq:domain_partition}. We now let~$i\in I_{\min}$ throughout the remainder of the proof.
            \paragraph{Analysis on~$\Fset \cup \Gset{\pm}$.\newline}
            These sets do not contribute to the estimates~\eqref{eq:laplace_high_energy},~\eqref{eq:laplace_generator_high_energy}, since~$\left(\Fset\cup\Gset{\pm}\right) \cap \supp\,\nabla \psi^\pm_\beta = \varnothing$ and~$\supp\,\cL_\beta\psi_\beta^\pm \subseteq \supp\,\nabla\psi_\beta^\pm$.
            
            Indeed, on~$\Gset{\pm}$, it holds that both~$\chi_\beta^{(i)}$ and~$\eta_\beta$ are zero, since~$\Gset{\pm}\cap \bigcup_{i\in I_{\min}}B(z_i,\largeRadius(\beta)) = \varnothing$, which ensures~$\chi_\beta^{(i)}\equiv 0$ according to~\eqref{eq:cutoff_supports}, and~$\Gset{\pm} \subset \left(\R^d\setminus \overline{\basin{z_0}}\right) \cup \left\{V\geq\Vstar+C_\eta\largeRadius(\beta)^2\right\}$, which ensures~$\eta_\beta\equiv 0$ according to~\eqref{eq:energy_cutoff}.
            On~$\Fset$, it holds that~$\varphi_\beta^{(i),\pm}\equiv 1$ for all~$i\in I_{\min}$, according to~\eqref{eq:energy_cutoff_support} and~$\eta_\beta\equiv 1$. Thus,~$\psi_\beta^\pm \equiv 1/Z_\beta^{\pm}$, as well as a convex combination of~$\eta_\beta$ and the~$\varphi_\beta^{(i),\pm}$.
            Therefore, for~$x\in\Fset\cup\Gset{\pm}$, one has~$\nabla\psi_\beta^\pm=0$ and~$\cL_\beta\psi_\beta^\pm = 0$.

            \paragraph{Analysis on~$\Aseti$.\newline}
            From~\eqref{eq:energy_cutoff},\eqref{eq:eta_cutoff_def} and~\eqref{eq:Aseti}, we have~$\eta_\beta \equiv 0$ on~$\Aseti$. Furthermore, from~\eqref{eq:local_quasimode_cutoff_support},~${\varphi_\beta^{(i),\pm} \equiv 1}$ in this set, hence~$\psi_\beta^\pm$ coincides with~$\frac1{Z_\beta^\pm}\chi_\beta^{(i)}$ on~$\Aseti$,
            which gives
            \[\nabla \psi_\beta^\pm = \frac1{Z_\beta^\pm}\nabla \chi_\beta^{(i)},\quad \cL_\beta \psi_\beta^\pm = \frac1{Z_\beta^\pm}\left(-\nabla V\cdot \nabla\chi_\beta^{(i)} + \frac1\beta \Delta \chi_\beta^{(i)}\right),\]
            from which it follows that
            \[\|\nabla\psi_\beta^\pm\|_{L^\infty(\Aseti)}=\O\left(\e^{\frac\beta2 V(z_0)}\beta^{\frac d4}\largeRadius(\beta)^{-1}\right),\|\cL_\beta \psi_\beta^\pm\|_{L^\infty(\Aseti)} = \O\left(\e^{\frac\beta2 V(z_0)}\beta^{\frac d4}\right),\]
            where we used the estimates~\eqref{eq:laplace_partition_function} and~\eqref{eq:linf_bound_nabla_chi}, the first-order estimate~$\nabla V\cdot\nabla \chi_\beta^{(i)}=\O(1)$ on the domain~$\Aseti$, and~$\beta^{-1}\largeRadius(\beta)^{-2}=\smallo(1)$ to absorb the contribution of the Laplacian term~$\Delta \chi_\beta^{(i)}/\beta$.
            We then estimate
            \[\|\nabla\psi_\beta^\pm\|^2_{\Lmu(\Aseti)}=\O\left(\beta^{\frac d2}\largeRadius(\beta)^{d-2}\e^{-\beta(\Vstar-V(z_0)+\energyCutoffConst\largeRadius(\beta)^2)}\right)=\O\left(\e^{-\beta(\Vstar-V(z_0)+\frac\energyCutoffConst3\largeRadius(\beta)^2)}\right),\]
            \[\|\cL_\beta\psi_\beta^\pm\|^2_{\Lmu(\Aseti)} = \O\left(\beta^{\frac d2}\largeRadius(\beta)^{d}\e^{-\beta(\Vstar-V(z_0)+\energyCutoffConst\largeRadius(\beta)^2)}\right)=\O\left(\e^{-\beta(\Vstar-V(z_0)+\frac\energyCutoffConst3\largeRadius(\beta)^2)}\right),\]
            using the inclusions~$\Aseti \subset \left\{V-\Vstar \geq \energyCutoffConst \largeRadius(\beta)^2\right\}$,~$\Aseti\subset B(z_i,\largeRadius(\beta))$, the fact that~$\largeRadius(\beta)$ is bounded by~$\varepsilon$, and the superpolynomial decay of~$\e^{-\beta\largeRadius(\beta)^2}$ which follows from~\eqref{hyp:scaling_deltai}.

            \paragraph{Analysis on~$\Bseti{\pm}$.\newline}
            From~\eqref{eq:energy_cutoff},~\eqref{eq:energy_cutoff_support} and~\eqref{eq:Bseti}, we still have~$\eta_\beta \equiv 0$ on~$\Bseti{\pm}$, however~$\varphi_\beta^{(i),\pm}$ is not constant over this set. Thus,~$\psi_\beta^{\pm}$ is given by~$\frac1{Z_\beta^\pm}\chi_\beta^{(i)}\varphi_\beta^{(i),\pm}$ on~$\Bseti{\pm}$, which yields
            \[\nabla\psi_\beta^{\pm} = \frac1{Z_\beta^\pm}\left(\varphi_\beta^{(i),\pm}\nabla\chi_\beta^{(i)}+\chi_\beta^{(i)}\nabla\varphi_\beta^{(i),\pm}\right),\quad \cL_\beta\psi_\beta^{\pm} = \frac1{Z_\beta^\pm}\left(\varphi_\beta^{(i),\pm}\cL_\beta\chi_\beta^{(i)}+\chi_\beta^{(i)}\cL_\beta\varphi_\beta^{(i),\pm}+\frac2\beta\nabla\varphi_\beta^{(i),\pm}\cdot\nabla\chi_\beta^{(i)}\right).\]
            At this point, we need uniform estimates in~$\Bseti{\pm}$ for derivatives~$\partial^{\alpha}\varphi_\beta^{(i),\pm}$ for~$|\alpha|\leq 2$. Since~$\varphi_\beta^{(i),\pm}$ is in fact simply a function of the affine map~$y_1^{(i)}$, the problem is that of bounding the first two derivatives of
            \[y_1 \mapsto \frac1{C_\beta^{(i),\pm}}{\displaystyle\int_{y_1}^{\frac{\epsLimit{i}}{\sqrt\beta}\pm2\smallRadius(\beta)} \e^{-\beta\frac{|\hessEigval{i}{1}|}{2}t^2} \xi_\beta(t)\,\d t},\quad C_\beta^{(i),\pm}:={\displaystyle\int_{-\infty}^{\frac{\epsLimit{i}}{\sqrt\beta}\pm2\smallRadius(\beta)} \e^{-\beta\frac{|\hessEigval{i}{1}|}{2}t^2} \xi_\beta(t)\,\d t}.\]
            We first estimate~$C_\beta^{(i),\pm}$ by a direct application of Proposition~\ref{prop:laplace} in the one-dimensional case.
            It gives
            \begin{equation}
                \label{eq:c_beta_asymptotic}
                C_\beta^{(i),\pm} = \sqrt{\frac{2\pi}{\beta|\hessEigval{i}{1}|}}\Phi\left(|\hessEigval{i}{1}|^{\frac12}\epsLimit{i}\right)\left(1+e_i(\beta)\right),
            \end{equation}
            where
            \begin{equation}
            \label{eq:c_beta_error}
            \begin{cases}
                e_i(\beta) = \O\left(\sqrt\beta\smallRadius(\beta) + \beta^{-\frac12}\right),&\epsLimit{i}<+\infty,\\
                e_i(\beta) = \O\left(\beta^{-1}\right),&\epsLimit{i}=+\infty.
            \end{cases}
        \end{equation}
            To obtain the formula, we note that for~$\mathcal G\sim\mathcal N(0,|\hessEigval{i}{1}|^{-1})$,
            \begin{equation}
                \label{eq:c_beta_integral_asympt}
            \P\left(\mathcal G\in \left(-\infty,\epsLimit{i}\pm2\sqrt\beta\smallRadius(\beta)\right)\right) \xrightarrow{\beta\to\infty} \Phi\left(|\hessEigval{i}{1}|^{\frac12}\epsLimit{i}\right),
            \end{equation}
            and
            \begin{equation}
                \label{eq:c_beta_integral_error}
                \P\left(\mathcal G\in\left(-\infty,\epsLimit{i}\right)\triangle\left(-\infty,\epsLimit{i}\pm2\sqrt\beta\smallRadius(\beta)\right)\right) = \O\left(\sqrt\beta\smallRadius(\beta)\right)
            \end{equation}
            in the case~$\epsLimit{i}<+\infty$.
            In the case~$\epsLimit{i}=+\infty$, note that the domain of integration is identically equal to~$\R$ for any~$\beta$, and thus the integration error vanishes, leaving only an error term in~$\beta^{-1}$ corresponding to the symmetry of the limiting domain.
            In any case,~$C_\beta^{(i),\pm} = \O\left(\beta^{-\frac12}\right)$, which is sufficient for our purposes, although the finer estimate~\eqref{eq:c_beta_asymptotic} will be useful for the analysis on~$\Dseti{\pm}$.

            We then compute
            \begin{equation}
                \label{eq:local_derivatives}
                \nabla\varphi_\beta^{(i),\pm} = -\frac1{C_\beta^{(i),\pm}}\e^{-\beta\frac{|\hessEigval{i}{1}|}{2}y_1^{(i)2}}\xi_\beta(y_1^{(i)})\hessEigvec{i}{1},\qquad\Delta\varphi_\beta^{(i),\pm}=\frac1{C_\beta^{(i),\pm}}\left[\beta |\hessEigval{i}{1}|y_1^{(i)}\xi_\beta(y_1^{(i)}) -\xi_\beta'(y_1^{(i)})\right]\e^{-\beta\frac{|\hessEigval{i}{1}|}{2}y_1^{(i)2}},
            \end{equation}
            using~$\nabla y_1^{(i)} \equiv \hessEigvec{i}{1}$ and~$|\hessEigvec{i}{1}|=1$.
            It follows from~$0\leq \e^{-\beta\frac{|\hessEigval{i}{1}|}{2}y_1^2},\xi_\beta^{(i)}(y_1)\leq 1$ that
            \[\|\nabla \varphi_\beta^{(i),\pm}\|_{L^\infty(\Bseti{\pm})} = \O(\beta^{\frac12}),\qquad\|\Delta \varphi_\beta^{(i),\pm}\|_{L^\infty(\Bseti{\pm})} =\O\left(\beta^{\frac12}\left[\beta\largeRadius(\beta)+\largeRadius(\beta)^{-1}\right]\right)=\O(\beta^{\frac32}),\]
            using~$\|\xi^{(i)\prime}_\beta\|_{\infty} = \O\left(\largeRadius(\beta)^{-1}\right)$,~$y_1^{(i)}=\O(\largeRadius(\beta))$ on~$\Bseti{\pm}$,~$\largeRadius(\beta)=\O(1)$ and the scaling~$\beta^{\frac12}\largeRadius(\beta)^{-1}\ll \beta/\sqrt{\log\beta}$ given by Assumption~\eqref{hyp:scaling_deltai}.
            It follows that~$\|\cL_\beta \varphi_\beta^{(i),\pm}\|_{L^\infty(\Bseti{\pm})}=\O(\beta^{\frac12})$.
            
            Plugging in these estimates and collecting terms (reusing the estimates~$\nabla\chi_\beta^{(i)}=\O(\largeRadius(\beta)^{-1})$,~$\cL_\beta\chi_\beta^{(i)}=\O(1)$ on~$\Bseti{\pm}$) gives, using~$0\leq \chi_\beta^{(i)},\varphi_\beta^{(i),\pm}\leq 1$:
            \[\|\nabla\psi_\beta^{\pm}\|_{L^\infty(\Bseti{\pm})}=\O\left(\e^{\frac\beta2 V(z_0)}\beta^{\frac d4}\left[\largeRadius(\beta)^{-1}+\beta^{\frac12}\right]\right)=\O\left(\e^{\frac\beta2 V(z_0)}\beta^{\frac d4+\frac12}\right),\]
            \[\|\cL_\beta \psi_\beta^{\pm}\|_{L^\infty(\Bseti{\pm})} = \O\left(\e^{\frac\beta2 V(z_0)}\beta^{\frac d4}\left[1+\beta^{\frac12}+\beta^{-\frac12}\largeRadius(\beta)^{-1}\right]\right)=\O\left(\e^{\frac\beta2 V(z_0)}\beta^{\frac d4+\frac12}\right),\]
            using~$\largeRadius(\beta)^{-1}=\smallo\left(\beta^{\frac12}\right)$ to estimate the square-bracketed terms. This leads to
            \[\|\nabla\psi_\beta^{\pm}\|^2_{\Lmu(\Bseti{\pm})},\|\cL_\beta \psi_\beta^{\pm}\|^2_{\Lmu(\Bseti{\pm})} = \O\left(\beta^{\frac d2+1}\largeRadius(\beta)^d\e^{-\beta(\Vstar-V(z_0)+\energyCutoffConst\largeRadius(\beta)^2)}\right)=\O\left(\e^{-\beta(\Vstar-V(z_0)+\frac\energyCutoffConst3\largeRadius(\beta)^2)}\right),\]
            similarly to the analysis on~$\Aseti$, since~$\Bseti{\pm} \subset B(z_i,\largeRadius(\beta))$ and~$\Bseti{\pm}\subset \{V>\Vstar+\energyCutoffConst\largeRadius(\beta)^2\}$ using~\eqref{eq:corona_high_energy_strip}.

            \paragraph{Analysis on~$\Cseti$.\newline}
            By the definition~\eqref{eq:Cseti},~$\Cseti \subset \{ y_1^{(i)} < -\gaussianCutoffConst\largeRadius(\beta)\}$, hence~$\varphi_\beta^{(i)}\equiv 1$ on this set. Thus, we have locally:
            \[\psi_\beta = \frac1{Z_\beta^{\pm}}\left(\eta_\beta + \chi_\beta^{(i)}(1-\eta_\beta)\right),\]
            whence
            \[\nabla\psi^\pm_\beta = \frac1{Z_\beta^{\pm}}\left(\left[1-\eta_\beta\right]\nabla \chi_\beta^{(i)}+\left[1-\chi_\beta^{(i)}\right]\nabla\eta_\beta\right),\quad \cL_\beta\psi_\beta^\pm = \frac1{Z_\beta^{\pm}}\left(\left[1-\eta_\beta\right]\cL_\beta \chi_\beta^{(i)}+\left[1-\chi_\beta^{(i)}\right]\cL_\beta\eta_\beta - \frac2\beta \nabla \chi_\beta^{(i)}\cdot\nabla\eta_\beta\right),\]
            by straightforward manipulations. One then only needs to check that
            \[\|\nabla \eta_\beta\|_{L^\infty(\Cseti)} = \O\left(\largeRadius(\beta)^{-2}\right),\quad \|\Delta \eta_\beta\|_{L^\infty(\Cseti)} = \O\left(\largeRadius(\beta)^{-4}\right)\]
            to obtain by similar arguments:
            \[\|\nabla\psi_\beta^\pm\|^2_{\Lmu(\Cseti)},\|\cL_\beta\psi_\beta^\pm\|^2_{\Lmu(\Cseti)} = \O\left(\beta^{\frac d2}\largeRadius(\beta)^{d-4}\e^{-\beta(\Vstar-V(z_0)+\frac\energyCutoffConst2\largeRadius(\beta)^2)}\right)=\O\left(\e^{-\beta(\Vstar-V(z_0)+\frac\energyCutoffConst3\largeRadius(\beta)^2)}\right),\]
            using the inclusions~$\Cseti\subset B(z_i,\largeRadius(\beta))$, $\Cseti \subset \left\{V-\Vstar \geq \frac\energyCutoffConst2 \largeRadius(\beta)^2\right\}$.
            \paragraph{Analysis on~$\Eset$.\newline}
            On the set~$\Eset$, we have~$\chi_\beta^{(i)}\equiv 0$ for all~$i\in I_{\min}$, hence~$\psi_\beta^\pm$ coincides with~$\frac1{Z_\beta^{\pm}}\eta_\beta$. Reusing the bounds on the derivatives of~$\eta_\beta$ from the analysis on~$\Cseti$ (here we use the fact that~$\Eset\subset\mathcal K$ is bounded), we obtain once again:
            \[\|\nabla\psi_\beta^\pm\|^2_{\Lmu(\Eset)},\|\cL_\beta\psi_\beta^\pm\|^2_{\Lmu(\Eset)} = \O\left(\beta^{\frac d2}\largeRadius(\beta)^{-4}\e^{-\beta(\Vstar-V(z_0)+\frac\energyCutoffConst2\largeRadius(\beta)^2)}\right)=\O\left(\e^{-\beta(\Vstar-V(z_0)+\frac\energyCutoffConst3\largeRadius(\beta)^2)}\right).\]
            Summing the estimates on~$\Aseti$,~$\Bseti{\pm}$,~$\Cseti$ and~$\Eset$, we obtain~\eqref{eq:laplace_high_energy} and~\eqref{eq:laplace_generator_high_energy}.            
            \paragraph{Analysis on~$\Dseti{\pm}$.\newline}
            By~\eqref{eq:quasimode_locally_fine},~$\psi_\beta^\pm$ coincides with~$\frac1{Z_\beta^{\pm}}\varphi_\beta^{(i),\pm}$ on~$\Dseti{\pm}$, and here we turn to the finer estimates provided by Proposition~\ref{prop:laplace}.
            Using the computation~\eqref{eq:local_derivatives} once again, we get
            \[|\nabla\varphi_\beta^{(i),\pm}|^2\e^{-\beta V}= \frac1{\left(C_\beta^{(i),\pm}\right)^2}\e^{-\beta\left(|\hessEigval{i}{1}|y^{(i)2}_1+ V\right)}\xi_\beta(y_1^{(i)})^2.\]
            We next note using the Taylor expansion~\eqref{eq:V_local} that $W_i(x):=|\hessEigval{i}{1}|y^{(i)}_1(x)^2+ V(x)$ has a strict local minimum at~$z_i$, with a Hessian given by~${\nabla^2 W_i(z_i) = \mathrm{abs}\left(\nabla^2 V(z_i)\right) = U^{(i)}\mathrm{diag}\left(|\hessEigval{i}{1}|,\dots, \hessEigval{i}{d}\right)}U^{(i)\top}$, see~\eqref{eq:eigvecs_hessian}. Moreover, this minimum is unique in~$B\left(z_i,\frac12\largeRadius(\beta)\right)$ for~$\largeRadius(\beta)$ sufficiently small, which we may assume upon reducing~$\largeRadius(\beta)$ once again.

            Since~$\xi_\beta(y_1^{(i)}(z_i))^2=\xi_\beta(0)=1$ we may estimate~$\|\nabla \varphi_\beta^{(i),\pm}\|_{\Lmu(\Dseti{\pm})}$ using Proposition~\ref{prop:laplace}.

            Let us first note that, according to~\eqref{eq:perturbed_domain_shape} and~\eqref{hyp:locally_flat},~$\sqrt\beta\left(\Dseti{\pm}-z_i\right)\xrightarrow{\beta\to\infty}\halfSpace{i}(\epsLimit{i})$ in the sense of~\eqref{eq:laplace_l4}.
            Let~$\mathcal G\sim \mathcal N\left(0,\mathrm{abs}\left(\nabla^2 V(z_i)\right)^{-1}\right)$.  
            It is then easy to check that
            \[\P\left(\mathcal G \in \halfSpace{i}(\epsLimit{i})\right)=\P\left(\mathcal G^\top\hessEigvec{i}{1}<\epsLimit{i}\right)=\Phi\left(|\hessEigval{i}{1}|^{\frac12}\epsLimit{i}\right),\]
            since~$\mathcal G^\top\hessEigvec{i}{1}\sim \mathcal N(0,|\hessEigval{i}{1}|^{-1})$. Furthermore,
            \begin{equation}
                \label{eq:hi_estimate}
                h_i(\beta):=\P\left(\mathcal G \in \sqrt\beta\left[\Dseti{\pm}-z_i\right] \triangle \halfSpace{i}(\epsLimit{i})\right) = \begin{cases}
                    \O\left(\sqrt\beta\smallRadius(\beta)+\e^{-c\beta\largeRadius(\beta)^2}\right),& \epsLimit{i}< +\infty,\\
                    \O\left(\e^{-c\beta\largeRadius(\beta)^2}\right),&\epsLimit{i} = +\infty.
                \end{cases}
            \end{equation}
            Indeed, it follows from~\eqref{eq:perturbed_domain_shape} and~\eqref{hyp:locally_flat} that the following inclusion holds:
            \begin{equation}
                \label{eq:union_bound}
                \sqrt\beta\left(\Dseti{\pm}-z_i\right)\triangle \halfSpace{i}(\epsLimit{i})\subset \left[\halfSpace{i}(\epsLimit{i})\triangle\halfSpace{i}(\epsLimit{i}\pm\sqrt\beta2\smallRadius(\beta))\right]\cup B\left(0,\frac{\sqrt\beta}{2}\largeRadius(\beta)\right)^{\mathrm c},
            \end{equation}
             which we use to estimate~$h_i(\beta)$ with the union bound. In the case~$\epsLimit{i}=+\infty$, the leftmost set in~\eqref{eq:union_bound} is empty and the only contribution is from the second term, which is handled using a standard Gaussian estimate
            $$\P\left(\mathcal G \not\in B\left(0,\frac{\sqrt\beta\largeRadius(\beta)}2\right)\right) = \O\left(\e^{-c\beta\largeRadius(\beta)^2}\right)$$ for some~$c>0$ depending only on~$i$.

            In the case~$\epsLimit{i}<+\infty$, we have a contribution from the leftmost set in~\eqref{eq:union_bound}
            \[\mathbb P\left(\mathcal G\in \halfSpace{i}(\epsLimit{i})\triangle\halfSpace{i}(\epsLimit{i}\pm 2\sqrt\beta\smallRadius(\beta))\right) = \mathbb P\left(\mathcal G^\top\hessEigvec{i}{1}\in \left(\epsLimit{i},\epsLimit{i}\pm 2\sqrt{\beta}\smallRadius(\beta)\right)\right),\]
            whose asymptotic behavior has already been computed in~\eqref{eq:c_beta_integral_error}. The union bound yields~\eqref{eq:hi_estimate}.
            
            Applying Proposition~\ref{prop:laplace}, we estimate
            \begin{equation}
                \begin{aligned}
                \|\nabla\varphi_\beta^{(i),\pm}\|^2_{\Lmu(\Dseti{\pm})} &= (C_\beta^{(i),\pm})^{-2}\left(\frac{2\pi}{\beta}\right)^{\frac d2}\e^{-\beta \Vstar}\left|\det \nabla^2 V(z_i)\right|^{-\frac12}\Phi\left(|\hessEigval{i}{1}|^{\frac{1}{2}}\epsLimit{i}\right)\left(1+e_i(\beta)\right)\\
                &=\frac{\beta |\hessEigval{i}{1}|}{2\pi\Phi\left(|\hessEigval{i}{1}|^{\frac{1}{2}}\epsLimit{i}\right)}\left(\frac{2\pi}{\beta}\right)^{\frac d2}\e^{-\beta \Vstar}\left|\det \nabla^2 V(z_i)\right|^{-\frac12}\left(1+e_i(\beta)\right).
                \end{aligned}
            \end{equation}
            The error term~$e_i$ is once again given by~\eqref{eq:c_beta_error} (since the limiting domain is symmetric if and only if~$\epsLimit{i}=+\infty$), and we used~\eqref{eq:c_beta_asymptotic} in the final line. Combining this estimate with~\eqref{eq:laplace_partition_function} finally yields~\eqref{eq:laplace_local_estimate}.
            
            Let us show~\eqref{eq:laplace_generator_local_estimate}. We write, for~$x \in \Dseti{\pm} \subseteq B\left(z_i,\frac12\largeRadius(\beta)\right)$, in the~$y^{(i)}$-coordinates and for~$\largeRadius(\beta)$ sufficiently small,
            \begin{align*}
                \cL_\beta\psi^{\pm}_\beta &= \frac1{Z_\beta^{\pm}}\left(-\nabla V \cdot \nabla \varphi_\beta^{(i),\pm} + \frac1\beta\Delta\varphi_\beta^{(i),\pm}\right)\\
                &=\frac1{Z_\beta^\pm C_\beta^{(i),\pm}}\left(\xi_\beta(y_1^{(i)})\nabla V \cdot \nabla y_1^{(i)}+\frac1\beta\left[- \xi_\beta(y_1^{(i)})\Delta y_1^{(i)} + \left(-\xi_\beta'(y_1^{(i)})+\beta \xi_\beta(y_1^{(i)})|\hessEigval{i}{1}|y_1^{(i)}\right)\abs{\nabla y_1^{(i)}}^2\right]\right)\e^{-\beta\frac{|\hessEigval{i}{1}|}{2}y_1^{(i)2}}\\
                &= \frac1{Z_\beta^\pm C_\beta^{(i),\pm}}\left(\xi_\beta(y_1^{(i)})\hessEigvec{i}{1}\cdot\left[\nabla V+|\hessEigval{i}{1}|y_1^{(i)}\hessEigvec{i}{1}\right]+\O\left(\beta^{-1}\|\xi_\beta'\|_{L^\infty(\R)}\right)\right)\e^{-\beta\frac{|\hessEigval{i}{1}|}{2}y_1^{(i)2}}\\
                &=\frac1{Z_\beta^\pm C_\beta^{(i),\pm}}\left(\O\left(|y^{(i)}|^2\right)+\O\left(\beta^{-1}\largeRadius(\beta)^{-1}\right)\right)\e^{-\beta\frac{|\hessEigval{i}{1}|}{2}y_1^{(i)2}}
            \end{align*}
            using a first-order Taylor expansion of $\nabla V$ around~$z_i$ in the last line. We now estimate the~$\Lmu\left(\Dseti{\pm}\right)$-norms.
            Noting that, by the change of variables~$z = \sqrt\beta y^{(i)}$,
            \[\int_{\Dseti{\pm}}|y^{(i)}|^4\e^{-\beta\left(|\hessEigval{i}{1}|y_1^{(i)2}+V\right)}=\O\left(\beta^{-\frac{d}2-2}\e^{-\beta\Vstar}\right),\]
            we get
            \begin{equation}
                \begin{aligned}
                \|\cL_\beta\psi_\beta^{\pm}\|^2_{\Lmu(\Dseti{\pm})} &= \left(Z_\beta^\pm C_\beta^{(i),\pm}\right)^{-2}\left[\O\left(\beta^{-\frac{d}2-2}\e^{-\beta\Vstar}\right)+\beta^{-2}\largeRadius(\beta)^{-2}\left\|\e^{-\beta\frac{|\hessEigval{i}{1}|}{2}y_1^{(i)2}}\right\|^2_{\Lmu(\Dseti{\pm})}\right]\\
                &=\left(Z_\beta^{\pm}C_\beta^{(i),\pm}\right)^{-2}\O\left(\beta^{-2}\largeRadius(\beta)^{-2}\right)\beta^{-\frac d2}\e^{-\beta \Vstar}\\
                &=\O\left(\beta^{-2}\largeRadius(\beta)^{-2}\right)\|\nabla \psi_\beta^\pm\|^2_{\Lmu(\Dseti{\pm})},
                \end{aligned}
            \end{equation}
            where we used the same change of variables in the second line, and the estimates~\eqref{eq:laplace_partition_function},~\eqref{eq:c_beta_asymptotic} and \eqref{eq:laplace_local_estimate} in the last line.
            This concludes the proof of~\eqref{eq:laplace_generator_local_estimate}.
        \end{proof}

        \subsection{Conclusion of the proof of Theorem~\ref{thm:eyring_kramers}}
        \label{subsec:eyring_kramers_final_proof}
        The last tool for the proof is the following resolvent estimate, which was already used for the estimation of metastable exit times in the semiclassical approach, see~\cite[Proposition 27]{LPN21} and~\cite[Proposition 3.4]{LRS24}.
        We include its proof (in our weighted~$\Lmu$ setting) for the sake of completeness.

        Throughout this section, let us denote by
        \[\lambda_{1,\beta}^{\pm} := \lambda_{1,\beta}(\Omega_\beta^\pm),\qquad u_{1,\beta}^\pm := u_{1,\beta}(\Omega_\beta^\pm),\]
        the principal Dirichlet eigenpairs of~$-\cL_\beta$ in~$\Omega_\beta^\pm$.

        We introduce the spectral projectors associated with the principal eigenspaces: for all~$\varphi\in \Lmu(\Omega_\beta^\pm)$,
        \[\pi^\pm_\beta \varphi := \frac{\left\langle u^\pm_{1,\beta},\varphi \right\rangle_{\Lmu(\Omega_\beta^\pm)}}{\|u^\pm_{1,\beta}\|^2_{\Lmu(\Omega_\beta^\pm)}}u^\pm_{1,  \beta}.\]
        \begin{lemma}
            \label{lemma:resolvent_estimate}
            Fix~$u\in H_{0,\beta}^1(\Omega_\beta^\pm)\cap H^2_\beta(\Omega_\beta^\pm)$. Then,
            \begin{equation}
                \label{eq:lemma_resolvent_eqa}
                \|(1-\pi^\pm_{\beta})u\|_{\Lmu(\Omega_\beta^\pm)} = \O(\|\cL_\beta u\|_{\Lmu(\Omega_\beta^\pm)}),
            \end{equation}
            \begin{equation}
                \label{eq:lemma_resolvent_eqb}
                \left\|\nabla \pi^\pm_{\beta} u\right\|_{\Lmu(\Omega_\beta^\pm)}^2 = \left\|\nabla u\right\|_{ \Lmu(\Omega_\beta^\pm)}^2 + \O\left(\beta\|\cL_\beta u\|_{\Lmu(\Omega_\beta^\pm)}^2\right).
            \end{equation}
        \end{lemma}
        \begin{proof}
            Let us first check that, thanks to Assumption~\eqref{hyp:one_minimum},~$\lambda_{1,\alpha}^{\mathrm H}=0$ and $\lambda_{2,\alpha}^{\mathrm H}>0$, where we recall~$\lambda_{j,\alpha}^{\mathrm H}$ is defined in~\eqref{eq:harm_spectrum_enumeration}.
            We recall the expressions~\eqref{eq:hermite_eigenfunction_scaled} and~\eqref{eq:full_harmonic_eigenstates} for the harmonic eigenvalues. In the following, we use the multi-index enumeration convention for the spectra, i.e.~$\mathrm{Spec}\left(K^{(i)}_{\alpha^{(i)}}\right)=\left\{\lambda^{(i)}_{n,\epsLimit{i}}\right\}_{n\in\N^d}$.
            Note that the ground state energy of each of these operators is given by~$\lambda^{(i)}_{(0,\dots,0),\epsLimit{i}}\geq 0$, and that~$\lambda_{1,\alpha}^{\mathrm H}=\lambda^{(0)}_{(0,\dots,0),\infty}=0$. Let us now show that~$\lambda_{2,\alpha}^{\mathrm H}>0$.
            
            It is clear that~$\lambda^{(0)}_{n,\infty}>0$ for~$n \neq 0\in\N^d$.
            Besides, for~$i\geq N_0$, it holds that
            $$\lambda^{(i)}_{(0,\dots,0),\epsLimit{i}} = |\hessEigval{i}{1}|\mu_{0,\alpha^{(i)}(|\hessEigval{i}{1}|/2)^{1/2}} -\frac{\hessEigval{i}{1}}{2} + \frac12\sum_{j=2}^d\left[|\hessEigval{i}{j}|-\hessEigval{i}{j}\right] \geq \frac12\sum_{j=1}^d \left[|\hessEigval{i}{j}|-\hessEigval{i}{j}\right]>0,$$
            since~$\mathrm{Ind}(z_i)\geq 1$. We used the inequality~$\mu_{0,\theta}\geq \mu_{0,\infty} = \frac12$ for any $\theta \in \R \cup \{\infty\}$, which follows directly from the Courant--Fischer principle, similarly to the proof of Proposition~\ref{prop:comparison_principle}.\newline
            To get~$\lambda_{2,\alpha}^{\mathrm H}>0$, it therefore remains to show that~$\lambda^{(i)}_{(0,\dots,0),\epsLimit{i}}>0$ for~$1\leq i < N_0$. For these local minima~$z_i$, it holds that~$\epsLimit{i}<+\infty$ by Assumption~\eqref{hyp:one_minimum}.
            It is thus sufficient to check that~$\mu_{0,\theta}>\frac12 = \mu_{0,\infty}$ for any~$\theta\in \R$. The inequality~$\mu_{0,\theta}\geq\mu_{0,\infty}$ follows again from the domain monotonicity of Dirichlet eigenvalues.
            For the strict inequality, we note that the identity~$\mathfrak{H}_\theta v_{0,\theta}=\mu_{0,\infty}v_{0,\theta}$ would contradict the fact that~$\mu_{0,\infty}$ is a simple eigenvalue of~$\mathfrak{H}_{\infty}$. Indeed, if~$\mu_{0,\theta}=\frac12$, the trivial extension of$v_{0,\theta}$ to~$\R$ is a minimizer in~$H^1(\R)$ of the Rayleigh quotient, hence an eigenfunction of~$\mathfrak{H}_{\infty}$. This is impossible because this extension and~$v_{0,\infty}$ are linearly independent, and~$\mu_{0,\infty}$ is simple. This concludes the proof of~$\lambda_{2,\alpha}^{\mathrm H}>0$.
            
            Note that the same argument applies to the perturbed domains~$\Omega_\beta^{\pm}$, which still satisfy the assumptions of Theorem~\ref{thm:harm_approx}. Therefore, there exist~$r,\beta_0>0$ such that for all~$\beta>\beta_0$:
            \[|\lambda_{1,\beta}| <r,\quad \lambda_{2,\beta}>3r,\]
            so that the circular contour~$\Gamma_{2r} = \{2r\e^{2i\pi t}, 0\leq t \leq 1\}$ is at distance at least~$r$ from the Dirichlet spectrum of~$-\cL_\beta$ on~$\Omega_\beta^\pm$.
            When needed, in this proof, we indicate explicitly by~$\cL_\beta^\pm$ the fact that we consider the Dirichlet realization of~$\cL_\beta$ on~$\Omega_\beta^\pm$.\newline
            A standard corollary of the spectral theorem then yields the following uniform resolvent estimate:
            \[\forall z \in \Gamma_{2r},\quad\|(-\cL_\beta^\pm-z)^{-1}\|_{\mathcal B(\Lmu(\Omega_\beta^\pm))} \leq \frac 1r.\]
            Furthermore, $\pi_{\beta}^\pm$ may be expressed using the contour integral
            \begin{equation}
               \pi_\beta^\pm = -\frac{1}{2i\pi}\oint_{\Gamma_{2r}}(-\cL_\beta^\pm-z)^{-1}\,\d z,
            \end{equation}
            so that, for all~$u\in H_{0,\beta}^1(\Omega_\beta^\pm)\cap H_\beta^2(\Omega_\beta^\pm)$, the second resolvent identity gives
            \begin{align*}
                (1-\pi^\pm_{\beta})u &= \frac{1}{2i\pi}\oint_{\Gamma_{2r}} \left[z^{-1}+(-\cL_\beta^\pm-z)^{-1}\right]u\,\d z\\
                &= \left(\frac{-1}{2i\pi}\oint_{\Gamma_{2r}}z^{-1}(-\cL_\beta^\pm-z)^{-1}\,\d z\right)\cL_\beta^\pm u.
            \end{align*}
            Estimating the~$\Lmu$ norm then yields~\eqref{eq:lemma_resolvent_eqa}:
            \[\|(1-\pi^\pm_{\beta})u\|_{\Lmu(\Omega_\beta^\pm)} \leq \frac1r\|\cL_\beta u\|_{\Lmu(\Omega_\beta^\pm)}.
                \]

            For~\eqref{eq:lemma_resolvent_eqb}, we use commutativity and the projector identity~$\pi^\pm_{\beta}\cL_\beta^\pm\pi^\pm_{\beta} = \pi^\pm_{\beta}\cL_\beta^\pm$ to write:
            \begin{align*}
                \left\|\nabla \pi^\pm_{\beta}u\right\|^2_{\Lmu(\Omega_\beta^\pm)} &= -\beta\left\langle \pi^\pm_{\beta}u , \cL_\beta \pi^\pm_{\beta}u\right\rangle_{\Lmu(\Omega_\beta^\pm)}\\
                &= -\beta\left\langle \pi^\pm_{\beta}u,\cL_\beta u\right\rangle_{\Lmu(\Omega_\beta^\pm)}\\
                &=-\beta\left\langle u,\cL_\beta u\right\rangle_{\Lmu(\Omega_\beta^\pm)} - \beta\left\langle (\pi^\pm_{\beta}-1)u,\cL_\beta u\right\rangle_{\Lmu(\Omega_\beta^\pm)}\\
                &= \left\|\nabla u\right\|_{\Lmu(\Omega_\beta^\pm)}^2 + \O(\beta\|\cL_\beta u\|_{\Lmu(\Omega_\beta^\pm)}^2),
            \end{align*}
            where we used a Cauchy--Schwarz inequality and~\eqref{eq:lemma_resolvent_eqa} to obtain the last equality.
        \end{proof}
        We are now in a position to derive the modified Eyring--Kramers formula~\eqref{eq:eyring_kramers}.
        \begin{proof}[Proof of Theorem~\ref{thm:eyring_kramers}]
            Recall that~$\psi_\beta^\pm$ denotes the quasimode defined in Section~\ref{subsec:ek_quasimodes}.
            We write, using~\eqref{eq:lemma_resolvent_eqa} and~\eqref{eq:lemma_resolvent_eqb}:
            \begin{align*}
                \lambda_{1,\beta}^\pm &= \frac1\beta\frac{\left\|\nabla \pi^\pm_\beta\psi_\beta^\pm\right\|^2_{\Lmu(\Omega_\beta^\pm)} }{\|\pi^\pm_\beta\psi_\beta^\pm\|^2_{\Lmu(\Omega_\beta^\pm)}}\\
                &=\frac1\beta\frac{\|\nabla \psi_\beta^\pm\|^2_{\Lmu(\Omega_\beta^\pm)}+\O(\beta\|\cL_\beta \psi_\beta^\pm\|^2_{\Lmu(\Omega_\beta^\pm)})}{\|\psi_\beta^\pm-(1-\pi_\beta^\pm)\psi_\beta^\pm\|^2_{\Lmu(\Omega_\beta^\pm)}}\\
                &=\frac1\beta\frac{\|\nabla \psi_\beta^\pm\|^2_{\Lmu(\Omega_\beta^\pm)}+\O(\beta\|\cL_\beta \psi_\beta^\pm\|^2_{\Lmu(\Omega_\beta^\pm)})}{1-\|(1-\pi_\beta^\pm)\psi_\beta^\pm\|^2_{\Lmu(\Omega_\beta^\pm)}}\\
                &=\frac1\beta\frac{\|\nabla \psi_\beta^\pm\|^2_{\Lmu(\Omega_\beta^\pm)}+\O(\beta\|\cL_\beta \psi_\beta^\pm\|^2_{\Lmu(\Omega_\beta^\pm)})}{1+\O(\|\cL_\beta \psi_\beta^\pm\|^2_{\Lmu(\Omega_\beta^\pm)})}\\
                &=\frac1\beta\|\nabla \psi_\beta^\pm\|^2_{\Lmu(\Omega_\beta^\pm)}\left(1+\O(\|\cL_\beta \psi_\beta^\pm\|^2_{\Lmu(\Omega_\beta^\pm)})\right) + \O(\|\cL_\beta \psi_\beta^\pm\|^2_{\Lmu(\Omega_\beta^\pm)}),\\
            \end{align*}
            where we used~\eqref{eq:lemma_resolvent_eqa} in the penultimate line, and the fact that~${\|\cL_\beta \psi_\beta^\pm\|_{\Lmu(\Omega_\beta^\pm)}=\smallo(1)}$ to conclude.
            Now, using the estimates~\eqref{eq:laplace_high_energy},~\eqref{eq:laplace_local_estimate} of Proposition~\ref{prop:semiclassical_estimates} yields
            \[\begin{aligned}\|\nabla \psi_\beta^\pm\|^2_{\Lmu(\Omega_\beta^\pm)} &= \sum_{i\in I_{\min}}\|\nabla \psi_\beta^\pm\|^2_{\Lmu(\Dseti{\pm})}+ \O\left(\e^{-\beta\left(\Vstar - V(z_0) +\frac{C_\eta}3 \largeRadius(\beta)^2\right)}\right)\\
                &=\e^{-\beta(\Vstar-V(z_0))} \left[\sum_{i\in I_{\min}} \frac{\beta|\hessEigval{i}{1}|}{2\pi\Phi\left(|\hessEigval{i}{1}|^{\frac12}\epsLimit{i}\right)}\sqrt{\frac{\det \nabla^2 V(z_0)}{\left|\det \nabla ^2 V(z_i)\right|}} \left[1 +\O\left(r_i(\beta)\right)\right] + \O\left(\e^{-\beta\frac{C_\eta}3 \largeRadius(\beta)^2}\right)\right]\\
            &=\e^{-\beta(\Vstar-V(z_0))} \left[\sum_{i\in I_{\min}} \frac{\beta|\hessEigval{i}{1}|}{2\pi\Phi\left(|\hessEigval{i}{1}|^{\frac12}\epsLimit{i}\right)}\sqrt{\frac{\det \nabla^2 V(z_0)}{\left|\det \nabla ^2 V(z_i)\right|}} \left[1 +\O\left(r_i(\beta)\right)\right]\right],\end{aligned}\]
            using the fact that~$\e^{-\beta\frac{C_\eta}3 \largeRadius(\beta)^2}=\O\left(\beta r_i(\beta)\right)$. The estimates~\eqref{eq:laplace_generator_high_energy} and~\eqref{eq:laplace_generator_local_estimate} give
            \[\begin{aligned}
                \|\cL_\beta \psi_\beta^\pm\|^2_{\Lmu(\Omega_\beta^\pm)} &= \sum_{i\in I_{\min}} \O(\beta^{-2}\largeRadius(\beta)^{-2})\|\nabla \psi_\beta^\pm\|^2_{\Lmu(\Dseti{\pm})} + \O\left(\e^{-\beta\left(\Vstar - V(z_0) +\frac{C_\eta}3 \largeRadius(\beta)^2\right)}\right)\\
                &= \O(\beta^{-2}\largeRadius(\beta)^{-2})\left(\|\nabla \psi_\beta^\pm\|^2_{\Lmu(\Omega_\beta^\pm)} - \|\nabla \psi_\beta^\pm\|^2_{\Lmu(\Omega_\beta^\pm\setminus \bigcup_{i\in I_{\min}}\Dseti{\pm})}\right) + \O\left(\e^{-\beta\left(\Vstar - V(z_0) +\frac{C_\eta}3 \largeRadius(\beta)^2\right)}\right)\\
                &= \O(\beta^{-2}\largeRadius(\beta)^{-2})\|\nabla \psi_\beta^\pm\|^2_{\Lmu(\Omega_\beta^\pm)} + \O\left(\e^{-\beta\left(\Vstar - V(z_0) +\frac{C_\eta}3 \largeRadius(\beta)^2\right)}\right)\\
                &= \O(\beta^{-2}\largeRadius(\beta)^{-2})\|\nabla \psi_\beta^\pm\|^2_{\Lmu(\Omega_\beta^\pm)}.
            \end{aligned}\]
            In the third line, we used~\eqref{hyp:scaling_deltai} to write~$\beta^{-2}\largeRadius(\beta)^{-2}\e^{-\beta\left(\Vstar - V(z_0) +\frac{C_\eta}3 \largeRadius(\beta)^2\right)} = \O\left(\e^{-\beta\left(\Vstar - V(z_0) +\frac{C_\eta}3 \largeRadius(\beta)^2\right)}\right)$.
            In the last line, we use the previous estimate~$\|\nabla \psi_\beta^\pm\|^2_{\Lmu(\Omega_\beta^\pm)} = \O(\beta\e^{-\beta(\Vstar-V(z_0))})$ to absorb the exponential error term in the prefactor~$\O(\beta^{-2}\largeRadius(\beta)^{-2})$. Combining these two estimates, we obtain
            \[\begin{aligned}\lambda^\pm_{1,\beta} &= \frac1\beta\|\nabla \psi_\beta^\pm\|^2_{\Lmu(\Omega_\beta^\pm)}\left(1+\O(\|\cL_\beta \psi_\beta^\pm\|^2_{\Lmu(\Omega_\beta^\pm)})+\O\left(\beta^{-1}\largeRadius(\beta)^{-2}\right)\right)\\
             &= \e^{-\beta(\Vstar-V(z_0))} \left[\sum_{i\in I_{\min}} \frac{|\hessEigval{i}{1}|}{2\pi\Phi\left(|\hessEigval{i}{1}|^{\frac12}\epsLimit{i}\right)}\sqrt{\frac{\det \nabla^2 V(z_0)}{\left|\det \nabla ^2 V(z_i)\right|}} \left[1 +\O\left(r_i(\beta)\right)\right]\right]\left(1+ \O(\beta^{-1}\largeRadius(\beta)^{-2})\right)\\
                &=\e^{-\beta(\Vstar-V(z_0))} \left[\sum_{i\in I_{\min}} \frac{|\hessEigval{i}{1}|}{2\pi\Phi\left(|\hessEigval{i}{1}|^{\frac12}\epsLimit{i}\right)}\sqrt{\frac{\det \nabla^2 V(z_0)}{\left|\det \nabla ^2 V(z_i)\right|}}\left[1+\O\left(r_i(\beta)+\beta^{-1}\largeRadius(\beta)^{-2}\right)\right]\right].
            \end{aligned}\]
        We note that, according to~\eqref{eq:laplace_error_terms},
        \[r_i(\beta)+\beta^{-1}\largeRadius(\beta)^{-2}=\left\{\begin{aligned}\beta^{-1}\largeRadius(\beta)^{-2},\quad&\epsLimit{i}=+\infty,\\
        \sqrt\beta\smallRadius(\beta)+\beta^{-1}\largeRadius(\beta)^{-2}+\beta^{-\frac12},\quad&\epsLimit{i}<+\infty.\end{aligned}\right.\]
        This concludes the proof of Theorem~\ref{thm:eyring_kramers} upon applying the comparison principle for Dirichlet eigenvalues (see Proposition~\ref{prop:comparison_principle}).
        \end{proof}

    \paragraph{Acknowledgments.}
    The authors thank Dorian Le Peutrec, Laurent Michel, Boris Nectoux, Danny Perez, Mohamad Rachid and Julien Reygner for helpful discussions and insightful comments, as well as pointing out some typos in preliminary versions of this work.
    This work received funding from the European Research Council (ERC) under the
    European Union's Horizon 2020 research and innovation programme (project
    EMC2, grant agreement No 810367), and from the Agence Nationale de la
    Recherche, under grants ANR-19-CE40-0010-01 (QuAMProcs) and
    ANR-21-CE40-0006 (SINEQ).
    \bibliographystyle{elsarticle-num}
    \bibliography{bibliography.bib}
    
\end{document}